\documentclass[11pt]{amsart}

\usepackage{hyperref}
\usepackage{epsfig,amsmath,amsfonts,latexsym}
\usepackage[abs]{overpic}
\usepackage{color}
 \usepackage{palatino}
 \usepackage{hyperref}
 \usepackage{blindtext}
\usepackage{enumitem}

\newtheorem{step}{Step}
\newtheorem{question}{Question}
\newtheorem{theorem}{Theorem}
\newtheorem{proposition}[theorem]{Proposition}
\newtheorem{lemma}[theorem]{Lemma}
\newtheorem{example}[theorem]{Example}

 \newtheorem*{mques}{Motivating Link Questions}

\newtheorem{corollary}[theorem]{Corollary}

\theoremstyle{definition}
\newtheorem{definition}[theorem]{Definition}

\theoremstyle{remark}
\newtheorem{remark}[theorem]{Remark}

\def\dfn#1{{\em #1}}
\def\R{\mathbb{R}}
\def\Z{\mathbb{Z}}
\def\C{\mathbb{C}}

\def\tb{Thurston-Bennequin }
\def\leg{\Lambda}
\def\maxtb{\overline{tb}}

\numberwithin{theorem}{section}
\theoremstyle{plain}

\begin{document}

\title{Legendrian Torus and Cable Links}

\author{Jennifer Dalton}
\address{New England Innovation Academy \\ Marlborough \\ Massachusettes}

\author{John B. Etnyre}

\author{Lisa Traynor}

\address{School of Mathematics \\ Georgia Institute
of Technology \\  Atlanta  \\ Georgia \\
}

\email{etnyre@math.gatech.edu}

\address{Department of Mathematics \\ Bryn Mawr College \\ Bryn Mawr \\Pennsylvania}

\email{ltraynor@brynmawr.edu}


\date{\today}

\begin{abstract} 
We give a classification of Legendrian torus links. Along the way, we give the first classification of infinite families of Legendrian links where some smooth symmetries of the link cannot be realized by Legendrian isotopies. We also give the first family of links that are non-destabilizable but do not have maximal \tb invariant and observe a curious distribution of Legendrian torus knots that can be realized as the components of a Legendrian torus link.  This classification of Legendrian torus links leads to a classification of transversal torus links.   

We also give a classification of Legendrian and transversal cable links of knot types that are uniformly thick and Legendrian simple. Here we see some similarities with the classification of Legendrian torus links but also some differences.  In particular, we show that there are Legendrian representatives of cable links of any uniformly thick knot type for which no symmetries of the components can be realized by a Legendrian isotopy, others where only cyclic permutations of the components can be realized, and yet others where all smooth symmetries are realizable.  \end{abstract}

\maketitle

\section{Introduction}

The study of Legendrian knots has gone hand-in-hand with the development of  contact geometry in dimension three, and there have been many classification results for Legendrian knots \cite{ChenDingLi15, EliashbergFraser09, EtnyreHonda01b, EtnyreHonda03, EtnyreHonda05, EtnyreLafountainTosun12, EtnyreNgVertesi13, EtnyreVertesi18, Ghiggini06c, Onaran18, Tosun13}. On the other hand, there have been surprisingly few concerning Legendrian links. Some key questions one might ask about such links are the following.  
\begin{mques} \label{mques} Let ${L}$ be a smooth link type with knot components ${K}_1\cup \ldots \cup {K}_n$.
\begin{enumerate}
\item {\em Realization of $n$-tuples:} \label{mques:real} Given Legendrian representatives $\Lambda_i$ of the knots ${K}_i$, is it possible to construct a Legendrian representative of ${L}$ that realizes the given Legendrian representatives $\Lambda_1, \ldots, \Lambda_n$ of the components? Some weaker sub-questions are:
\begin{enumerate}
\item What can be said about ``classical link geography", i.e.,  which Thurston-Bennequin ($tb$) invariants and rotation numbers ($r$) 
can be realized for the components of a Legendrian representative of ${L}$?
\item Are there Legendrian representatives of ${L}$ that do not destabilize as a link even though individual components do destabilize as knots?
\end{enumerate}
\item {\em Unordered classification:} \label{mques:unorder} How many ways can the smooth unordered  link $L$ be realized with Legendrian components realizing 
specified classical geography?
\item\label{q3} {\em Ordered classification:} \label{mques:order} Given a fixed representative $\Lambda$ of ${L}$, what smooth symmetries of the components of $\Lambda$ can be realized by a Legendrian isotopy?
\end{enumerate}
\end{mques}

\noindent
{\bf History:} One of the first results concerning the ordered classification ``symmetry'' questions was by the third author \cite{Traynor97} who used generating function techniques to show that there was a two component link called the ``helix link" in the 1-jet space of $S^1$ where there two components could be smoothly exchanged, but there was a Legendrian representative of the link where the two components could not be exchanged. This work was extended in  \cite{NgTraynor04, Traynor01}. The first results about Legendrian symmetries for links in $S^3$ were due to Mishachev, who in \cite{Mishachev03} used contact homology to show that the components of the $n$-copy of a Legendrian unknot could only be cyclically permuted, even though they could be smoothly permuted arbitrarily. This was followed by work of Ng \cite{Ng03} who showed that it was not possible to exchange the components of a $2$-copy of the maximal $tb$ invariant Legendrian figure-eight knot, which implies that it is not possible to do {\it any} permutations of the $n$-copy of the Legendrian figure-eight knot with max $tb$.

The first geography realization result was by Mohnke \cite{Mohnke01}, who used bounds on the Thurston-Bennequin invariant of links coming from the HOMFLY-PT polynomial to show that there were no realizations of the Borromean rings or the right handed Whitehead link where all the components had maximal Thurston-Bennequin invariant for their knot types.    

The first classification results for links were given by Ding and Geiges in \cite{DingGeiges07}. There they showed that the ``cable link", this is an unknot together with a $(p,q)$-cable of it, were Legendrian simple --- that is determined by their classical invariants (knot type, Thurston Bennequin invariant, and rotation number). They also showed an analogous result in the 1-jet space of $S^1$. With the exception of the helix link mentioned above,  these knots have no smooth symmetries so do not address Question~\eqref{q3}. In the paper \cite{DingGeiges10}, Ding and Geiges addressed the helix link in the 1-jet space of $S^1$, thus giving the first classification result where one can see topological symmetries of a link that cannot be realized by a Legendrian link in that link type. Recently Geiges and Onaran gave a classification of Legendrian links realizing the positive Hopf link \cite{GeigesOnaran20}. This work very interestingly goes beyond the classification of Legendrian Hopf links in the standard contact structure on $S^3$, but also gives the classification in all contact structures (both the tight one and all overtwisted ones) in $S^3$. 
\vskip .1in
\noindent
{\bf New results:} Below we give a complete classification of Legendrian torus links and prove several results about cable links. The main new phenomena that arise from this work are as follows.  Throughout this paper, $tb$ will refer to the \tb invariant, and $r$ will refer to the rotation number invariant.
 
\begin{enumerate}
\item The components of any Legendrian torus link   must all have a common destabilization.  This provides a strong restriction on the geography of negative torus links with 
knotted components.   
\item All max $tb$ representatives of a negative torus link with knotted components must have components with the same $tb$ and  $r$ invariants, however
 there are max $tb$ representatives of two-component negative torus links with unknotted components where the components have different $tb$ values.
When there are $3$ or more components of these negative torus links with unknotted components, there are representatives   that do not have max $tb$  
yet do not destabilize. These are the first such examples for links. Similar results hold for cable links. 
\item For max $tb$ representatives of negative torus links  with at least $3$ components, only cyclic permutations of the components    can be achieved through Legendrian isotopy. This gives an infinite family of examples where the ordered and unordered classification of Legendrian links differs.

\item Although the components of  a smooth cable link can be arbitrarily permuted via a smooth isotopy, every  uniformly thick, {non-cable} knot type, which includes the Figure-eight knot,   
admits some Legendrian cables (with Legendrian equivalent components) where no non-trivial permutations are possible, others where
only cyclic-permutations are possible, and yet others where all permutations are possible.    
\end{enumerate}

Many of the unordered torus link classification arguments parallel the torus knot classification arguments
of \cite{EtnyreHonda01b}, however the unordered classification of negative torus links  with unknotted components, $(n,-nq)$-torus links, requires new ideas. 
The most challenging portion of the torus link classification is to understand the restrictions on permutations of the components of negative torus links with maximal $tb$ invariant.  
Although this restriction was previously established for $(n,-nq)$-torus links, \cite{Mishachev03},
our analysis of pre-Lagrangian tori and annuli in terms of convex surface theory gives a 
geometric explanation for why non-cyclic permutations are impossible for all $(np, -nq)$-torus links. Moreover these convex surface arguments can also be adapted to
the setting of Legendrian cable links and allows for understanding when even cyclic permutations are not possible.

%Our analysis of the restrictions on permutations of the components of torus links gives a geometric explanation for the previously known obstructions to such permutations that were known in a few very special cases, and the restrictions for cable links was previously unknown. This explanation is in terms of pre-Lagrangian tori and annuli, see Theorem~\ref{thickenedtori} and Section~\ref{list}. We note that this is the most novel part of the paper. Much of the unordered classification follows fairly standard convex surface theory techniques. In particular, most readers of \cite{EtnyreHonda01b} could easily produce the unordered classification of torus links with knotted components though the classification with unknotted components is new despite using fairly standard convex surface theory. 
%We include sketches of the arguments here as they are short and help make the paper reasonably self-contained. However, the ordered classification requires one to interpret rigidity phenomena of pre-Lagrangian tori and annuli in terms of convex surface theory and is significantly new. In addition, finding a general criteria for cable links to have no Legendrian symmetries seems to be completely new as the only previously known example was for special cables of the figure eight knot mentioned above and there was no clear geometric idea how to generalize that example as it came down to intricacies in the algebra of its contact homology. 

\subsection{Torus Links}

Smooth torus links are links that, after a smooth isotopy, lie on an unknotted torus in $\mathbb R^3$; they form an important family of smooth links.  
Every torus link has 
components that are all unknots or all topologically equivalent torus knots.  Moreover, it is not hard to see that it is
possible to arbitrarily permute the components of any smooth torus link with a smooth isotopy.  
We will address the classification of Legendrian torus links.  First, it is useful to recall what is known about the
classification of Legendrian unknots and Legendrian torus knots.

Oriented Legendrian unknots were classified by Eliashberg and Fraser in \cite{EliashbergFraser98}; 
some alternate proofs can be found in \cite{EtnyreHonda01b, GeigesOnaran15}.
They found that unknots are simple in the sense that they are classified by their $tb$ and $r$ invariants. In particular,
there is a unique Legendrian unknot with $tb = -1$ and $r = 0$, and all other unknots are obtained
by stabilizations of this one with maximal $tb$ invariant. These results were obtained using an
intricate analysis of the characteristic foliations of Seifert disks for the knots.

Oriented Legendrian torus knots were classified by the second author and Honda in
\cite{EtnyreHonda01b} using the technique of convex surfaces. Since the $(p, q)$-torus link agrees
with the $(q, p)$-torus link and with the $(-p, -q)$-torus link, it suffices to look at
$(p, \pm q)$-torus links where $q \geq p \geq 1$.  An unknot arises when  $p = 1$; nontrivial torus knots correspond to $p > 1$ and
$\gcd(p,q) = 1$.  As is the case for unknots, Legendrian torus knots are classified by their
topological knot type, and their $tb$ and $r$ invariants. However, in \cite{EtnyreHonda01b} it is shown that the possible
range of the classical invariants is more intricate and has a different flavor depending on whether
one is looking at positive or negative torus knots. In particular:
\begin{enumerate}
\item When considering Legendrian torus knots that are topologically $(p, +q)$-torus knots, there is a unique one with max $tb$ invariant of $pq - p -q$; this max $tb$ representative will have $r = 0$, and all other representatives are stabilizations of this one with max $tb$;
\item When considering Legendrian torus knots that are topologically $(p, -q)$-torus knots, if 
$-m-1 < -q/p < -m$, $m \in \mathbb Z^{+}$, 
then there will be $2m$ representatives with 
$\max tb = -pq$;  each of these $\max tb$ represetntatives is distinguished by their $r$ invariant, which is an element of the set
\[
\{ \pm (q - p - 2pk) : k \in \Z, \quad 0 \leq k \leq m-1 \},
\]
and all other representatives are stabilizations of one with max $tb$.
\end{enumerate}
 
 It is common to represent all Legendrian representatives of a fixed topological type by a mountain range,
 which consists of  
 an infinite graph with vertices labeled by pairs of integers.  Every vertex of the mountain range represents a
unique Legendrian  knot with vertex labels giving the knot's $r$ and $tb$ invariants.
 Two vertices are connected if and only if the corresponding knots differ by a positive or
negative stabilization: a positive stabilization will lower the \tb invariant and raise the
rotation number while a negative stabilization will lower the \tb invariant and lower the rotation
number.
We can rephrase the above classification results by saying that we understand the mountain range of the unknot and all torus knots.
 For example, Figure \ref{fig-tree3,2} represents all Legendrian $(2, +3)$-torus knots. All
positive torus knots will have a mountain range  graph of this shape, however the ``peak" of the graph will occur at $tb
= pq - p -q$. In contrast, nontrivial negative torus knots will always have ``multiply peaked" mountain ranges.
Figure~\ref{fig-tree-7,3} gives the mountain range representing all Legendrian $(3, -7)$-torus knots.
    
\begin{figure}[ht]
\tiny
\begin{overpic} 
{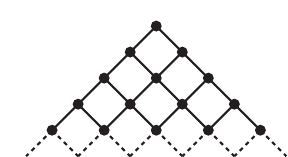}
\put(-5,72){$r=$}
\put(42,72){$-2$}
\put(54,72){$-1$}
\put(73,72){$0$}
\put(87,72){$1$}
\put(98,72){$2$}
\put(-5,63){$tb=1$}
\put(12,52){$0$}
\put(7,41){$-1$}
\end{overpic}
\caption{The mountain range representing all possible Legendrian $(2, +3)$-torus knots.}
\label{fig-tree3,2}
\end{figure}    
    
\begin{figure}[ht]
\tiny
\begin{overpic} 
{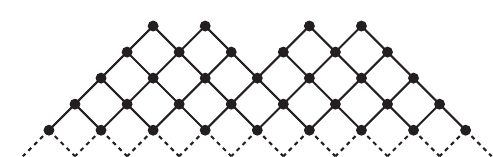}
\put(-10,72){$r=$}
\put(52,72){$-5 -4 -3 \!-2 \!-1 \;\;\: 0 \quad 1\quad\,\, 2\quad \; 3\quad  \,\, 4 \quad \,5$}
\put(-10,63){$tb=-21$}
\put(7,52){$-22$}
\put(7,41){$-23$}
\end{overpic}
  \caption{The mountain range representing all possible Legendrian $(3, -7)$-torus knots.}
  \label{fig-tree-7,3}
\end{figure}
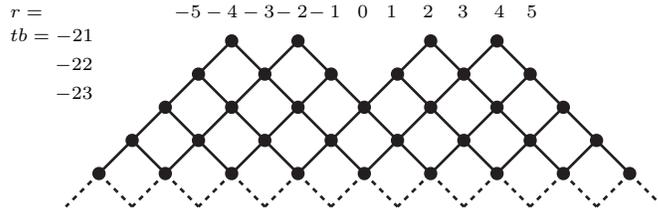

In this paper, we classify both ordered and unordered Legendrian torus links: $(np, \pm nq)$-torus links
where  $q \geq p \geq 1$, $\gcd(p,q) = 1$, and $n \geq 2$.
Since an $(np, \pm nq)$-torus link has $n$ components each of which is a $(p, \pm q)$-torus knot, a useful way
to visualize an $(np, \pm nq)$-torus link is as an $n$-tuple of vertices on the mountain range that represents the
possible $(p,\pm q)$-torus knots.   
    {We will sometimes denote an $(np, \pm nq)$-torus link as an $n(p,\pm q)$-torus link.}
 
The following theorem answers Motivating Link Questions~(\ref{mques:real}) and (\ref{mques:unorder}) in the case of torus links. 
\begin{theorem} \emph{(Realization and Unordered Legendrian Torus Link Classification)}\label{unorderedclass}
Consider two oriented Legendrian torus links $L$ and $L'$ that are topologically
equivalent to the $(np, \pm nq)$-torus link, where $q \geq p \geq 1$, $\gcd(p,q) = 1$, and $n \geq 2$.
 If labels can be given to the components $L = \amalg_{i=1}^{n} \Lambda_{i}'$, $L' = \amalg_{i=1}^{n} \Lambda_{i}'$ so that $tb(\Lambda_i) = tb(\Lambda'_i)$ and 
$r(\Lambda_i) = r(\Lambda'_i)$, $i = 1, \dots, n$, then there exists a contact isotopy
taking $L$ to $L'$ (but not necessarily $\Lambda_i$ to $\Lambda'_i$). Moreover, the precise
range of the classical invariants is given as follows:
\begin{enumerate}
\item For any  positive torus link, namely any $(np, +nq)$-torus
  link, there exists a unique Legendrian representative  
  with maximal $tb$ invariant; for such a link all components will have $tb = pq-p-q$. For 
  any positive torus link with non-maximal $tb$ invariant, the components can be destabilized to obtain the one with maximal $tb$
  invariant.
   \item For any negative torus link with knotted components,  namely any $(np, -nq)$-torus
  link with $p > 1$, if $-m-1 < \frac{-q}{p} < -m$, $m \in \mathbb Z^{+}$, 
  there are $2m$ Legendrian realizations of the $(np, -nq)$-torus link with maximal $tb$
  invariant; all components of a maximal version will have $tb = -pq$ and the same rotation number.
  For any negative torus link with non-maximal $tb$ invariant, the components can be destabilized to obtain the one with maximal $tb$
  invariant.
\item For negative torus links with unknotted components, namely $(n, -nq)$-torus links, there is a set of $ \frac{q(q+1)}{2}$ nondestabilizable Legendrian realizations
  consisting of the $n$-copy of an unknot with $tb = -q$ and, for 
  $0 < t < q$, the $t$ Legendrian twist of the $n$-copy
  of an unknot with $tb = -q + t$. The $n$-copy will have maximal $tb$ invariant while
  the Legendrian twist versions will have maximal $tb$ invariant if and only if $n = 2$. Any other
  Legendrian $(n, -nq)$-torus link will destabilize to one in this nondestabilizable set.
\end{enumerate}
\end{theorem}
\noindent
This theorem follows from Theorem~\ref{thm:p-unorder-class} for the Case~(1), from Theorem~\ref{oclass} and Lemma~\ref{negmaxtb} for Case~(2), and Theorem~\ref{uoclass} for Case~(3). 
The $n$-copy of a Legendrian link will be defined in Section~\ref{generalLeg}, and the $t$ Legendrian twist of the $n$ copy will be defined in Definition~\ref{t-twist}.
 
We will  represent the components of Legendrian $(np, \pm nq)$-torus link by circling $n$ vertices on the mountain range that represents
all possible  Legendrian $(p, \pm q)$-torus knots.  
\begin{figure}[ht]
\tiny
\begin{overpic} 
{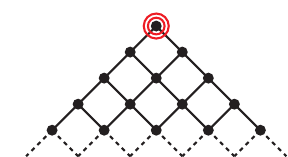}
\put(-5,74){$r=$}
\put(42,74){$-2$}
\put(54,74){$-1$}
\put(73,74){$0$}
\put(88,74){$1$}
\put(100,74){$2$}
\put(-5,63){$tb=1$}
\put(12,52){$0$}
\put(7,41){$-1$}
\end{overpic}
  \caption{The two circles represent the unique unordered Legendrian $2(2,+3)$-torus link with maximal $tb$ invariant.  
    In fact, any two vertices on this mountain range for the $(2,+3)$-torus knot represent a  Legendrian $2(2,+3)$-torus link.}                    
  \label{fig-tree6,4}
\end{figure}
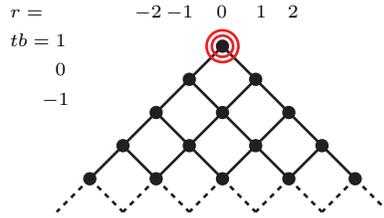   
 
By Theorem~\ref{unorderedclass}, an $n$-tuple of vertices on the mountain range of the $(p, \pm q)$-torus
knot 
can represent at most one unordered $(np, \pm nq)$-torus link.
Figure \ref{fig-tree6,4} represents the unique Legendrian $(4, +6)$-torus link with maximal $tb$ invariant. In
fact, by the above theorem, there is a one-to-one correspondence between $(4,+6)= 2(2,+3)$-torus links and
pairs of vertices on the $(2, +3)$-mountain range. However, for negative torus links, the above theorem
indicates that there are more restrictions on the $n$-tuples of vertices that correspond to torus
links. For example, Figure~\ref{fig-nonexist-14,6} shows some pairs of vertices that do {\it not}
\begin{figure}[ht]
\tiny
\begin{overpic} 
{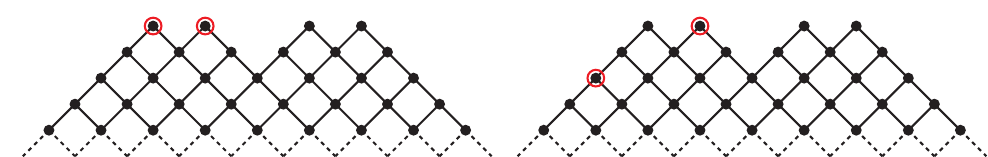}
\end{overpic}
  \caption{It is not possible to construct Legendrian $2(3,-7)$-torus links with components as
    indicated in these mountain ranges.}                      
  \label{fig-nonexist-14,6}
\end{figure}   
correspond to a $2(3,-7)$-torus link. When the components are nontrivial knots ($p > 1$),   Legendrian $(np,
-nq)$-torus links with maximal $tb$ invariant correspond to a ``peak" of the $(p, -q)$-mountain range chosen $n$ times;
if the components are unknots, the nondestabilizable  Legendrian representatives of $(n,-nq)$-torus links
correspond to a vertex with $tb=-q$
chosen $n$ times or $n$ vertices all with the same rotation number consisting of one vertex with $tb=(-q + t)$ 
and $(n-1)$ vertices with $tb=(-q -t)$. This ``split level" representative
will have maximal $tb$ invariant only when $n = 2$. See Figures~\ref{fig-tree-6,2max} and~\ref{fig-2maxtb}.
\begin{figure}[ht]
\begin{overpic} 
{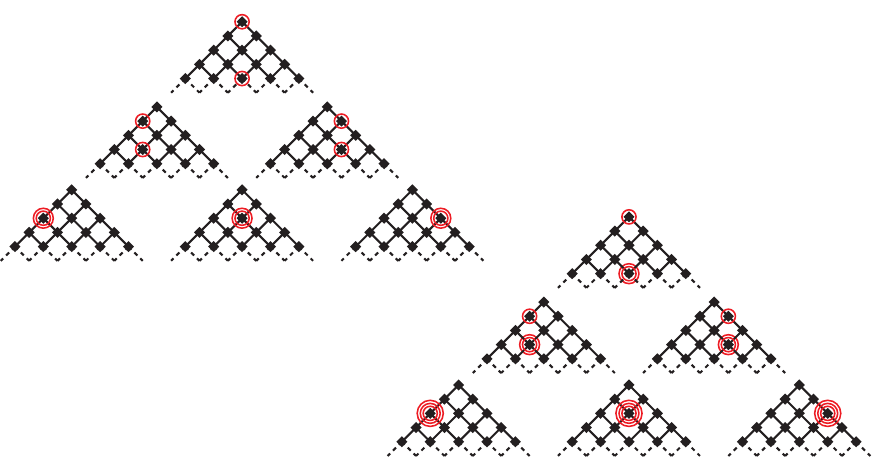}
 \put(72,  70){$2(1,-3)$-torus links}
\put(265, 133){$3(1,-3)$-torus links}
\end{overpic}
  \caption{The $6= \frac{3(4)}{2}$ Legendrian $2(1,-3)$-torus links with max $tb$ invariant, and the $6$ non-destabilizable $3(1,-3)$-torus knots, three of which have max $tb$ invariant.  }                   
  \label{fig-tree-6,2max}        
\end{figure}

\begin{figure}[ht]
\small
\begin{overpic} 
{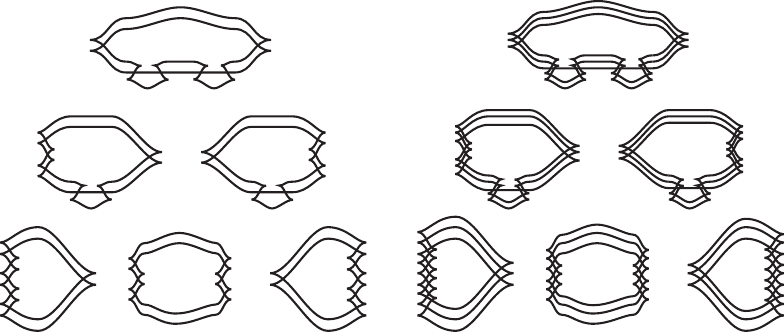}
\end{overpic}
        \caption{Max $tb$ invariant Legendrian $2(1, -3)$-torus links on the left and the non-destabilizable $3(1,-3)$-torus links on the right. All components are oriented {clockwise}.}
        \label{fig-2maxtb}
\end{figure}  

Now we address the symmetry question in the Motivating Link Question~(\ref{mques:order}). When considering the ordered classification of torus links,
notice that the components inherit a natural cyclic ordering from the underlying torus. The next
theorem tells us that for positive torus links, the unordered and ordered classifications agree, but
in the ordered classification of negative torus links, the cyclic order of particular components
must be preserved. This result extends the findings of Mishachev \cite{Mishachev03} for $(n,-nq)$-torus links.

\begin{theorem}  \emph{(Ordered Legendrian Torus Link Classification)} \label{orderedclass}
Consider two ordered, oriented Legendrian torus links $L = (\Lambda_1, \dots, \Lambda_n)$ and $L' =
(\Lambda'_1, \dots, \Lambda'_n)$ that are topologically equivalent to the $(np, \pm nq)$-torus
link,  where  $q \geq p \geq 1$, $\gcd(p,q) = 1$, and $n \geq 2$. Suppose $tb(\Lambda_i) = tb(\Lambda'_i)$ and $r(\Lambda_i) = r(\Lambda'_i)$, $i = 1, \dots, n$.
\begin{enumerate}
\item If $L$ and $L'$ are positive torus links, then there exists a contact isotopy taking
  $L$ to $L'$ such that $\Lambda_i$ is mapped to $\Lambda'_i$.
\item If $L$ and $L'$ are negative torus links, then there exists a contact isotopy
  taking $L$ to $L'$ such that $\Lambda_i$ is mapped to $\Lambda'_i$ if and only if the cyclic
  ordering of the components with $tb = -pq$ is preserved.
\end{enumerate}
\end{theorem}  

See Figures~\ref{fig-perm-21,9} and \ref{fig-perm-8,4} for illustrations of this theorem.

\begin{figure}[ht]
\small
\begin{overpic} 
{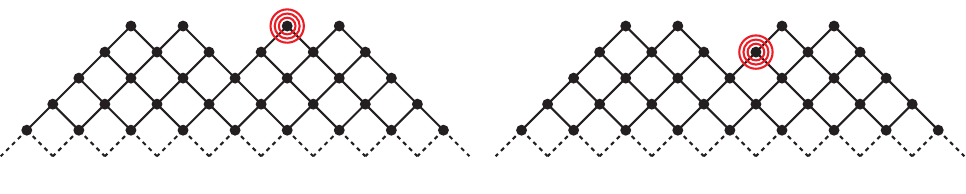}
\put(107,4){(a)}
\put(345,4){(b)}
\end{overpic}
  \caption{Two examples of unordered Legendrian $3(3,-7)$-torus links.  When considered as 
    ordered links, in (a) noncyclic permutations produce nonequivalent ordered oriented Legendrian  
    links, while in (b) all  permutations produce equivalent ordered oriented  Legendrian links.}                 
  \label{fig-perm-21,9}    
\end{figure}  

\begin{figure}[ht]
\small
\begin{overpic} 
{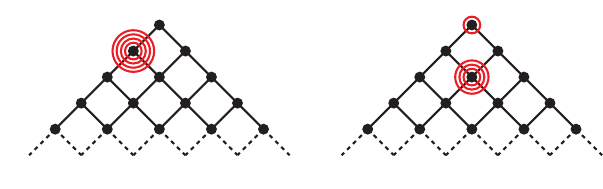}
\put(75,4){(a)}
\put(228,4){(b)}
\put(-7,87){$r=$}
\put(40,87){$-2 \!-1 \quad\!\! 0 \;\;\; 1\;\;\; 2$}
\put(-7,78){$tb=-1$}
\put(14,65){$-2$}
\put(14,52){$-3$}
\end{overpic}
  \caption{Two examples of unordered Legendrian $4(1,-2)$-torus links.  When considered as 
    ordered links, in (a) noncyclic permutations produce nonequivalent ordered oriented Legendrian  
    links, while in (b) all  permutations (preserving $tb$ and $r$) produce equivalent ordered 
    oriented Legendrian links.}                    
  \label{fig-perm-8,4}
\end{figure}
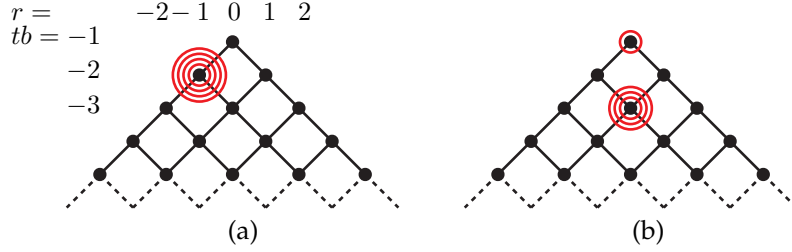

From this classification of Legendrian torus links, we can deduce the classification of transversal torus links.
Recall that a  \dfn{ smooth knot type is transversely simple} if  any two transversal representatives $\mathcal T$ and $\mathcal T'$ with the same
self-linking number are transversally isotopic.  
A \dfn{ smooth knot type is negatively stable simple} if for any two Legendrian representatives $\Lambda$ and $\Lambda'$
satisfying $tb(\Lambda) - r(\Lambda) = tb(\Lambda') - r(\Lambda')$, there exists $n_1$ and $n_2$ such that
$(S_-)^{n_1}(\Lambda)$ is Legendrian isotopic to $(S_-)^{n_2}(\Lambda')$, where $S_\pm(\Lambda)$ denote the $\pm$-stabilization of $\Lambda$.  
  In fact,  
it is shown in \cite{EpsteinFuchsMeyer01, EtnyreHonda01b} a knot type is negatively stable simple if and only if it is
transversally simple.
Analogously, a \dfn{ smooth link type is transversally simple}  if for any two transversal representatives
$(\mathcal T_1, \dots,  \mathcal T_m)$ and $(\mathcal T_1', \dots, \mathcal T_m')$ with corresponding self-linking numbers, 
$$ \ell(\mathcal T_i) = \ell(\mathcal T_i'), \quad \text{for all }i,$$
there exists a transversal isotopy from $(\mathcal T_1, \dots,  \mathcal T_m)$ to $(\mathcal T_1', \dots, \mathcal T_m')$. 
We will say that a \dfn{ smooth link type is negatively stable simple} if for any two Legendrian representatives
$(\Lambda_1, \dots, \Lambda_m)$ and $(\Lambda_1', \dots, \Lambda_m')$ with
$$ tb(\Lambda_i) - r(\Lambda_i) = tb(\Lambda_i') - r(\Lambda_i'), \quad \text{for all }i,$$
there exists $n_1, \dots, n_m$ and $n_1', \dots, n_m'$ such that
$\left( (S_-)^{n_1}(\Lambda_1), \dots, (S_-)^{n_m}(\Lambda_m) \right)$ is Legendrian isotopic to 
$\left( (S_-)^{n_1'}(\Lambda_1'), \dots, (S_-)^{n_m'}(\Lambda_m') \right)$.  Observe that Theorems~\ref{unorderedclass} and \ref{orderedclass} show that
all torus links are negatively stable simple.
 The proof in
\cite{EtnyreHonda01b} can easily be extended to show that if a link type is negatively stable simple, then it is
transversally simple.

\begin{theorem}[Ordered Transversal Torus Link Classification] For $q \geq p \geq 1$, $\gcd(p,q) = 1$, and $n \geq 2$,
 there is a unique transversal $(np, \pm nq)$ torus link that cannot be destabilized;
all components of this nondestabilizable link have self-linking number $\pm q(p-1) - p$. All other
transversal  representatives of this torus link destabilize to this one.   In particular, two transversal $(np, \pm nq)$-torus links are determined by the self-linking
numbers of their components. 
 For all transversal torus links, all (self-linking number preserving) permutations of the components  
 are realizable through transversal isotopy.\hfill \qed
\end{theorem}

\subsection{Cable Links}\label{cable-links}
 Let $K$ denote an oriented knot type.  Then for $n \geq 1$, and $p, q \in \mathbb Z$ such that $p \geq 1$ and $\gcd(p,q) = 1$,  
the  \dfn{ $(np,nq)$-cable}  of $K$, denoted ${K}_{(np, nq)}$,  is the $n$-component link type obtained by taking an $(np,  nq)$-curve on the boundary of a tubular neighborhood of a representative of ${K}$.    Here by a $(p, q)$-curve we mean one that runs $p$ times longitudinally, with the  longitude given by a Seifert surface for ${K}$, and $q$ times meridionally. 
In our definition of a cable, we allow $K$ to be an unknot \footnote{This convention differs from the frequent convention in the definition of the more general satellite knots}, so in this paper we consider torus links to be cable links and the unknot and torus knots to be cable knots. 
Observe that if $q/p \in \mathbb Z$,  then $p = 1$, and thus ${K}_{(p, q)}$ is topologically equivalent to $K$. We allow all values of $q$ since, in contrast to the case of torus links,  ${K}_{(np, nq)}$ need not be topologically equivalent to ${K}_{(nq, np)}$.  We restrict to $p \geq 1$ since ${K}_{(-np, -nq)}$ will be topologically equivalent to $-K_{(np, nq)}$.

We will be able to understand the
Legendrian classification of ${K}_{(np, nq)}$ when $K$ is a ``uniformly thick'' and a Legendrian simple knot type. 
A knot type ${K}$ is called   \dfn{ uniformly thick}, {\em cf.\ }\cite{EtnyreHonda05}, if given any solid torus $S$ whose core is in the knot type ${K}$ then there is a solid torus $S'$ containing $S$ and isotopic to $S$, such that $S'$ is a standard neighborhood of a max-$tb$  Legendrian representative of ${K}$; see  Definition~\ref{std-nbhd} for an explanation of a standard neighborhood.    There are many uniformly thick knot types: the figure-eight knot \cite{MinPre},  all negative torus knots \cite{EtnyreHonda05}, and most twist knots are uniformly thick \cite{MinPre}. 
 However not all knots are: for example, the unknot, positive torus knots \cite{EtnyreHonda05, EtnyreLafountainTosun12}, and many Lagrangian slice knots \cite{McCulloughPre18} are not uniformly thick.

Towards understanding the Legendrian classification, it will be important to understand the non-destabilizable representatives of ${K}_{(np,nq)}$.  In Definition~\ref{cable-defn}, we define {standard Legendrian $(np, nq)$-cables} in terms of standard neighborhoods of any knot type $K$;  here we give their definition in terms of front projections.  
\begin{definition} \label{cable-front-algorithm} Given a knot type $K$, let $\overline{tb}(K)$ denote the maximal \tb invariant that can be realized by any Legendrian representative of $K$.  Then 
 for $n \geq 1$ and $p, q \in \mathbb Z$ such that $p \geq 1$ and $\gcd(p,q) = 1$, 
the \dfn{standard Legendrian $(np, nq)$-cable} of $K$ is defined according to the slope $q/p$ as follows; see Figure~\ref{fig8cable} for an illustration.
 
\begin{enumerate}
\item For {\it $\overline{tb}(K)$-slope cables}, $\frac{q}p =  \overline{tb}(K)$: 
\begin{itemize}
\item Fix $\leg$ such that $tb(\leg) = \overline{tb}(K)$.
\item Then $\Lambda_{(np,  nq)} = \Lambda_{(n1, nq)} $ is the $n$-copy of $\leg$, so its
front diagram can be obtained by starting with the front of $\leg$ and making $n$-copies from slight shifts in the $z$-direction.  
\end{itemize}
\item  To construct a {\it greater-slope cable},  $\frac{q}{p} > \overline{tb}(K)$, start by writing $\frac{q}{p} = \overline{tb}(K) + \frac{s}{p}$, for $s \in \mathbb Z$ positive.
\begin{itemize}
\item Fix $\leg$ such that $tb(\leg) = \overline{tb}(K)$.
\item Make the  $np$-copy of $\leg$.
\item Form $\leg_{(np,nq)}$  by  replacing a trivial $np$-stranded tangle of the $np$-copy with $\frac{ns}{np} = \frac{s}{p}$ 
of a full positive twist, which corresponds to repeating the fundamental positive crossing strand tangle, as shown at the top of Figure~\ref{changes}, $ns$ times. 
\end{itemize}
\begin{figure}[ht]
\small
\begin{overpic} 
{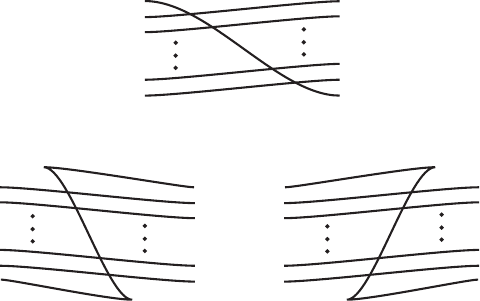}
\end{overpic}
        \caption{The upper diagram is a $1/p$ twist when there are $p$ strands. On the lower left is an $S$-tangle and on the lower right is a $Z$-tangle.} 
        \label{changes}
\end{figure}  
\item For {\it lesser-slope cables}, $\frac{q}p <  \overline{tb}(K)$,   start by writing $\frac{q}{p} =   \lceil \frac{q}{p} \rceil  -  \frac{s}{p}$, for  $ 0 \leq s < p$.  
\begin{itemize}
\item Let  $\leg$ be a Legendrian representative of $K$ with $tb(\leg) =  \lceil \frac{q}{p} \rceil \leq \overline{tb}(K)$.
\item Make the $p$-copy of $\leg$.  
\item Form $\leg^\pm_{(p,q)}$ by replacing a trivial $p$-stranded tangle of the $p$-copy of $\leg$ with either  $s$ fundamental $p$-stranded $Z$-tangles or $s$ fundamental $p$-stranded $S$-tangles as shown in Figure~\ref{changes}.
Observe that when $p =1$, $s = 0$, and thus
$\leg^+ = \leg^- = \leg$.
 \item Take the $n$-copy of $\leg^\pm_{(p,q)}$ to form $\leg^\pm_{(np,nq)}$.
\end{itemize}
\end{enumerate}
\end{definition}
\begin{remark} When $K$ is the unknot, and $q > p > 1$,  the standard Legendrian $(np, nq)$-cables of $K$ are the Legendrian torus links with maximal \tb invariant. 
\end{remark}

Figure~\ref{fig8cable} shows the front diagrams of some Legendrian cables of the figure-eight knot in all three slope types.
\begin{figure}[ht]
\small
\begin{overpic} 
{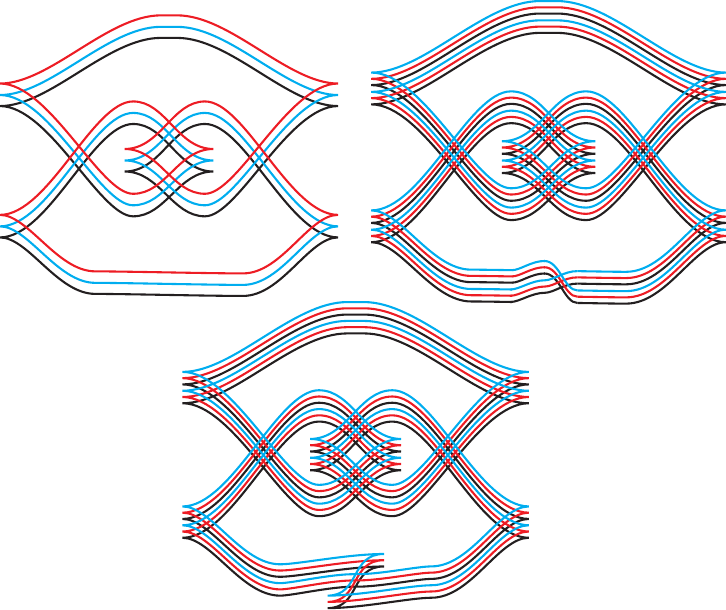}
\end{overpic}
        \caption{Some cables of the figure-eight knot $K$ with $\overline{tb}(K)= -3$.  The upper left is the standard ($\overline{tb}(K)$-slope) $3(1,-3)$-cable of the figure-eight knot $K$. The upper right is the standard (greater-slope) $3(2,-5)$-cable of the figure-eight knot.  
         The lower diagram is a standard (lesser-slope) $3(2,-7)$-cable of the figure-eight knot.  
 } 
         
        \label{fig8cable}
\end{figure}  
For a uniformly thick knot type ${K}$, all of the standard Legendrian $(np,nq)$-cables will have max $tb$, see Lemma~\ref{std-cables-max-tb}.
When $\frac{q}{p} = \frac{q}{1} < \maxtb(K)$, observe that each standard  lesser slope $(n,nq)$-cable of $K$ is the  $n$-copy of a Legendrian representatives $\leg$ of $K$ with $tb(\leg) = q$.  There will be some additional non-destabilizable versions coming from  ``twisted $n$-copies'' that
parallel the torus links with unknotted components pictured in Figure~\ref{fig-2maxtb}. As was the situation for
the torus link with unknotted components, these integral and lesser-sloped twist versions will have 
max $tb$ if and only if $n = 2$.   
 {Since each component of a Legendrian integral, lesser-slope cable of $K$ is topologically equivalent to $K$, the components can be represented as an $n$-tuple in the mountain range of $K$.  For example, when $K$ is the figure-eight knot, and $\frac{q}{p} = -5 < -3 = \overline{tb}(K)$,  in parallel to Figure~\ref{fig-tree-6,2max}, there will be six maximal \tb Legendrian representatives of  the $(2, -10) = 2(1,-5)$ cable of the figure-eight knot, and six  non-destabilizable Legendrian representatives of  the $(3,-15) = 3(1,-5)$ cable of the figure-eight knot, three of which will have maximal \tb invariant.
 
   In Propositions~\ref{destab-non-integer} and \ref{destab-integer}, we show that  if $K$ is uniformly thick, then every
Legendrian representative of $K_{(np,nq)}$ will destabilize to either a standard Legendrian $(np,nq)$-cable or, in the case that $\frac{q}{p} = \frac{q}1 < \maxtb(K)$, to a twisted $n$-copy.
This allows us to classify unordered  Legendrian cables of a uniformly thick and Legendrian simple knot type $K$; see Theorem~\ref{unorderedcable}.
In particular,  for a uniformly thick and Legendrian simple   knot  type $K$, knowing the mountain range of $K$ allows us to determine the  mountain range of all Legendrian representatives of $K_{(np,nq)}$.

We also establish the ordered classification of Legendrian cable links of uniformly thick and simple knot types. As was seen in the case of Legendrian torus links, there is interesting flexibility and rigidity in terms of permutations of the
components in the max $tb$ representatives of ${K}_{(np, nq)}$.

\begin{theorem}\label{cableperms}
Let $K$ be a uniformly thick knot type. For $n \geq 2$, $p, q \in \mathbb Z$ such that $p \geq 1$ and $\gcd(p,q) = 1$, if $L_{(np,nq)} =(\Lambda_1,\ldots, \Lambda_n)$ is a  {standard } Legendrian $(np,nq)$-cable of $K$, where the $\Lambda_i$ are ordered as they appear on the torus or annulus used in the definition of standard cables,  the following permutations of the components are possible.
\begin{description}[align=left]
\item[greater-slope cables] If $q/p>\overline{tb}(K)$, then  any permutation of the $\Lambda_i$ is possible by a Legendrian isotopy. 
\item[$\overline{tb}(K)$-slope cables] If  $q/p=\overline{tb}(K)$, and $K$ is not a cable knot or $K$ is an $(r,s)$-cable knot and $q/p\not=rs$, then no permutation of the $\Lambda_i$ can be realized by a Legendrian isotopy. 
\item[lesser-slope cables ] If $q/p<\overline{tb}(K)$, and $K$ is not a cable knot or $K$ is an $(r,s)$-cable knot and $q/p\not=rs$, then
only cyclic permutations  of the $\Lambda_i$ can be realized. 
 \end{description}
\end{theorem}

\begin{remark}
One may observe in the proof of Theorem~\ref{cableperms} that for standard cables with $q/p>\overline{tb}(K)$ does not need the hypothesis that $K$ is uniformly thick, but the other cases do need this hypothesis. 
\end{remark}

\begin{remark} \label{rem:tb-slope-non-cables} In the $\overline{tb}(K)$-slope cables of Theorem~\ref{cableperms},  the hypothesis that $K$ is not a cable is necessary in order to forbid {\it any} permutations.  For example, 
suppose $K$ is the $(2,-3)$-torus knot,  
 which is a $(2,-3)$-cable of the unknot.  Since $\overline{tb}(K) = -6$,   
the standard Legendrian $n(1, -6)$-cable of $K$ is  the $n$-copy of a max-$tb$ representative of $K$, and, by Theorem~\ref{orderedclass}, we know cyclic permutations of the components are allowed.  \end{remark}

\begin{example}  \label{fig-8-cables}  If $K$ is the uniformly thick, (non-cable) figure-eight knot, then, since $\overline{tb}(K) = -3$, Theorem~\ref{cableperms} tells us that it is not
possible to do any permutation of the $n$ components in  $L_{n(1, -3)}$, 
the standard   Legendrian $n(1,-3)$-cable of $K$;
this case recovers Ng's result from \cite{Ng03} that it is not possible to do any permutations of the $n$-copy of max-$tb$ Legendrian figure-eight knot.
However, if the slope $q/p > -3$ then it is possible to arbitrarily permute the components of $L_{n(p,q)}$, while  if $q/p < -3$, then it is only possible to cyclically permute the components of $L_{n(p,q)}$.
 \end{example}

After establishing  which permutations can be realized in the max-$tb$ standard Legendrian cables, we can give the {\it ordered} Legendrian classification of $K_{(np,nq)}$, where
$K$ is a uniformly thick and a Legendrian simple knot type.
\begin{theorem}  \emph{(Ordered Cable Link Classification)} \label{orderedcable} For $n \geq 2$, $p, q \in \mathbb Z$ such that $p \geq 1$ and $\gcd(p,q) = 1$, the $(np, nq)$-cable of a uniformly thick, Legendrian simple  knot type $K$ is Legendrian simple as an unordered link. The range of possible \tb and rotation number invariants for such a link will be given in Section~\ref{cables-non-destab}. Consider two ordered, oriented Legendrian links $L = (\Lambda_1, \dots, \Lambda_n)$ and $L' =
(\Lambda'_1, \dots, \Lambda'_n)$ that represent the knot type $K_{(p,q)}$, and
 suppose $tb(\Lambda_i) = tb(\Lambda'_i)$ and $r(\Lambda_i) = r(\Lambda'_i)$, for $i = 1, \dots, n$. Then: 
\begin{description}[align=left]
\item[Greater-Slope Cables] If $q/p>\overline{tb}(K)$, then  there exists a contact isotopy taking
  $L$ to $L'$ such that $\Lambda_i$ is mapped to $\Lambda'_i$.
\item[$\maxtb(K)$-Slope Cables] If $q/p = \overline{tb}(K)$ and   $K$ is not a cable knot or $K$ is an $(r,s)$-cable knot and $q/p\not=rs$,
then there exists a contact isotopy
  taking $L$ to $L'$ such that $\Lambda_i$ is mapped to $\Lambda'_i$ if and only if the
  ordering of the components  with $tb = \overline{tb}(K)$  is preserved.
\item[Lesser-Slope Cables] If $q/p<\overline{tb}(K)$ and $K$ is not a cable knot or $K$ is an $(r,s)$-cable knot and $q/p\not=rs$,
then there exists a contact isotopy
  taking $L$ to $L'$ such  that $\Lambda_i$ is mapped to $\Lambda'_i$ if and only if the cyclic
  ordering of the components with $tb =pq$   is preserved.
\end{description}
\end{theorem} 

Immediately from the Legendrian classification, we obtain a classification of transverse representatives of cable links.
\begin{theorem}[Ordered Transversal Cable Link Classification] 
For $n \geq 2$, $p, q \in \mathbb Z$ such that $p \geq 1$ and $\gcd(p,q) = 1$, the cable link $K_{(np,nq)}$ of a uniformly thick, Legendrian simple  knot type $K$ is transversely simple. All such transverse links destabilize to the unique maximal self-linking number representative whose self-linking number will depend on the slope $q/p$ as follows.
\begin{description}[align=left]
\item[Greater-Slope Cables] If $q/p>\overline{tb}(K)$, then $\overline{sl}(K_{(p,q)})= pq-q+p\, \overline{sl}(K)$,
\item[Integral and  Lesser-Slopes Cables] If $q/p \leq \overline{tb}(K)$ and {$p = 1$},  
 then $\overline{sl}(K_{(p,q)})= \overline{sl}(K)$,
\item[Nonintegral and Lesser-Slope Cables] If $q/p<\overline{tb}(K)$ {and $p > 1$}, then \\
$\overline{sl}(K_{(p,q)})= pq-p(\overline{tb}(K)-\overline{sl}(K)+\lceil q/p\rceil)+q$. \end{description}
Moreover, all the components of such a transverse link can be permuted if they have the same self-linking number. 
\hfill \qed
\end{theorem}

The results in this paper give a complete picture of Legendrian torus links as well as cables of Legendrian simple, uniformly thick knot types, but there are several interesting questions left open and brought up by this work. 
\begin{question}
Can one classify Legendrian representatives of uniformly thick non-simple knot types?
\end{question}
We expect the techniques developed in this paper could lead to a classification of such Legendrian links. For such a knot type $K$, one has some finite number of maximal \tb invariant  
knots. We expect that all cable links will destabilize to a standard cable of one of these knots or to a twisted version in the integral case when the cable slope is less than  
{$\overline{tb}(K)$},
 and that two such links will become isotopic only if all the components of each link have been stabilized enough for the constituent knots to become isotopic.  We believe the ordered classification will follow along lines similar to the ones discussed above. In particular, in some cases, we expect the cable links to be non-transversely simple. Good candidates for such knots are negative twist knots, see \cite{EtnyreNgVertesi13, MinPre}. 
 \begin{question}
What is the classification of Legendrian representatives of cables of a non-uniformly thick knot type?
\end{question}

A knot $K$ can fail to be uniformly thick in two ways. It can admit tori in the knot type with convex boundary having dividing slope less than {$\overline{tb}(K)$}
that do not thicken to a standard neighborhood of a maximal \tb invariant representative of the knot type, or there can be tori in the knot type that have convex boundary with dividing slope greater than $\overline{tb}(K)$.
  In the former case we expect there to be similar results to those presented here. For example, we expect it to be a tractable problem to determine the Legendrian cable links of positive torus knots. The latter case seems to be a more difficult problem requiring new ideas. 

\begin{question}
In a basic slice, are two pre-Lagrangian tori with characteristic foliations of the same slope  necessarily   isotopic? If the tori have some leaves in common, can the isotopy be done relative to those leaves?
\end{question}

A key part in the proofs of our ordered classification results involves understanding the ordering of the components of torus links and cable links coming from pre-Lagrangian tori. From this we expect that the answer to the above questions to be YES, and such an answer would simplify several of our proofs (and make their geometric content more obvious). We discuss this more thoroughly in Remark~\ref{pre-lag-questions}.

 \medskip
 
 \subsection{Contact structures on thickened and solid tori}
In our proofs we need to use many results about contact structures on $T^2\times[0,1]$, some of which are new. We mention several of the ones that might be of general interest here. Any unfamiliar terminology will be defined  where the result occurs in the text. 

The first result restricts the slopes of pre-Lagrangian tori in a thickened torus. 
\begin{proof}[\bf Lemma~\ref{notut}]{\em
Let $\xi$ be a minimally twisting contact structure on $T^2\times [-1,1]$ that is the union of a $\mp$ basic slice on $T^2\times [-1,0]$ and a $\pm$ basic slice on $T^2\times[0,1]$. If $s_i$ denotes the slope of the dividing curves on $T^2\times\{i\}$, $i = -1, 0, 1$, then there is no pre-Lagrangian torus parallel to the boundary in $(T^2\times[-1,1],\xi)$ whose characteristic foliation has slope $s_0$. }
\end{proof}
We also study the relation between pre-Lagrangian tori and convex tori in thickened tori and enhance Lemma~3.17 from \cite{EtnyreHonda01b}. To state this result we first must develop the idea of a {\em complementary annulus} in Section~\ref{nonrotativesection}. This is an annulus in a thickened torus with a non-rotative contact structures that determines the contact structure and can also give a cyclic order to the dividing curves of convex tori.
\begin{proof}[\bf Lemma~\ref{prelagmodel}]{\em 
Consider a   basic slice $T^2\times [0,1]$, where $T^2 \times \{0\}$ has slope $s_0$, and
$T^2 \times \{1 \}$ has slope $s_1$. Suppose $s\in (s_0, s_1) \subset \partial D^{2}$, and 
$T' \subset T^{2} \times [0,1]$ is a boundary parallel convex torus in standard form with slope $s$. Then:

\begin{enumerate}
\item  There is a pre-Lagrangian torus $T$ isotopic to $T'$ that intersects $T'$ transversely, and  $T' \cap T$ is exactly the union of the Legendrian divides of $T'$.
 \item  Furthermore, 
 $T'$ is contained in a non-rotative thickened torus $R'$ of slope  $s$ that has 
a complementary annulus $A'$ with two dividing curves that run from one boundary component of $A'$ to the other; $A'$ 
defines a cyclic ordering of the Legendrian divides of $T'$ that agrees with the cyclic ordering of these curves on the pre-Lagrangian $T$.   
\end{enumerate}}
\end{proof}
We prove an analogous result in Lemma~\ref{prelagmodel2} for tight contact structures on solid tori. 
 
 \medskip
 
 \subsection{Outline} 
The remainder of the paper is organized as follows.  In Section~\ref{sec:background}, after defining our notation conventions for torus links and reviewing classification results of contact structures on thickened tori,
we show that it is possible to use a ``complementary annulus'' to define a cyclic ordering of the Legendrian divides of a convex torus in a basic slice or in a solid torus with 
convex boundary having two dividing curves parallel to the core of the solid torus. We also find the maximum \tb invariant for Legendrian torus links, which in the case of $(n,-nq)$-torus links is more restrictive than the known upper bounds  for the components.  In Section~\ref{sec:positive} we establish the unordered classification of all Legendrian positive torus links, and in Section~\ref{sec:negative-knotted} we establish the unordered classification of all Legendrian negative torus links with knotted components.
In Section~\ref{negunknotedcpts}, we give the unordered classification of all Legendrian negative torus links that have unknotted components;
 this involves defining the  $t$-twisted $n$-copies of a Legendrian unknot and establishing that these
links are non-destabilizable and, when $n\geq 3$,  do not have maximum \tb invariant.  In Section~\ref{oclassification} we establish the ordered classification of all Legendrian torus links.  In particular, using convex surface theory, we reprove Mishachev's result that forbids non-cyclic permutations of the $(n,-nq)$-torus links with max $tb$, and we describe which permutations can and cannot be obtained for all Legendrian $(np,\pm nq)$-torus links.  Non-cyclic rigidity appears in permutations of the components of a Legendrian $(np,-nq)$-torus link with $tb = -pq$. By studying pre-Lagrangian tori, we first show that it is impossible to do these
non-cyclic permutations in a basic slice, and then we show that we can ``localize'' isotopies in $S^{3}$ to the basic slice situation.    
In Section~\ref{sec:cable}, we give the ordered and unordered
classification of Legendrian cable links of Legendrian simple, uniformly thick knot types.  In particular, we show that many of the Legendrian
 torus link results generalize to Legendrian cable links, however now a study of pre-Lagrangian annuli in solid tori shows that there are some slopes for which in the ordered classification not even cyclic permutations are possible.

\smallskip
\noindent
{\bf Acknowledgements:} The authors thank Lenny Ng for helpful conversations and an anonymous referee for many valuable suggestions. The authors also thank IAS where the authors rekindled their work on this project; {the third author was supported at IAS from The Fund
for Mathematics.} The second author was partially supported by  the NSF CAREER Grant DMS-0239600, NSF Focused Research Grant FRG-024466, and NSF grants DMS-1608684, DMS-1906414, and DMS-2203312. 

\section{Background and Preliminary Results} \label{sec:background}

We assume the reader is familiar with the basic notions concerning Legendrian knots and convex surface theory. Sections~2 and~3 of \cite{EtnyreHonda01b} should be sufficient, but the reader might also want to consult \cite{Etnyre05}. 

In Subsection~\ref{toruslinksconvention}, we recall the definition of a  torus link.  In Subsection~\ref{farey}, we will review the construction of the Farey graph, which provides a useful labeling scheme for curves on a torus.  Subsections~\ref{ssec:solidtori} and \ref{ssec:solid-tori}
establish what we need to know about contact structures on thickened and solid tori. In Subsection~\ref{generalLeg}, we give the definition of the $n$-copy of a Legendrian knot and make some observations about its relation to torus links.  Finally, in Subsection ~\ref{ssec:tb-invariants}, we discuss the \tb invariant for links and, in particular, for torus links.

\subsection{Torus Knots and Links}\label{toruslinksconvention}

Recall that a standardly embedded torus $T$ provides a genus one Heegard splitting of $S^3$, $S^3 =
V_0 \cup_T V_1$, where $V_0$ and $V_1$ are solid tori. Then any curve on $T$ can be written as $ p \lambda + q\mu$, 
where $\mu$ is the unique curve that bounds a disk in $V_0$, and $\lambda$ is the unique
curve that bounds a disk in $V_1$; such a curve will be called a \dfn{ $(p, q)$-torus knot}.
 We orient $\mu$ arbitrarily and then orient $\lambda$ so that
$\lambda, \mu$ form a positive basis for $H_1(T)$, where $T$ is oriented as the boundary of $V_0$. 
We will often identify the torus $T$ with a quotient of the square $[0,1] \times [0,1]/ \sim$ with  $(0, y) \sim (1, y)$ and $(x, 0) \sim (x,1)$ for all $x, y$,   
where the circle obtained as the quotient of a horizontal line (slope $0$)  corresponds to the longitude $\lambda$, and the circle obtained as the
 quotient of a vertical line (slope $\infty$) corresponds
  to the meridional curve $\mu$.  More generally, when $q \geq p \geq 1$ and $\gcd(p,q) = 1$,  a line of slope $\pm q/p$ corresponds to a $(p, \pm q)$-torus knot.

\begin{remark} \label{rem:slopes}
The slope convention here 
 is the inverse of the slope convention in \cite{EtnyreHonda01b} and in much of the early contact topology literature, but agrees with the convention typically used by topologists.  
\end{remark}

 \begin{remark} \label{rem:orientations}
 We always choose consistent orientations on the components of a $(np, \pm nq)$-torus link $(K_1, \dots, K_n)$. So we have that the linking numbers  $lk(K_i, K_j) = \pm pq$, for all $i \neq j$.
 \end{remark}

\subsection{Farey graph}\label{farey} A convenient way  to keep track of curves on $T^2$ is through  the {Farey graph}. Consider the unit disk $D^2$ with the interior given the standard hyperbolic metric. We will
label a collection of points in $\partial D^2$; these  labeled points will  be in  one-to-one correspondence with embedded curves on $T^2$.

Label the point $(0,1)$ by $0=0/1$ and  the point $(0,-1)$ by $\infty=1/0$; join these two labeled points by a hyperbolic geodesic. 
Now inductively consider points $(x,y)\in \partial D^{2}$ with $x > 0$ half way between two points with labelings $a/b$ and $c/d$. Label this point $(a+c)/(b+d)$ and join it to the points with labels $a/b$ and $c/d$  by hyperbolic geodesics. We can label points in $\partial D^{2} \cap \{x<0\}$ similarly except $\infty$ is considered as $-1/0$. See Figure~\ref{fareygraph}. Once we have a basis $\lambda,\mu$ fixed for $H_1(T^2)$, a point $a/b$ in the Farey graph corresponds to the embedded curve representing the homology class $b\lambda + a\mu$. Additionally, two curves give a basis for $H_1(T^2)$ if and only if they are joined by an edge in the Farey graph.  It is also useful to know that if two curves are represented by $a/b$ and $c/d$, then their minimal geometric intersection is given by $|ad-bc|$. Given two points $s_0$ and $s_1$ in the Farey graph, we denote by $[s_0,s_1]$ the interval on $\partial D^2$ obtained from a clockwise-oriented curve from $s_0$ to $s_1$.  
\begin{figure}[htb]{\small
\begin{overpic} 
{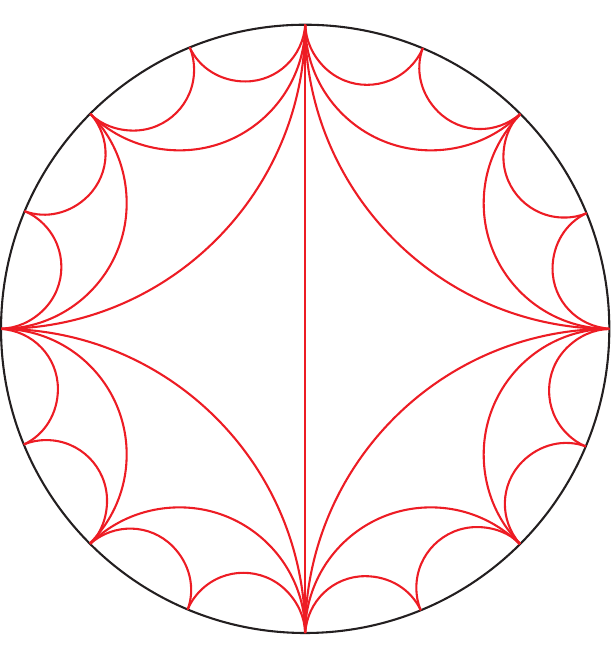}
\put(142, 2){$\infty$}
\put(145, 306){$0$}
\put(-15, 154){$-1$}
\put(297, 154){$1$}
\put(29, 43){$-2$}
\put(254, 43){$2$}
\put(19, 261){$-1/2$}
\put(254, 261){$1/2$}
\put(72, 296){$-1/3$}
\put(200, 296){$1/3$}
\put(-13, 212){$-2/3$}
\put(285, 212){$2/3$}
\put(-14, 96){$-3/2$}
\put(286, 96){$3/2$}
\put(76, 10){$-3$}
\put(203, 10){$3$}
\end{overpic}}
\caption{The Farey graph: a basic slice can be represented by a signed geodesic. }
\label{fareygraph}
\end{figure}
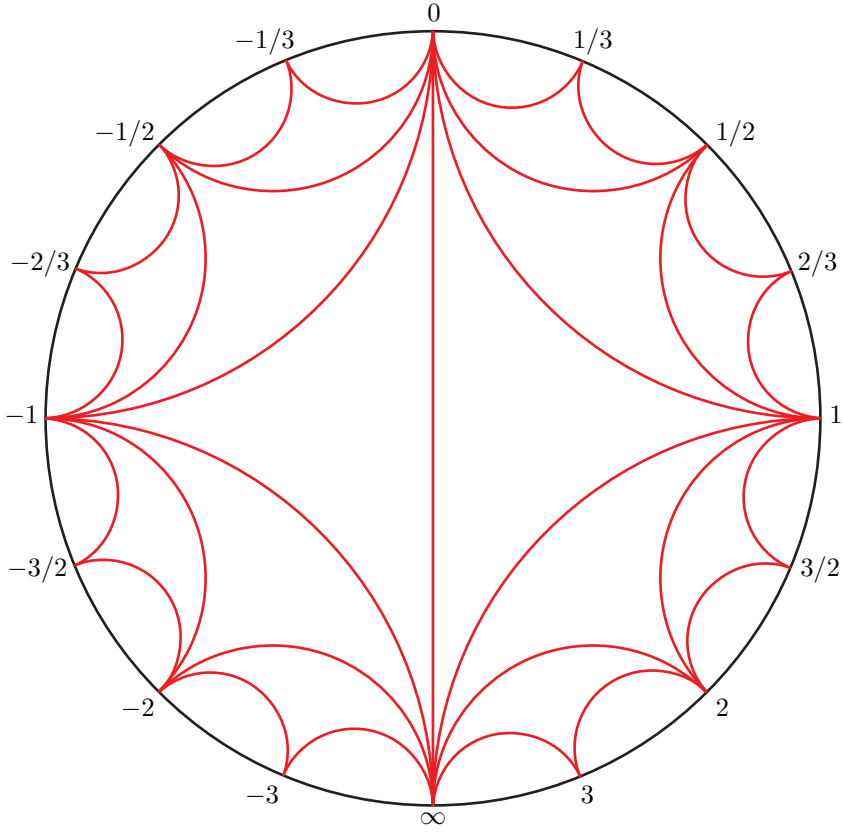

\subsection{Tight Contact Structures on $T^2 \times [0,1]$} \label{ssec:solidtori}
In this section we will consider contact structures on thickened tori. In particular, we consider basic slices and  minimally twisting contact structures, which are stacks of basic slices.  
We also develop the notion of a ``complementary annulus'' for a non-rotative contact structure; this will be useful in 
 Section~\ref{oclassification} when we analyze symmetries of negative torus links.  

Throughout this paper, we will be using established facts about convex surfaces.  In particular, recall that a convex surface will have \dfn{dividing curves} that more-or-less encode
the contact structure near the surface.   For convex tori, we will often apply  ``Giroux Flexibility" \cite{Giroux91} to assume that  the characteristic 
foliation  is in \dfn{ standard form}, meaning that it has curves of singularities parallel to the dividing
curves, called \dfn{ Legendrian divides}, and the rest of the foliation consists of linear
curves (containing singularities) of a slope $s$ that is not equal to the slope of the dividing curves.

\subsubsection{Basic Slices and Minimally Twisting Contact Structures} \label{sec:basic-slice}
A \dfn{ basic slice} is a tight contact structure $\xi$ on $T^2\times [0,1]$ such that the boundary tori $T_i= T^2\times \{i\}, i=0,1,$ are convex with {\it two} dividing curves of slope $s_i$ that are connected by a geodesic in the Farey graph, and any convex torus $T$ in $T^2\times [0,1]$ parallel to the boundary has dividing slope in $[s_0,s_1] \subset \partial D^{2}$.  Honda \cite{Honda00a} and Giroux \cite{Giroux00} have shown that there are exactly two basic slices $\xi_\pm$ for any such $s_0$ and $s_1$, and  $\xi_-=-\xi_+$. So as unoriented contact structures they are the same. 
  A basic slice can be represented by a geodesic in the Farey graph equipped with a sign. 
One can build explicit models for a basic slice with boundary  dividing
slopes $s_{0}$ and $s_{1}$ and see that any  $s\in (s_0,s_1)$ can be realized as the slope of a pre-Lagrangian torus parallel to the boundary;
recall a \dfn{ pre-Lagrangian torus} is a torus with a non-singular linear characteristic foliation.

A  contact structure $\xi$ on $T^2\times [0,1]$  is \dfn{minimally twisting} if the boundary tori $T_i= T^2\times \{i\}, i=0,1,$ are convex with 
  dividing curves of slope $s_i$
and any convex torus $T$ parallel to the boundary has dividing slope in $[s_0,s_1] \subset \partial D^2$. 
A minimally twisting  contact structure is necessarily tight and when each boundary component has just two dividing curves, it can be broken into pieces that are basic slices, \cite{Giroux00, Honda00a}. More specifically, if one takes a minimal path (of signed geodesics) in the Farey graph in the clockwise direction from $s_0$ to $s_1$, then each edge in the path corresponds to a basic slice, and $(T^2\times [0,1], \xi)$ is the concatenation of these basic slices.  It is known that $\xi$ is \dfn{ universally tight} if and only if all the signs are the same, \cite{Honda00a}.  Moreover if $\xi$ is universally tight, then for any $s\in (s_0,s_1)$ there is a pre-Lagrangian torus $T$ in $T^2\times [0,1]$ whose characteristic foliation has slope $s$, \cite{Honda00a}.   In Section~\ref{oclassification}, we will use the fact that  
this is not the case if the contact structure on $T^2\times [0,1]$ is not universally tight:

\begin{lemma}\label{notut}
Let $\xi$ be a minimally twisting contact structure on $T^2\times [-1,1]$ that is the union of a $\mp$ basic slice on $T^2\times [-1,0]$ and a $\pm$ basic slice on $T^2\times[0,1]$. If $s_i$ denotes the slope of the dividing curves on $T^2\times\{i\}$, $i = -1, 0, 1$, then there is no pre-Lagrangian torus parallel to the boundary in $(T^2\times[-1,1],\xi)$ whose characteristic foliation has slope $s_0$. 
\end{lemma}

\begin{proof} 
For a contradiction, suppose $T_0$ is a pre-Lagrangian torus in $T^2\times [-1,1]$ parallel to the boundary with characteristic foliation having slope $s_0$. Choose coordinates on $T^2$ so that $s_0=0$. Since the characteristic foliation of a surface determines the contact structure in a neighborhood, there is a neighborhood $N_0$ of $T_0$ that agrees with the standard model $T^2\times (-\epsilon,\epsilon)$ with contact structure $\ker(\cos t\, d\phi +\sin t\, d\theta)$. For sufficiently large $n$, we can find convex tori in this model with {two dividing curves of } slopes $-1/n$ and $1/n$ that cobound a thickened torus $N_0' \subset N_0$ that contains $T_0$. {Since $N_0'$ is minimally twisting (being a subset of $T^2\times [-1,1]$) and in the Farey graph we know that  there are geodesic paths from $-1/n$ to $0$ and from $0$ to $1/n$,
 the contact structure on $N_0'$ is the union of two basic slices.} The contact structure on $N_0$ (and hence on $N_0'$) is universally tight, and thus the signs of the basic slices making up $N_0'$ have the same sign.  Let $T_{0}'$ be the torus in $N_0'$ that divides $N_0'$ into two basic slices, $B_-$ and $B_+$ of the same sign.
Now $T^2\times [-1,1] \setminus N_0'$ is a union of two thickened tori $C_-$ and $C_+$ with boundary slopes $s_{-1}$ and $-1/n$ on the first and $1/n$ and $s_1$ on the second.   
The {assumption that $T^{2} \times [-1,1]$ is the union of two basic slices implies that   $C_-\cup B_-$ and  $C_+ \cup B_+$ are basic slices.} Since they are both tight, by Theorem 4.25 in \cite{Honda00a},    
 all the signs of the basic slices making up $C_{-}$ agree with those of $B_{-}$, and similarly all signs in the basic slices making up $C_{+}$ agree with the sign of $B_{+}$. But since we already observed that $B_{-}$ and $B_{+}$ have the same sign, we see that all the signs that determine $\xi$ on $T^2\times [-1,1]$ must be the same. This means that $\xi$ is universally tight, contradicting our hypothesis.  
\end{proof}

\subsubsection{Non-rotative contact structures}\label{nonrotativesection} 
 A contact structure $\xi$ on $T^2\times[0,1]$ is called \dfn{ non-rotative} if the boundary tori $T_0$ and $T_1$ are convex with equal dividing slopes, $s_1 = s_0$, and any convex torus $T$ in $T^2\times[0,1]$ that is parallel to the boundary also  has dividing slope $s_0$; these conditions imply $\xi$ is tight. 
For a non-rotative
$(T^2\times[0,1], \xi)$, we can construct  a ``complementary annulus'' as follows.

\begin{definition}[Complementary Annulus] \label{defn:complementary} Suppose $T^2\times[0,1]$ has a non-rotative contact structure with dividing slope $s$. 
Choose $t$ such that  $s$ and $t$ are connected by an edge in the Farey graph. 
By isotopy, we can assume the boundary of $T^2\times [0,1]$ has
ruling curves of slope $t$; let $L_t \subset T^2$ be a curve of slope $t$.  It is possible to
construct  an annulus $A$ in $T^2\times[0,1]$ such that $A$  is isotopic to $L_t \times [0,1]$, $A$ is convex, and $\partial A$ consists of ruling curves.
We call this a \dfn{ complementary annulus} for the non-rotative $T^2\times[0,1]$.
 \end{definition}

 \begin{lemma}\label{rem:annulus-unique} A non-rotative contact structure on $T^{2} \times [0,1]$ of slope $s$ is uniquely determined by the dividing curves on a complementary annulus.
 \end{lemma}
 
\begin{proof}    
If you have two non-rotative contact structures on $T^2\times[0,1]$ with the same boundary (and hence we assume the characteristic foliation on their boundaries is the same) and the dividing curves on their complementary annuli agree, then after isotopy we can assume that the complementary annuli and the annuli's characteristic foliations agree. 
 Then 
the contact structures are isotopic in a neighborhood of $\partial(T^2\times [0,1])\cup A$. The complement of this region is a solid torus with convex boundary having two
 longitudinal dividing curves (since the slope of the dividing curves and the slope of the complementary annulus form a basis for the homology of $T^2$), which has a unique tight contact structure; see Theorem~\ref{thm:solid-tori}. 
 \end{proof}

\begin{remark} [Non-rotative models]  \label{rem:non-rot-model}
Given an  annulus $A$ and a set of curves $\Gamma$ that can arise as the dividing curves for a non-rotative structure (that is the dividing curves run from one boundary component of the annulus to the other), 
we can build an explicit model for a non-rotative contact structure on $T^{2} \times [0,1]$ such that $A$ is a complementary annulus with dividing set $\Gamma$; our construction will yield a contact structure that is $S^1$-invariant in the direction of the dividing curves on the boundary.  
For this construction, start with  $A$ and $\Gamma \subset A$. We first 
  construct a foliation on $A$ satisfying the conditions of $\Gamma$ being dividing curves for the foliation, \cite{Giroux91}.
  There is an $\R$-invariant contact structure on $\R\times A$ that induces $\Gamma$ as the dividing curves on $A\times \{0\}$. Quotienting by the action of $\Z$ on $\R$ will give a contact structure on $S^1\times A=T^2\times [0,1]$ that is $S^1$-invariant. Moreover, one can show that the contact structure is non-rotative and that the dividing curves on the boundary of $T^2\times [0,1]$ are parallel to the first $S^1$ factor. 
 Thus we have an explicit model for any non-rotative contact structure that is $S^1$-invariant, where the $S^1$-action is in the direction of the dividing curves on the boundary. 
\end{remark}

{In a pre-Lagrangian torus, there is a natural cyclic ordering of the leaves of the foliation. In a convex torus, there is again a natural cyclic ordering of
leaves made from {\it ruling curves}, but it is less obvious how to define the ordering of leaves 
formed from the {\it Legendrian divides}.}
In Subsection~\ref{order}, 
we will define a cyclic ordering of the Legendrian divides of a convex torus that {lies either in a basic slice or in a universally tight $T^{2} \times [0,1]$.} {Our model} will be 
 a  convex torus with a non-rotative $T^{2} \times I$ neighborhood that has
a complementary annulus $A$ with a dividing set
consisting of {\it two} curves that run from one boundary component {to the other}.  
We now explain our model $T^2 \times I$ and how  the complementary annulus defines a cyclic  ordering of the Legendrian divides on convex tori  in this model. 

\begin{lemma}  \label{lem:model-order}Consider a non-rotative $T^2 \times [0,1]$  of slope $s$ 
with a complementary annulus $A$ that has two dividing curves that run from one boundary component of $A$ to the other. A closed curve $\gamma$ in the interior of $A$ isotopic to the core of $A$ that transversally intersects the dividing curves of $A$ gives rise to a convex torus $T_{\gamma}$.  From $A$, we can define 
a cyclic ordering of the Legendrian divides of $T_{\gamma}$. Moreover, any complementary annulus isotopic to $A$ 
 will induce the same cyclic ordering on the Legendrian divides of $T_{\gamma}$.
\end{lemma}

\begin{proof}  As mentioned in Lemma~\ref{rem:annulus-unique}, the slope of the non-rotative neighborhood and the dividing 
curves on the complementary annulus uniquely determine the contact structure.  So 
  let $A$ be an annulus with dividing set $\Gamma$ consisting of
two curves  that each run from one boundary component to the other.
Using the construction from Remark~\ref{rem:non-rot-model}, we obtain a non-rotative contact structure $\xi$ on $T^2\times [0,1]$ where
each boundary torus will have two dividing curves. This contact structure is $S^{1}$-invariant in the direction of the dividing curves on the boundary tori and is also  $[0,1]$-invariant.

From curves on $A$, we can construct convex tori.
Let $\gamma$ be any closed curve in the interior of $A$ that is isotopic to the core of $A$ and has a transversal intersection with $\Gamma$.  
Then $T_\gamma=S^1\times \gamma$   is a convex torus that splits $T^2\times [0,1]$ into two thickened tori $R_-$ and $R_+$. The contact structure $\xi$ restricted to each of these is non-rotative. We also know that the number of dividing curves on $T_\gamma$   is $2n= |\gamma \cap \Gamma|$ and $S^1\times (\gamma \cap \Gamma)$ are Legendrian divides on $T_\gamma$, since the dividing curves on $A$ are precisely where the contact structure is tangent to the $S^1$ fibers.

 {The complementary annulus  $A$ defines a cyclic ordering of the Legendrian divides of $T_\gamma$} as follows; {see Figure~\ref{makeorder}.}     
  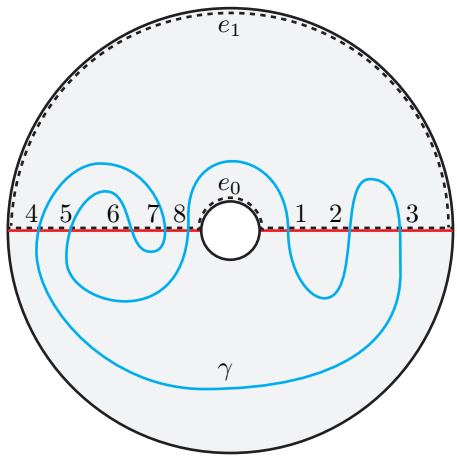
\begin{figure}[ht]
\small
\begin{overpic}%[grid,tics=10] 
{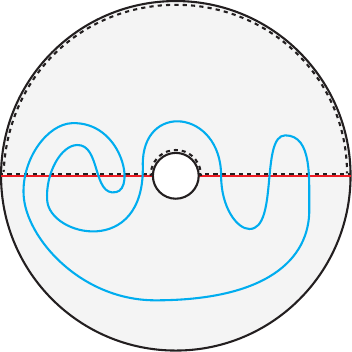}
\put(80, 30){$\gamma$}
\put(109, 88){$1$}
\put(122, 88){$2$}
\put(151, 88){$3$}
\put(7, 88){$4$}
\put(20, 88){$5$}
\put(38, 88){$6$}
\put(53, 88){$7$}
\put(63, 88){$8$}
\put(80, 100){$e_0$}
\put(80, 160){$e_1$}
\end{overpic}
  \caption{The annulus $A$ in the construction of  $T^2\times [0,1]$. The dividing curves are the horizontal red curves and $\gamma$ is shown in blue. The dotted circle is used to order the Legendrian divides on $T_\gamma$. }
\label{fig:pre-lag-model}                    
  \label{makeorder}    
\end{figure}  

 Recall $\Gamma$ divides the convex $A$ into positive and negative regions; for $i = 0,1$, let $e_i$ be the arc in $S^1\times \{i\}\subset \partial A$ that is contained in the negative region of  $A$.  Then the union of $e_0,e_1,$ and the two arcs in $\Gamma$ form a closed loop on $A$; orient this loop so that the portions of it coming from $e_i$ are oriented as the boundary of $A$. Now when one traverses this oriented loop, one encounters all of the elements of $\gamma \cap \Gamma$, and hence a cyclic order is induced on these points, which correspond to the Legendrian divides on $T_\gamma$. 
  
We claim any other convex annulus isotopic to $A$ defines the same ordering. To see this let $A'$ be another annulus isotopic to $A$ with boundary rulings curves and intersecting $T_\gamma$ in a ruling curve. The torus $T_\gamma$ breaks $A$ and $A'$ into two sub-annuli $A_\pm$ and $A'_\pm$ where $A_+$ and $A'_+$ have boundary on $T^2\times \{1\}$. Notice that $A_\pm\times S^1$ is an $S^1$-invariant contact structure determined by the dividing curves on $A_\pm$. According to \cite[Proposition~4.4]{Honda00b} any other convex surface isotopic to $A_\pm$ will have dividing set containing that of $A_\pm$ (after isotopy). Thus the dividing set of $A'_\pm$ contains that of $A_\pm$. But since the dividing set of $A_\pm$ has arcs that run across the annuli, there can be no closed dividing curves in $A'_\pm$. Moreover, since the number of times the dividing curves intersect $T_\gamma$ and $T^2\times \{i\}$ is determined by the number of dividing curves on $T_\gamma$ and $T^2\times \{i\}$, we see that the dividing curves on $A'_\pm$ agree with those on $A_\pm$. Thus the ordering of the Legendrian divides coming from $A'$ will be the same as the one coming from $A$. 
\end{proof}

\subsubsection{Neighborhoods of pre-Lagrangian tori and convex tori} \label{order}

In \cite[Lemma 3.17]{EtnyreHonda01b}, it is shown that the Legendrian divides on a convex torus arise as the intersections of the convex torus with a pre-Lagrangian torus.
We will need an enhanced version of this lemma that includes a special non-rotative neighborhood of the convex torus; the corresponding complementary annulus  will enable
 us to define a cyclic ordering of the Legendrian divides on the convex torus, which will coincide with the cyclic ordering on the pre-Lagrangian. 
 We explain this special neighborhood in  the following local model {where we will start with a pre-Lagrangian
 torus and construct models for convex tori.}

 \begin{lemma}\label{local-model} In any neighborhood of a pre-Lagrangian torus $T$, one can construct a convex torus $T'$ with Legendrian divides given as $T' \cap T$. {For any choice of such $T'$}, 
 one can find
 a non-rotative $T^2 \times I$ containing $T'$ that has a complementary annulus $A$ consisting of two dividing curves that run from one boundary component to the other. This
  $A$ induces a cyclic
 ordering of the Legendrian divides of $T'$, {and this cyclic ordering agrees with that given by the pre-Lagrangian $T$. }
 \end{lemma}

\begin{proof} 
Recall that the characteristic foliation on a surface determines the contact structure up to isotopy in a neighborhood of the surface.  Thus
given a pre-Lagrangian torus $T$, we can choose coordinates on $T$ so that the foliation has slope $0$, and the contact structure in a neighborhood 
$T \times (-\epsilon, \epsilon)$ is given by $\ker (\cos t\, d\phi + \sin t\, d\theta)$, where $t \in (-\epsilon, \epsilon)$, and  $\theta,\phi \in S^1$ are angular coordinates on $T$.  
Consider the  circle $C \subset T \times \{ t = 0 \}$ given by $C = \{ \theta = 0, t = 0 \}$, and then the annulus  
$A =  C \times (-\epsilon, \epsilon) = \{\theta = 0, t \in (-\epsilon, \epsilon) \} \subset T \times (-\epsilon, \epsilon)$. Observe that $A$ is convex: the contact vector field $\frac{\partial}{\partial \theta}$ is transverse to $A$,
and the dividing curve of $A$  is the circle $C$.  Keep in mind that our pre-Lagrangian  $T$ is $S^1_{\theta}\times C$.

 \begin{figure}[ht]
\small
\begin{overpic} 
{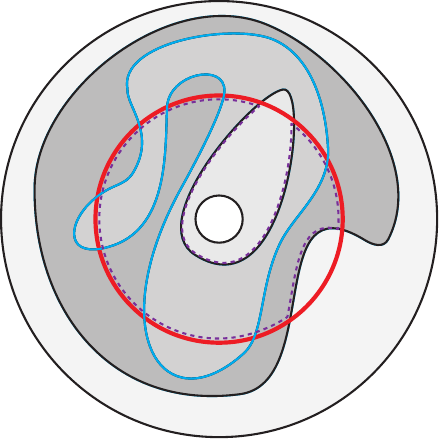}
\put(50, 60){$C$}
\put(53, 160){$\gamma$}
\put(109, 168){$1$}
\put(85, 164){$2$}
\put(65, 156){$3$}
\put(39, 118){$4$}
\put(42, 83){$5$}
\put(63, 51){$6$}
\put(124, 41){$7$}
\put(162, 129){$8$}
\put(136, 30){$\gamma^+$}
\put(100,73){$\gamma^-$}
\end{overpic}
  \caption{ Starting with a pre-Lagrangian torus $T = C \times S^1_\theta$, $T_{\gamma} = \gamma \times S^1_\theta$ is a convex torus with
Legendrian divides given by $(\gamma \cap C) \times S^1_\theta = T_{\gamma} \cap T$.  The tori $T_{\gamma^\pm} = \gamma^\pm \times S^1_\theta$ form the boundary of a region $R$ with a
non-rotative contact structure such that $A \cap R$ is a complementary annulus with dividing curves given by $A\cap C$.
The dotted curve $p$, constructed as in the proof of Lemma~\ref{lem:model-order} 
from the dividing curves and portions of the boundary of the complementary annulus, 
 induces a cyclic ordering on the Legendrian divides of $T_{\gamma}$ that agrees with 
the cyclic ordering from $T$. }
\label{fig:pre-lag-model}                    
  \label{ffig:pre-lag-mode}    
\end{figure}
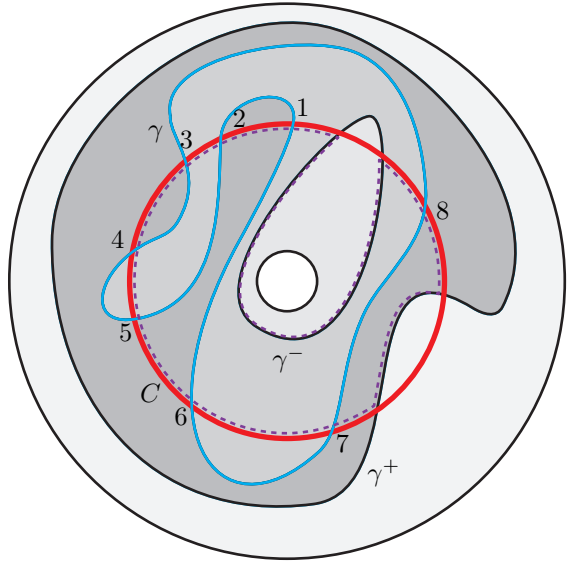

 We now show how curves on $A$ give rise to convex tori whose Legendrian divides are given by the intersections of
these tori with the pre-Lagrangian $T$.  
Let $\gamma \subset A$ be the image of the circle $C$ under a smooth isotopy such that $\gamma \cap C \neq \emptyset$, and,   
 for all $p \in\gamma \cap C$,
$T_p\gamma$ is spanned by $\frac{\partial}{\partial t}$; {see Figure~\ref{fig:pre-lag-model}}.    Then  $T_\gamma=   S^1_{\theta} \times \gamma$
is a convex torus: singularities of the characteristic foliation happen precisely along $S^1_{\theta}\times (\gamma \cap C) $, and it follows that we can choose dividing curves for the foliation
thus guaranteeing $T_{\gamma}$ is convex.
  By perturbing $\gamma$ relative to $\gamma \cap C$,  it is possible to assume
  the characteristic foliation is in standard form with ruling curves of slope $\infty$, meaning parallel to $C$ ; observe that  there are $|\gamma\cap  C|$ dividing curves
  and the  Legendrian divides of $T_{\gamma}$ are exactly $T_\gamma \cap T$.
  
  Next we show that, from curves $\gamma^{\pm}$ on $A$, we can construct a non-rotative neighborhood $R$ of $T_\gamma$ such that both boundary components of $\partial R$ are
  convex tori with two dividing curves, and {the complementary annulus for $R$ is given by $A \cap R$}.
Let $\gamma^\pm \subset A$ be disjoint images of the circle $C$ under a smooth isotopy such that  $\gamma^\pm$ transversally intersects $C$ in {\it two} points, and $\gamma$ is contained
in the region of $A$ bounded by $\gamma^\pm$; see Figure~\ref{fig:pre-lag-model}.  Following a procedure as in the construction of $T_{\gamma}$ above, we let 
 $T_{\gamma^\pm} = S^1_\theta\times  \gamma^\pm$;   the $T_{\gamma^\pm}$ are convex with two dividing curves of slope $0$, meaning parallel to the  $\theta$-axis.
 The contact structure on the region $R$ between the two tori $T_{\gamma^\pm}$  is non-rotative.
 It follows that $R$ can be identified with an $I$ invariant neighborhood of a convex torus with two dividing curves.
  Notice that $T_\gamma$ splits $R$ into two regions $R_- \cup R_+$. As mentioned above, we can assume all the curves are chosen so that the
  tori $T_{\gamma^-}, T_\gamma$ and $T_{\gamma^+},$ have ruling curves of slope $\infty$. The annuli $A_\pm =A\cap R_\pm$ can be slightly perturbed to have boundary being the union of ruling curves. Then the $A_\pm$ are convex annuli with Legendrian boundary, and the dividing curves can be taken to be $T \cap A_\pm = C \cap A_\pm$. The contact structure on each $R_\pm$ is non-rotative and is completely determined by the isotopy class of the dividing curves on $A_\pm$. Notice that $R$ is exactly the type of non-rotative thickened torus considered in Lemma~\ref{lem:model-order}.
     Thus the complementary annulus $A_+\cup A_-$ determines an ordering on the Legendrian divides of $T_\gamma$ that agrees with the ordering coming from $T$.
   \end{proof}

  The following is our needed enhancement of \cite[Lemma 3.17]{EtnyreHonda01b}.

\begin{lemma}\label{prelagmodel} 
Consider a   basic slice $T^2\times [0,1]$, where $T^2 \times \{0\}$ has slope $s_0$, and
$T^2 \times \{1 \}$ has slope $s_1$. Suppose $s\in (s_0, s_1) \subset \partial D^{2}$, and 
$T' \subset T^{2} \times [0,1]$ is a boundary parallel convex torus in standard form with slope $s$. Then:

\begin{enumerate}
\item  There is a pre-Lagrangian torus $T$ isotopic to $T'$ that intersects $T'$ transversely, and  $T' \cap T$ is exactly the union of the Legendrian divides of $T'$.
 \item  Furthermore, 
 $T'$ is contained in a non-rotative thickened torus $R'$ of slope  $s$ that has 
a complementary annulus $A'$ with two dividing curves that run from one boundary component of $A'$ to the other; $A'$ 
defines a cyclic ordering of the Legendrian divides of $T'$ that agrees with the cyclic ordering of these curves on the pre-Lagrangian $T$.   
\end{enumerate}
\end{lemma}

The statement of $(1)$ is given in \cite[Lemma 3.17]{EtnyreHonda01b}, but no explicit proof is given.  Here we will give the proof of the strengthened statement.

\begin{proof}  Fix a basic slice $T^2 \times I$ with boundary slopes $s_{0}, s_{1}$, and fix $s \in (s_{0}, s_{1})$.
 
Given a convex torus $T'$ of slope $s$ in a basic slice that is parallel to the boundary, one can find convex tori $T_\pm$ in $T^2\times[0,1]$ that both have two dividing curves of slope $s$ and cobound a region $R'$ with a non-rotative contact structure that contains $T'$ (this is just by attaching bypasses to $T'$ that can be found on annuli from $T'$ to $T^2\times\{i\}$  with boundary being curves of slope $s_i$). 
Choose a complementary annulus $A'$ (see Definition~\ref{defn:complementary}) for the non-rotative region $R'$ where the  boundary curves of $A'$ have  slope $t$;
 {the convex $A'$ will have two dividing curves that run from one boundary component to the other.}
 In this way, we break our basic slice into three regions: $R'_0$ (containing $T^2\times \{0\}$), $R'$, and $R'_1$ (containing $T^2\times \{1\}$). Moreover, $R'$ is split into two thickened tori $R'_-\cup R'_+$ by $T'$ and we can set $A'_\pm=A'\cap R'_\pm$. 
 
We now move to a
model situation {that incorporates the pre-Lagrangian torus} that has the stated properties, and then 
we will prove this model is contactomorphic to our situation as described in the previous paragraph. 
 In our basic slice, we can find a pre-Lagrangian torus $T$ {of slope $s$}. By Lemma~\ref{local-model}, in any neighborhood of  $T$
 we can find a convex torus $T_\gamma$ and convex tori $T_{\gamma^\pm}$ that bound a non-rotative region $R$ containing $T_\gamma$. 
 Choose a complementary annulus $A$ of slope $t$ for $R$. The torus $T_\gamma$ splits $R$ into two thickened tori $R_-\cup R_+$. Set $A_\pm=A\cap R_\pm$. 
 
 By the appropriate choice of $\gamma$ we can assume that the dividing curves on $A_\pm$ and $A'_\pm$ agree. 
Now the   tori $T_{\gamma^{\pm}}$ divide
our basic slice into three pieces: $R_0$, $R$, and $R_1$. Since the dividing curves on $A$ and $A'$ agree, $R$ and $R'$ are contactomorphic; see Lemma~\ref{rem:annulus-unique}.  Moreover, this contactomorphism takes $T_\gamma$ to $T'$. Indeed, since the dividing curves on $A_\pm$ and $A'_\pm$ agree, there is a contactomorphism from $R_-$ to $R'_-$ and from $R_+$ to $R'_+$. Taken together these give the desired contactomorphism of $R$ to $R'$. The classification of contact structures on thickened tori implies that $R_i$ and $R'_i$ are also contactomorphic. Thus we have a contactomorphism of our basic slice that takes $T_\gamma$ to $T'$.

 The image of $T$ is the desired pre-Lagrangian torus  that satisfies (1).
For the first part of statement $(2)$ we can take $T_\pm$
 to be the image of $T_{\gamma^\pm}$ under the contactomorphism. In the model, 
 the ordering of the Legendrian divides on $T_\gamma$ given by $A$ is the same as given by $T$, and thus the same will be true for $T'$. 
\end{proof}
 
\begin{remark}\label{prelagmodel-extended}
The lemma can be extended from basic slices to universally tight contact structures $\xi$ on $T^2\times[0,1]$. We know that a universally tight $\xi$ is obtained by stacking together several basic slices with the same sign. If the slope $s$ lies in a basic slice then one just repeats the proof in this basic slice and ignores the others. If $s$ is one of the slopes on the boundary of one (and hence two) basic slices, then from our discussion above we know that there is a pre-Lagrangian torus with the desired slope contained in the union of the two basic slices with boundary having slope $s$. Then one may repeat the above proof with little modification to reach the same conclusion. 
\end{remark}

\subsection{Tight Contact Structures on $S^1 \times D^2$ and neighborhoods of Legendrian knots}\label{ssec:solid-tori}

We begin with the simplest classification result. 
 
\begin{theorem} [Kanda 1997 \cite{Kanda97}] \label{thm:solid-tori} For all $q \in \mathbb Z$, there is a unique  contact structure on $S^1 \times D^2$ with
 convex boundary having two dividing curves of slope $q$.
\end{theorem}

\begin{remark} Classification results on contact manifolds with convex boundary are usually stated in terms of the dividing curves.
It is useful to keep in mind that the characteristic foliation on the boundary is an invariant.   So, for example, the uniqueness statement for contact structures on $S^{1} \times D^{2}$ with a ``convex boundary having two dividing curves of slope $q$''  is shorthand for a uniqueness statement for contact structures on $S^{1} \times D^{2}$ with a
 ``fixed characteristic foliation that has two dividing curves of slope $q$;'' {uniqueness then means unique up to an isotopy fixing the boundary.}
\end{remark}

\begin{remark}[Model for solid torus with boundary slope $0$]\label{ex:n-copy-model} 
There is a simple model for the unique contact structure on $S^{1} \times D^{2}$ with convex boundary having two dividing curves of slope $0$.  
Consider a neighborhood $N_0$ of the $x$-axis in
$\R^3/\sim$, where $(x,y,z)\sim (x+1,y,z)$ and $\R^3$ has the standard contact structure $\xi=\ker
(dz-ydx)$. We can assume the boundary of the neighborhood is convex with two dividing
curves of slope $0$.  By Giroux Flexibility \cite{Giroux91}, we can assume the characteristic 
foliation is in standard form
 (as described at the beginning of Subsection~\ref{ssec:solidtori}).
Notice that $\Lambda_0 = \{ y = z = 0 \}$ is a Legendrian knot in $N_0$, and the annulus $A_0 = \{ y= 0\} \subset N_0$ is
foliated by Legendrian curves isotopic to $\Lambda_0$.  
\end{remark}

\begin{remark}  By considering the standard model in Remark~\ref{ex:n-copy-model}, we see that if $S^1\times D^2$ has convex boundary  with dividing slope $q\in \mathbb Z$, for any 
$s \in (-\infty, q) \subset \partial D^{2}$ (as defined in Subsection~\ref{farey}), there is a  pre-Lagrangian torus $T^2$ parallel to the boundary that has slope $s$.   
\end{remark}

We use the $N_{0}$  from Remark~\ref{ex:n-copy-model} as a ``standard neighborhood'' of a Legendrian knot with $tb = 0$: the convex boundary of this neighborhood will have
two dividing curves and a characteristic foliation in standard form.   Standard neighborhoods for Legendrian knots
with $tb \neq 0$ are obtained by applying a diffeomorphism to $N_{0}$.

\begin{definition} \label{std-nbhd} Given a Legendrian knot $\Lambda$,  by the Legendrian Neighborhood Theorem (see, for example, \cite[Corollary 2.5.9]{Geiges08}), we know
there is a neighborhood $N$ of $\Lambda$ and a contact diffeomorphism of $N_0$ (as defined in Remark~\ref{ex:n-copy-model}) onto $N$; such an $N$ is a \dfn{ standard neighborhood of $\Lambda$}.
The diffeomorphism from $N_{0}$ to $N$ sends $\Lambda_0$ to $\Lambda$, the two slope $0$ 
dividing curves  of $\partial N_0$ to curves of slope $tb(\Lambda)$ on $\partial N$, and the annulus $A_0$ to an annulus $A$ in $N$ foliated by curves isotopic to $\Lambda$.  
\end{definition}

The analog of Lemma~\ref{prelagmodel} for solid tori is the following. 
\begin{lemma}\label{prelagmodel2} 
Let $(S^1\times D^2, \xi)$ be a solid torus with convex boundary having 2 dividing curves of slope 0 (that is they are parallel to $S^1\times \{p\}$).
Suppose $s \in \mathbb Q$ satisfies $s\in (-\infty, 0) \subset \partial D^{2}$.  
\begin{enumerate}
\item Given any convex surface $T' \subset S^1\times D^2$ in standard form with boundary slope $s$ and isotopic
to $\partial (S^1\times D^2)$, there is a pre-Lagrangian torus $T$ isotopic to $T'$ that intersects $T'$ transversely, and  $T' \cap T$ is exactly the union of the Legendrian divides of $T'$.
 \item  Furthermore, there are two convex tori $T_-$ and $T_+$ that co-bound a non-rotative thickened torus $R'$ containing $T'$ where $T_\pm$ each have two Legendrian divides of slope $s$ that arise as $T_\pm \cap T$.  The tori $T_\pm$ together with a complementary annulus $A$ to the slope $s$ define a cyclic ordering of the Legendrian divides of $T'$ that agrees with the cyclic ordering of these curves on $T$.   
\end{enumerate}
\end{lemma}
\begin{proof}
Given the convex torus $T'$ in Item~(1) we can use the classification of contact structures on solid tori to find a universally tight thickened torus containing $T'$ in $S^1\times D^2$. The result now follows from Lemma~\ref{prelagmodel} and Remark~\ref{prelagmodel-extended}.
\end{proof}

\subsubsection{Neighborhoods of Legendrian knots} \label{stab-basic-slice}
 
Given a Legendrian knot $\Lambda$, inside  a standard neighborhood $N$ of $\Lambda$ (as described in Remark~\ref{ex:n-copy-model}), there is a standard neighborhood $N_\pm$ of the $\pm$ stabilized $\Lambda$. The dividing curves on the boundary of $N_\pm$ have slope $tb(K)-1$, so the region $R_\pm$ between $N$ and $N_\pm$ is a basic slice. Moreover the sign of the basic slice is determined by the sign of the stabilization. 

\subsubsection{Constructing $(p,-q)$-torus knots} \label{sec:p-q-torus}
At this point we review the construction of  non-trivial, maximal \tb invariant $(p,-q)$-torus knots ($p > 1$) with varying rotation numbers, \cite[Section~4]{EtnyreHonda01b}.  
Choose $m \in \mathbb Z$ such that 
$-m - 1 < -q/p < -m$.  We can decompose $S^3$ as $U_{-m - 1} \cup ([0,1] \times T^2) \cup S_{-m}$, where
$S_{-m}$ is the closure of the complement of a standard neighborhood of a Legendrian unknot $U_{-m}$ with $tb = -m$,
$U_{-m-1}$ is a standard neighborhood of a Legendrian unknot $U_{-m-1}$ with $tb = {-m-1}$, and $T^2\times [0,1]$ is a basic slice with dividing slopes $-m-1$ and $-m$, see Figure~\ref{fig:realizeneg}. 
 \begin{figure}[ht]
\small
\begin{overpic}%[grid] 
{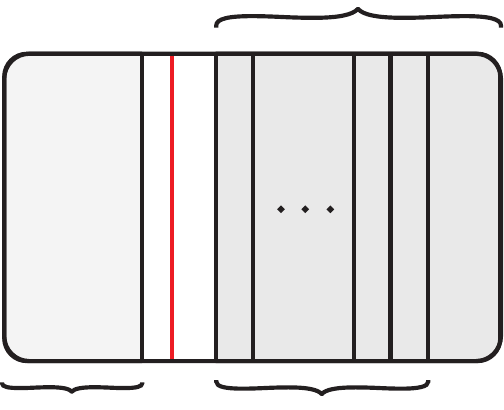}
\put(169, 200){$S_{-m}$}
\put(30, -3){$U_{-m-1}$}
\put(130, -3){$m-1$ basic slices}
\put(109, 120){$\pm$}
\put(175, 120){$\pm$}
\put(191, 120){$\pm$}
\end{overpic}
  \caption{The decomposition of $S^3$ used to find maximal \tb $(p,-q)$-torus knots. Each vertical line represents a Heegaard torus for $S^3$ and the $(p,-q)$-torus knot is found on the red torus.}                 
  \label{fig:realizeneg}    
\end{figure}
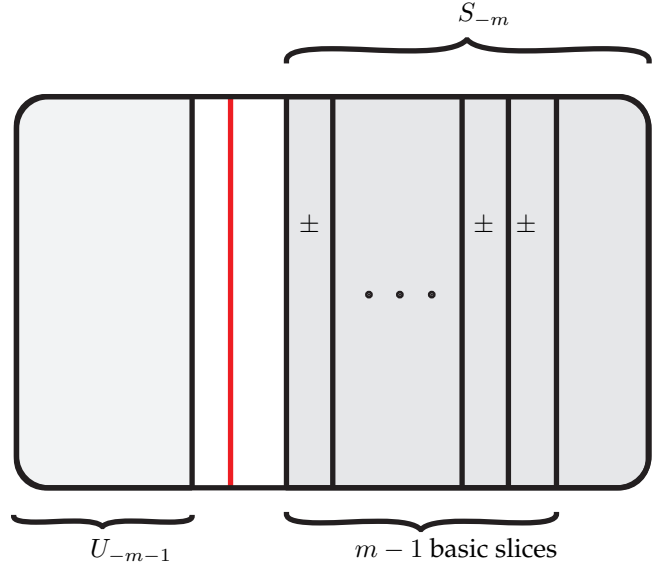  
Thus in $T^2\times [0,1]$ there is a pre-Lagrangian torus with linear characteristic foliation of slope $-q/p$. Our Legendrian $(p,-q)$-torus knot $K_{p,-q}$ is one of the leaves of this pre-Lagrangian torus. In \cite{EtnyreHonda01b} it was shown that any maximal Thurston-Bennequin invariant negative torus knot sits on such a torus for some choice of unknots $U_{-m}$ and $U_{-m-1}$. 

For $m > 1$, there will be options for the Legendrian unknot $U_{-m}$ distinguished by $r(U_{-m})$, and two choices for $U_{-m-1}$ determined by whether a positive or negative stabilization is done to $U_{-m}$; {this leads to the $2m$ Legendrian representatives of $K_{(p,-q)}$.}  The different choices for $r(U_{-m})$ and signs of stabilization for $U_{-m-1}$ will determine $r(K_{p,-q})$.  

When $p=1$, the $(1,-q)$-torus knot is an unknot, but $(n,-nq)$-torus links will be distinct links  as $q$ varies. 
One can still build models for maximal Thurston-Bennequin invariant $(n,-nq)$-torus links using pre-Lagrangian tori as shown in the proof of Lemma~\ref{maxntb}.

\subsection{The $n$-copy}\label{generalLeg}
Many non-destabilizable Legendrian negative torus links will be $n$-copies of a Legendrian torus knot.   

\begin{definition} \label{n-copy}  
For any Legendrian knot $\Lambda$, we see from Definition~\ref{std-nbhd} that 
  any standard neighborhood of $\Lambda$ contains an annular region $A$
containing $\Lambda$ that is foliated by Legendrian curves isotopic to $\Lambda$; 
the image of any $n$ curves in $A$  is the \dfn{ $n$-copy of $\Lambda$}, see the bottom row of Figure~\ref{fig-2maxtb} and upper left picture in Figure~\ref{fig8cable} for examples of $n$-copies. If $tb(\Lambda) = m$, then the $n$-copy is topologically the $n(1,m)$-cable of $\Lambda$.
\end{definition}

\begin{remark} \label{n-copy-front} In the front projection, the components of the $n$-copy can be obtained as slight shifts of $\Lambda$ 
in the $z$-direction, \cite{Mishachev03}.  
\end{remark}

\begin{lemma}\label{ncopyalt}
Let $T$ be a pre-Lagrangian torus and $\Lambda$ a leaf of the characteristic foliation of $T$. 
Then any $n$ leaves in the characteristic foliation of $T$ can be taken to be the $n$-copy
 of $\Lambda$. 
\end{lemma}

\begin{proof}
Let $\Lambda_1, \Lambda_2, \ldots, \Lambda_n$ be $n$ leaves in the characteristic foliation of $T$. Let $A \subset T$ be an annulus  that 
 contains  $\cup _i \Lambda_i$. Since the characteristic foliation on a surface determines the contact structure in a neighborhood of the surface we see that $A$ has a neighborhood $N$ contactomorphic to a neighborhood of the $A_0$ defined in  Remark~\ref{ex:n-copy-model}. {So} $N$ is a standard neighborhood of $\Lambda$, and  
 $\Lambda_i$ are leaves in the foliation of $A$. Thus the $n$ leaves form the {$n$-copy of $\leg$.}
\end{proof}

The $n$-copy exists for any Legendrian knot.  When $\leg$ is an unknot
or a negative torus knot with maximal $tb$, then  the $n$-copy will be a torus link. 
\begin{lemma} \label{n-copy-maxtb}
\begin{enumerate}
\item  {For $q > p \geq 2$, $\gcd(p,q) = 1$}, if $\Lambda$ is a nontrivial Legendrian  $(p, -q)$-torus knot with $tb(\leg) = -pq$,  
 then the $n$-copy of $\Lambda$ is
a Legendrian $(np, -nq)$-torus link $L = (\Lambda_1, \dots, \Lambda_n)$ with 
$tb(\Lambda_1) + \dots + tb(\Lambda_n) = -npq$.  
\item If $\Lambda$ is a Legendrian unknot with $tb(\Lambda) = -q$, then the $n$-copy of $\Lambda$ is
a Legendrian $(n, -nq)$-torus link $L = (\Lambda_1, \dots, \Lambda_n)$ with
$tb(\Lambda_1) + \dots + tb(\Lambda_n) = -nq$.
\end{enumerate}
\end{lemma}

\begin{proof} We need to show that when $\Lambda$ is an unknot or a negative torus knot, the $n$-copy lies on a standardly embedded torus.  
This is clear when $\Lambda$ is an unknot.
When $\Lambda$ is a  $(p,-q)$-torus knot with $tb = -pq$,   as described in Section~\ref{sec:p-q-torus}, it was shown
 in \cite{EtnyreHonda01b} that
 $\Lambda$ is a leaf in the characteristic foliation of a pre-Lagrangian torus $T$ that bounds an unknotted solid torus. Thus taking $n$ leaves in the foliation will give the $n$-copy of $\Lambda$ by the Lemma~\ref{ncopyalt}; but this is also the $(np,-nq)$-cable of the unknot, that is the $(np,-nq)$-torus link. 
 \end{proof}

\begin{remark}
Notice that for any Legendrian knot $\Lambda$, the $n$-copy of $\Lambda$ will be a $(n,n\, tb(\Lambda))$-cable of $\Lambda$.
\end{remark}

We will see in the next section, that these torus links formed as $n$-copies will have maximal \tb invariant.

\subsection{\tb  invariant of Legendrian torus links}\label{ssec:tb-invariants} \label{tb-torus-links}
The Thurston-Bennequin invariant for a link $L$ is defined in the same way as for knots. In
particular, if $L'$ is the Legendrian push-off of $L$ (that is a copy of $L$ obtained by flowing a
short time by a Reeb flow), then $tb(L)$ is the total linking of $L$ with $L'$. Using the combinatorial
description of $tb$ from a count of crossings and cusps, see for example \cite[2.62]{Etnyre05}, and the fact that each crossing is either from the same or different components,  it is not hard 
to see that for $L = (\Lambda_1, \dots, \Lambda_n)$,
\[
tb(L) = tb(\Lambda_1) + tb(\Lambda_2) + \dots + tb(\Lambda_n) + 2 \sum_{i< j} lk(\Lambda_i, \Lambda_j),
\]
For $(np, \pm nq)$-torus links $L=(\Lambda_1,\ldots,\Lambda_n)$, since  $lk(\Lambda_i, \Lambda_j) = \pm pq$ when $i \neq j$ (see Remark~\ref{rem:orientations}), we have
\begin{equation}\label{totaltb}
tb(L)= tb(\Lambda_1) + tb(\Lambda_2) + \dots + tb(\Lambda_n) \pm (n-1)(n) pq.
\end{equation}
From this we see that a link has {maximal \tb invariant} precisely when $tb(\Lambda_1) + tb(\Lambda_2) + \dots +
tb(\Lambda_n)$ is maximized. 

We now have the following observation. 

\begin{proposition} \label{maxtb}
For an oriented, Legendrian $(np, \pm nq)$-torus link, with $n \geq 2$,  let $\Lambda_1, \dots, \Lambda_n$ denote the $n$
components. 
\begin{enumerate} 
\item  For the Legendrian  $(np, +nq)$-torus link, 
$tb(\Lambda_1) + \dots + tb(\Lambda_n) \leq n(pq - p - q);$
\item For the Legendrian  $(np, -nq)$-torus link, 
$tb(\Lambda_1) + \dots + tb(\Lambda_n) \leq -npq$.
\end{enumerate}
\end{proposition}

As discussed in the introduction we know the maximum \tb invariant of torus knots, from which the proposition easily follows for torus knots with knotted components
($p > 1$) and positive torus knots with unknotted components ($1 = p \leq q$).  However, observe that for  $(n, -nq)$-links, which are negative torus links with unknotted components,
 the known upper bound of $-1$ for
 a Legendrian unknot says that for a Legendrian $(n, -nq)$ torus link, $tb(\Lambda_1) + \dots + tb(\Lambda_n) \leq -n$, which is weaker than Proposition~\ref{maxtb} when $q \geq 2$.
For negative torus knots with unknotted components we need the following result to prove the proposition. 

\begin{theorem}[Epstein 1997, \cite{Epstein97}] \label{Epstein}
If $L$ is a Legendrian $(np, -nq)$-torus link with $np$ even, then $tb(L) \leq -n^2pq.$
\end{theorem}

\begin{proof}[Proof of Proposition~\ref{maxtb}]  As discussed above we only need to verify that for a Legendrian $(n, -nq)$-torus link with $n \geq 2$ and  $q \geq 2$,
$$tb(\Lambda_1) + \dots + tb(\Lambda_n) \leq -nq.$$
 
We first assume that $n$ is even.  Then by  Theorem~\ref{Epstein}, we know
${tb(L)} \leq  -n^2q$.  Thus Equation~\eqref{totaltb} gives 
$$ tb(\Lambda_1) + tb(\Lambda_2) + \dots + tb(\Lambda_n) - (n-1)nq  \leq -n^2q.$$
So $tb(\Lambda_1) + \dots + tb(\Lambda_n) \leq -nq$, as desired.

Lastly consider the case where $n$ is odd, and thus $n \geq 3$, and suppose for a contradiction that it is possible to construct a Legendrian
version of $(n, -nq)$ with $tb(\Lambda_1) + \dots + tb(\Lambda_n) > -nq$. Then we know that there must be at
least one term with \tb invariant greater than $-q$: say
\[
tb(\Lambda_1) \geq tb(\Lambda_2) \geq \dots \geq tb(\Lambda_i) > -q \geq tb(\Lambda_{i+1}) \geq \dots \geq tb(\Lambda_n).
\]
If $i \geq 2$, then $\Lambda_1$ and $\Lambda_2$ will make a $(2, -2q)$-torus link with $tb(\Lambda_1) + tb(\Lambda_2) > -2q$, a
contradiction to the paragraph above. If $i = 1$, then we can again conclude that $tb(\Lambda_1) + tb(\Lambda_2)
> -2q$ as follows. Write $tb(\Lambda_1) = -q + j_1$, $tb(\Lambda_2) = -q - j_2$ for $j_1 > 0$, $j_2 \geq 0$. We can argue that $j_1 > j_2$ as follows.
If $j_1 \leq j_2$ then writing $tb(\Lambda_i) = -q -j_i$ where $j_i \geq 0$ for $i= 3, \dots, n$, we have
\[
\begin{aligned}
tb(\Lambda_1) + \dots + tb(\Lambda_n) &= (-q + j_1) + (-q -j_2) + (-q  - j_3) + \dots + (-q  - j_n), \\ 
&= -nq + (j_1 - j_2 - j_3 - \dots -j_n) \\ 
& \leq  -nq + (- j_3 - \dots -j_n) \\ 
& \leq -nq,
\end{aligned} 
\]
a contradiction to our starting assumption. Thus $j_1 > j_2$ and so the components $\Lambda_1$, $\Lambda_2$ make up a
$(2, -2q)$-torus link with $tb(\Lambda_1) + tb(\Lambda_2) = -q + j_1 + -q - j_2 = -2q + j_1 - j_2 >  -2q$, a contradiction to the
above paragraph.
\end{proof}

%%%%%%%%%%%%%%%%%%%%%%%%%%%%%%%%%%%%%%%%%%%%%%%%%%%%%%%%%%%
\section{Positive Torus Links}\label{sec:positive}
%%%%%%%%%%%%%%%%%%%%%%%%%%%%%%%%%%%%%%%%%%%%%%%%%%%%%%%%%%%
In this section, we will give the unordered classification of all Legendrian positive torus links. In other words, we will classify all Legendrian links that are topologically $(np, +nq)$-torus links with $n \geq 2$,  
$q \geq p \geq 1$,
and $\gcd(p, q) = 1$. 
  Observe that when $p = 1$, these are all links of unknots.  The ordered classification will be given in Section~\ref{oclassification}.

\begin{theorem} \label{thm:p-unorder-class}
Given $1 \leq p \leq q$
and $\gcd(p, q) = 1$, an unordered, oriented Legendrian  $(np, +nq)$-torus link 
is classified by the Thurston-Bennequin invariants and rotation numbers of the components.
\end{theorem}

The strategy to prove this theorem is to understand all oriented Legendrian representatives of the $(np, +nq)$-torus link
with maximal $tb$ invariant, and then show that if a link does not have maximal \tb invariant, it must destabilize to one with
maximal $tb$. 

\begin{lemma} \label{uniquemaxpos}
Given   $q \geq p \geq 1$ 
and $\gcd(p, q) = 1$, there exists a unique oriented Legendrian $(np, +nq)$-torus link with maximal
\tb invariant.  
\end{lemma}

\begin{proof}    This argument parallels the proof of \cite[Lemma 4.7]{EtnyreHonda01b}; additional details can be found there.

We first show the existence of a  Legendrian $(np, +nq)$-torus link $L$ with components $\Lambda_1, \dots, \Lambda_n$
such that $tb(\Lambda_i)  = pq - p - q$, for all $i$. Let
  $N$ be a solid torus neighborhood of
a Legendrian unknot with $tb = -1$ with convex boundary $T$ in standard form.  Then $T$ has dividing curves of slope $-1$; we can
assume that the ruling curves have slope $q/p$.  Let $L$ consist of $n$ ruling curves on $T$.  By
analyzing the number of intersections between the ruling and dividing curves, we find that
each ruling curve has $tb=pq-p-q$, as desired.

Next we consider uniqueness.  If $L$ and $L'$ are both Legendrian $(np, +nq)$-torus links
with maximal \tb invariant, then we can assume they lie as a subset of the ruling curves on convex tori $T$ and $T'$, where both $T$ and $T'$ have two dividing
curves of slope $-1$.   By Honda's classification of  contact structures on solid tori, we know there is a contact diffeomorphism of $S^3$ that takes $T'$ to $T$, and then a result of Eliashberg \cite[Corollary~2.4.3]{Eliashberg92a}
shows that there is a contact isotopy of $S^3$ that takes $T'$ to $T$.  So after isotopy, we can assume that $L$ and $L'$ are both collections of ruling curves on the
same convex torus.  It follows that there is an isotopy from the unordered link $L'$ to $L$. \end{proof}

\begin{figure}
\centerline{ \includegraphics{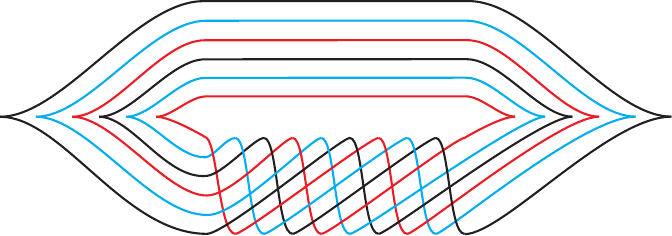} }
\caption{The front projection of a $3(2,+3)$-torus link.}
\label{pos-torus-ex22}
\end{figure}
\begin{remark} There is a simple algorithm for constructing a front projection of the max $tb$ Legendrian $(np,+nq)$-torus link:
\begin{itemize}
\item Begin with $np$ nested copies of the max $tb$ unknot with two cusps;
\item Replace a trivial $np$-stranded tangle with $q/p$ of a full positive twist; this corresponds to repeating the fundamental positive crossing tangle, as shown on the top of Figure~\ref{changes}, $nq$ times.
\end{itemize}
See Figure~\ref{pos-torus-ex22}.
\end{remark}

\begin{lemma} \label{posdestab} Suppose  $q \geq p \geq 1$
and $\gcd(p, q) = 1$.
Let $L = \amalg_{i=1}^{n} \Lambda_{i} $ be an oriented Legendrian $(np, +nq)$-torus link. 
If $tb(\Lambda_1) + \dots + tb(\Lambda_n) < n(pq - p -q)$,
then there exists a Legendrian $(np, +nq)$-torus link $L^\prime = \amalg_{i=1}^{n} \Lambda_{i}' $
such that $tb(\Lambda_1) + \dots + tb(\Lambda_n) < tb(\Lambda_1^\prime) + \dots + tb(\Lambda_n^\prime)$
and $L$ is a stabilization of $L^\prime$.
\end{lemma}

\begin{proof} This argument parallels the proof of \cite[Lemma 4.8]{EtnyreHonda01b}; additional details can be found there.

From the Relative Convex Realization Principle \cite{Kanda97}, we can assume $L$ sits on the convex
boundary  $T$ of an unknotted solid torus. We first notice that if the dividing curves on $T$ do not intersect each component of $L$ minimally, then there is a bigon in $T$ with one side on the dividing curves and one side on $L$. This gives a bypass which can be used to destabilize $L$. Thus after destabilization of some of the components of $L$ , we can assume that all components of $L$ intersect the dividing curves minimally, and thus can be assumed to be ruling curves of $T$.

If at this point, $T$ has two dividing curves of slope $-1$, then all components of $T$ have maximial \tb invariant.
Else, either  the slope of the dividing curves on $T$ is not $-1$, or $T$ has more than $2$ dividing curves of slope $-1$.
 In both of these cases, it can be argued that there is a convex torus $T'$ that is disjoint from $T$ and parallel to $T$ with two dividing curves of slope $-1$ and ruling
 curves of slope $q/p$.  Applying the Imbalance Principle, \cite[Proposition 3.17]{Honda00a},  to  an annulus with one boundary component on a component of $L$ in $T$ and the other on a ruling curve of $T'$
gives the existence of a destabilization of the component of $L$, as desired.  

In this way, we can continue the destabilization process until all components have maximal \tb invariant. 
\end{proof}

Theorem~\ref{thm:p-unorder-class} follows directly from Lemmas~\ref{uniquemaxpos} and \ref{posdestab}.

%%%%%%%%%%%%%%%%%%%%%%%%%%%%%%%%%%%%%%%%%%%%%%%%%%%%%%%%%%%
\section{Negative Torus Links with Knotted Components} \label{sec:negative-knotted}
%%%%%%%%%%%%%%%%%%%%%%%%%%%%%%%%%%%%%%%%%%%%%%%%%%%%%%%%%%%

In this section, we will classify all Legendrian negative torus links with knotted components, namely  links that are topologically $(np, -nq)$,
with $n \geq 2$, $q > p \geq 2$,    and 
 $\gcd(p,q) = 1$.  In this section,  we only consider the unordered classifcation. We will discuss the ordered classification in 
Section~\ref{oclassification} below.
  
\begin{theorem}\label{oclass}
Given   $q > p \geq 2$ and $\gcd(p,q) = 1$, unordered, oriented Legendrian $(np, -nq)$-torus links are determined by the Thurston-Bennequin and rotation number invariants of the components.
\end{theorem}

As in Section~\ref{sec:positive}, this classification follows from first understanding those with max $tb$.

\begin{lemma}\label{negmaxtb} 
 Given $q > p \geq 2$ and $\gcd(p,q) = 1$, choose $m\in \mathbb Z$ such that 
 $-m-1 < -q/p < -m$.  
Then there are $2m$ Legendrian realizations of the $(np, -nq)$-torus link with maximal \tb invariant.  
Each of these maximal $tb$ examples arise as
an $n$-copy of a Legendrian
$(p, -q)$-torus knot with maximal tb.   
\end{lemma}
\begin{proof}
Compare the proof of \cite[Lemma 4.11]{EtnyreHonda01b}.

A Legendrian  $(np,-nq)$-torus link $L$ will have maximal \tb invariant precisely when
 $tb(\Lambda_i)=-pq$, for all $i$.   It is well-known that the difference between the Seifert framing and the framing coming from $T$ is $-pq$:
 \begin{equation}\label{tb-tw}
 tb(\Lambda_i) - tw(\Lambda_i, T) = -pq.
 \end{equation}
 Thus we find that $tw(\Lambda_i,T)= 0$. 
 Applying the Relative Convex Realization Principle,
 we can assume that $L$ lies in the  convex boundary $T$ of  an unknotted solid torus, and  each component $\Lambda_i$
must be disjoint from the dividing curves. 
We may thus take $L$ to be a subset of the Legendrian divides when $T$ is isotoped
to be a convex torus in standard form with dividing slope $-p/q$.  This torus sits inside a basic slice, and hence by Lemma~\ref{prelagmodel}, we know that all the Legendrian divides
are contained as leaves inside of some pre-Lagrangian torus. Thus, by Lemma~\ref{ncopyalt}, any maximal \tb invariant $(np,-nq)$-torus link is the 
Legendrian $n$-copy of a maximal \tb invariant $(p,-q)$ torus knot.  Hence all the components have the same rotation
number, and the link is classified by this rotation number.
\end{proof}

\begin{lemma}\label{negtorusstab}
Given  
$q > p \geq 2$ and $\gcd(p,q) = 1$, let $L = \amalg_{i=1}^{n} \Lambda_{i}$ be a Legendrian $(np,-nq)$-torus link. 
 If $tb(\Lambda_1) + \dots + tb(\Lambda_n) < -npq$,
then there exists a Legendrian $(np, -nq)$-torus link $L^\prime = \amalg_{i=1}^{n} \Lambda_{i}' $
such that $tb(\Lambda_1^\prime) + \dots + tb(\Lambda_n^\prime) > tb(\Lambda_1) + \dots + tb(\Lambda_n) $
and $L$ is a stabilization of $L^\prime$.
\end{lemma}
\begin{proof}
From the Relative Convex Realization Principle, 
we can assume that $L$ is contained in the convex boundary $T$ of an unknotted solid torus.   
By applying a destabilization
if necessary, we can assume that each component of $L$ intersects the dividing set minimally.  

If at this point if $T$ has dividing curves of slope $-\frac qp$, then all components of $L$ have maximial \tb invariant.
Else, the slope of the dividing curves on $T$ is not $-\frac qp$.
 It can be argued that there is a convex torus $T'$ that is disjoint form $T$ and parallel to $T$ with dividing curves of slope $-\frac qp$.  Applying the Imbalance Principle, \cite[Proposition 3.17]{Honda00a},  to  an annulus with one boundary component being a Legendrian divide of $T'$ and the other 
a component of $L$ in $T$ (and is otherwise disjoint from $L$) 
gives the existence of a destabilization of the component of $L$, as desired.   
\end{proof}

We now know that every negative torus link will destabilize to one with maximal \tb invariant.  To finish the classification, we need
to show that a given an $n$-tuple of vertices on the mountain range can represent at most one (unordered) Legendrian link.  Any $n$-tuple of vertices that can
be destabilized to at least one $n$-tuple peak with max $tb$ invariant can be represented by exactly one Legendrian link.  Some $n$-tuples can be destabilized to a unique $n$-tuple peak.
However, there are many
$n$-tuples of vertices that can be destabilized to $n$-copies of different Legendrian knots with max $tb$.  For example, if $L$ is the $n$-copy of the Legendrian $(3, -7)$-torus knot with max $tb$ and $r = 2$, and 
$L'$ is the $n$-copy of the Legendrian $(3, -7)$-torus knot with max $tb$ and $r = 4$, 
then we need to see that applying one $+$ and stabilization to all components of $L$
produces a link that is Legendrian isotopic to the link obtained from $L'$ by applying  $-$ stabilizations to all components.
That is we need to show that 
$S_{+, all} L$ is Legendrian isotopic to $S_{-, all} L'$; 
here $S_{+, all} L$ (resp $S_{-, all} L'$) means that we have applied a $+$ (resp $-$) stabilization of all components of $L$ (resp. $L'$).  
More generally, the strategy for uniqueness is to show that 
the $n$-tuple arising from  using the valley point between adjacent peaks $n$ times has a unique representative.
From this, it will follow that all $n$-tuples of vertices on the mountain range that can represent a Legendrian link will represent a  
Legendrian link that is unique up to isotopy.

As shown in \cite{EtnyreHonda01b},  in the Legendrian mountain range of a $(p, -q)$-torus knot,  if $q = mp + e$,
all ``adjacent'' maximal $tb$ representatives with  have rotation numbers that
differ by $2e$ or by $2(p-e)$.  In the following, if $L$ is an $n$-component link, 
$S_{\pm, all}^{m} L$ means that we have applied  $m$ $\pm$ stabilizations to each component
of $L$.

\begin{lemma}\label{negstabtoiso}
Let $L$ and $L^{'}$ be two topologically isotopic, oriented, Legendrian negative torus links with each component having maximal Thurston-Bennequin invariant.   If the rotation numbers of each component of $L$
and $L^{'}$ are $r$ and $r - 2e$, respectively, then $S_{-,all}^{e}(L)$ and $S_{+,all}^{e}(L^{'})$ are Legendrian isotopic.  If the rotation numbers of each component of $L$ and $L^{'}$ are $r$ and $r -
2(p-e)$, respectively, then $S_{-,all}^{p-e}(L)$ and $S_{+,all}^{p-e}(L^{'})$ are Legendrian isotopic.  
\end{lemma}
\begin{proof} This argument parallels the proof of \cite[Lemma 4.12]{EtnyreHonda01b}; additional details can be found there.

 We consider the case where the rotation number of $L'$ is $r-2e$; the other case is similar. 
Let $T$ be a convex torus containing $L$ as in the proof of Lemma~\ref{negmaxtb}. From the classification of contact structures on solid tori \cite{Honda00a}, we know that
parallel to $T$ is a convex torus $T'$ with dividing curves of slope $-m$ (where  $q = mp + e$) that bounds a solid torus containing $T$. Make the ruling curves on $T'$ have slope 
$-\frac qp$. Let $L_s$ be any collection of $n$ ruling curves on $T'$. By connecting $L$ to $L_s$ with $n$ annuli and applying the Imbalance Principle, \cite[Proposition 3.11]{EtnyreHonda01b},
one sees that $L_s$ is $S_{-,all}^{e}(L)$. Thus $S_{-,all}^{e}(L)$ sits as $n$ ruling curves on the boundary of
a standard neighborhood of a Legendrian unknot with $tb=-m$ and rotation number $r$; the calculation of the rotation number is
discussed on pages 88 and 89 of \cite{EtnyreHonda01b}. One can similarly realize
$S_{+,all}^{e}(L^{'})$ as $n$ ruling curves on the boundary of a standard neighborhood of a Legendrian unknot with $tb =-m$ and rotation number $r$. After applying an isotopy,
we can assume that both $S_{-,all}^{e}(L)$ and $S_{+,all}^{e}(L^{'})$ lie among the ruling curves on the same convex torus, and thus are Legendrian isotopic.
 \end{proof}

%%%%%%%%%%%%%%%%%%%%%%%%%%%%%%%%%%%%%%%%%%%%%%%%%%%%%%  
\section{Negative Torus Links with Unknotted Components}\label{negunknotedcpts}
%%%%%%%%%%%%%%%%%%%%%%%%%%%%%%%%%%%%%%%%%%%%%%%%%%%%%%

In this section, we will classify all Legendrian negative torus links that have unknotted components. In other words, we will classify all Legendrian links that are topologically $(n, -nq)$ for $q \geq 1$. 
In this section, we only classify the links as unordered Legendrian links; the ordered classification can be found in 
Section~\ref{oclassification}.  The following theorem summarizes the classification; the non-destabilizable $t$-twisted $n$-copy of Legendrian unknot referenced here will be defined in Definition~\ref{t-twist}.

\begin{theorem}\label{uoclass} For $q \geq 1$,
two (unordered) oriented Legendrian $(n, -nq)$-torus links $L = \amalg_{i=1}^{n} \Lambda_{i}$ and $L' = \amalg_{i=1}^{n} \Lambda_{i}'$ are Legendrian isotopic if and only if 
there exists $\sigma \in S_{n}$ such that $tb(\Lambda_{i}) =tb(\Lambda_{\sigma(i)}')$ and $r(\Lambda_{i}) =r(\Lambda_{\sigma(i)}')$, for all $i$. 
In addition, every Legendrian $(n, -nq)$-torus link is
a stabilization of either:
\begin{itemize}
\item the $n$-copy of the Legendrian unknot with rotation number $r$ and 
$tb=-q$, denoted $nU_{-q}^r$,  or
\item  if $q \geq 2$,  the $t$-twisted $n$-copy of the Legendrian unknot with rotation number $r$ and $tb=-q+t$, for $1\leq t \leq q-1$, denoted  $T^t(nU_{-q+t}^r)$.
\end{itemize}
   When $q \geq 2$ and $n \geq 3$, the $t$-twisted $n$-copies $T^t(nU_{-q+t}^r)$ do not
have maximal \tb invariant even though they do not destabilize.  
\end{theorem}

We saw in Lemma~\ref{n-copy-maxtb} that the $n$-copy of 
an unknot with $tb = -q$ gives us
one way to construct a Legendrian $(n, -nq)$-torus link.   In fact, when $n\geq 3$, all Legendrian $(n, -nq)$-torus links with maximal 
Thurston-Bennequin invariant can be obtained as the Legendrian $n$-copy of a Legendrian
unknot with $tb = -q$.

\begin{proposition} \label{maxntb} 
For $n \geq 3$ and $q \geq 1$, the unordered oriented  $n$-component $(n, -nq)$-torus link has precisely $q$
Legendrian realizations  with maximal \tb invariant. Such a version can be constructed as the $n$-copy
of one of the $q$ Legendrian unknots with $tb = -q$. 
\end{proposition}
Before giving the proof we make two observations.
\begin{lemma}\label{tb-observation}
For $n \geq 2$, if one component of a Legendrian $(n, -nq)$-torus link has $tb=-q+t$, for $t\geq 1$,
then all the other components have $tb\leq -q-t$.
\end{lemma}
\begin{proof} Suppose $L$ is a Legendrian $(n, -nq)$-torus link.
Let $\Lambda^{+}$ be a component of $L$  with $tb(\Lambda^{+})=-q+t$. If $\Lambda$ is any other component of $L$ then $\Lambda^{+} \cup \Lambda$ is a
$(2, -2q)$-torus link, and thus Proposition~\ref{maxtb} says $tb(\Lambda^{+})+tb(\Lambda)\leq -2q$.  The result follows.
\end{proof}

{
In Section~\ref{sec:p-q-torus} we explained that the max-$tb$ non-trivial negative torus knots lie as leaves of a pre-Lagrangian torus in a basic slice, or, equivalently by 
Lemma~\ref{prelagmodel}, as the Legendrian
divides on a convex torus in a basic slice.  Legendrian unknots can be seen as Legendrian divides of a convex torus in a ``nice'' union of basic slices.}

\begin{lemma} \label{integral-union-basic} Suppose $T$ is a standardly embedded, convex torus in standard form with Legendrian divides of slope $-q$, for $q \geq 1$.  Then $T$ is contained in the universally
tight union of two basic slices.
\end{lemma}

\begin{proof} Notice that $T$ splits $S^3$ into two solid tori $V_{-\infty}$ and $V_{0}$, each with convex boundary having dividing slope $-q$, where $V_{-\infty}$ contains convex tori with dividing slope in the range $(-\infty, -q]$ and a torus $T_{-\infty}$ with two dividing curves of slope $-q$; similarly, $V_0$ contains convex tori with slopes in the range $[-q,0)$ and a torus $T_0$ with two dividing curves of slope $-q$. The region $R_{I}$ between $T_{0}$ and $T_{-\infty}$ is an $I$-invariant contact structure containing $T$.  Choose any $s$ in $[-q,0)$ that has an edge to $-q$ in the Farey graph. Let $B_{0} \subset V_{0}$ be the region between the convex torus with slope $s$ and $T_{0}$; this is a basic slice with some sign. We know that $V_{-\infty}$ is a solid torus neighborhood of a Legendrian knot. Inside $V_{-\infty}$ we have two solid tori that are neighborhoods $N_\pm$ of the $\pm$ stabilization of the Legendrian knot; $B_{-\infty}^{\pm} = V_{-\infty}\setminus N_\pm$ is a basic slice with sign $\pm$.  
 By choosing the appropriate $B_{-\infty}^{\pm}$ so that the sign agrees with the sign of  $B_{0}$,  we find that $T$ is contained in $B = B_{0} \cup \left(R_{I} \cup B_{-\infty}^{\pm}\right)$,  a universally tight union of basic slices. 
\end{proof}

\begin{proof}[Proof of Proposition~\ref{maxntb}] We first argue that, when $n \geq 3$, each component of a max-$tb$, Legendrian $(n, -nq)$-torus link  has $tb = -q$.
From Lemma~\ref{tb-observation}, we know that if one component of a Legendrian $(n, -nq)$-torus link
has $tb=-q+t$, for $t \geq 1$, then all other components have $tb \leq -q -t$.  Thus the sum of the $tb$ invariants of the components is  at most $-nq - (n-2)t < -nq$, since $n \geq 3$.
So, by Proposition~\ref{maxtb}, in order for a Legendrian $(n, -nq)$-torus link to have the max $tb$, all the components must have $tb=-q$.

There are $q$ Legendrian unknots with $tb = -q$.  By taking the $n$-copy of each of these,  we get $q$ distinct Legendrian
$(n, -nq)$-torus links with max $tb$. We must now show that if $L$ is any Legendrian $(n,-nq)$-torus link with
max $tb$, then $L$ is isotopic to one of these $n$-copies.  As argued in the above paragraph, each component of $L$
must have $tb=-q$. So if $T$ is a standardly embedded torus on which $L$ sits then the twisting of
each component of $L$ with respect to $T$ is 0. Thus we may make $T$ convex and standard relative to $L$; vanishing twist tells us that 
$L$ will be a subset of the Legendrian divides on $T$.  By Lemma~\ref{integral-union-basic}, $T$ is contained in the interior of a universally tight union of two basic slices.  This  concludes the proof since
 by Remark~\ref{prelagmodel-extended} {and Lemma~\ref{prelagmodel} },  $L$ can be taken to be leaves in a pre-Lagrangian torus, and thus, by Lemma~\ref{ncopyalt}, $L$ is the $n$-copy of one of the leaves, 
  which is a Legendrian unknot with $tb=-q$. 
\end{proof}

In contrast, when $n=2$, not all max-$tb$ Legendrian $2(1,-q)$-torus links are obtained as $2$-copies.  
For example, Figure \ref{fig-2maxtb} shows numerous examples of Legendrian $2(1,-3)$-torus
links with maximal $tb$.   All elements of the bottom row 
are obtained as doubles of an unknot with $tb = -3$.  However elements of the second and first rows have components with different $tb$ invariants, and thus cannot be doubles.  These are obtained by  introducing ``Legendrian twists" into $2$-copies of unknots with $tb = -2$, and $tb=-1$, respectively.
In general,  when $q \geq 2$, one can construct Legendrian
$n(1,-q)$-torus links by introducing ``Legendrian twists" into $n$-copies of  an unknot with $-q+1 \leq tb \leq -1$.
{Although in this section, we are only interested in twists of Legendrian unknots, in Section~\ref{sec:cable} we will consider twists of more general Legendrian knots, and so we give the general definition here.}

\begin{definition} \label{t-twist} Given a Legendrian knot $\leg$ and $t \in \mathbb Z^{+}$,  the \dfn{$t$-twisted $n$-copy of $\leg$}, denoted $T^{t}(n\leg)$ is constructed as follows.  Consider a standard
neighborhood of $\leg$; there will be two dividing curves of slope $tb(\leg)$, and we can assume that the ruling curves have slope $tb(\leg) - t$.  Then $T^{t}(n\leg)$ is the union of $\leg$ and $(n-1)$ ruling curves. 
\end{definition} 

\begin{remark} The $0$-twisted $2$-copy is merely the $2$-copy; however, in the spirit of Definition~\ref{t-twist}, we could also define the $0$-twisted $2$-copy to be the union of $\leg$ and a Legendrian divide on a standard neighborhood of $\leg$.
\end{remark}

\begin{lemma}\label{n-copy-tb-r} If $tb(\leg) = \beta + t$ and $r(\leg) = \rho$, then all components of $T^{t}(n\leg)$ will have $r = \rho$; one component will have $tb = \beta + t$ and the remaining $n-1$ components
will have $tb = \beta - t$.
\end{lemma}

\begin{proof} 
 It suffices to verify these calculations of $tb$ and $r$ in the $t$-twisted $2$-copy, for $t \geq 1$.
Let $\Lambda_1$ be one of the ruling curves of slope $tb(\leg) - t$ on a standard neighborhood of $\leg$.  Then
$tw(\leg_1,\partial N)=-\frac 12 \#(\leg_1 \cap \Gamma_{\partial N})=-t$, so, using Equation~(\ref{tb-tw}),  
$$tb(\leg_1)= -t + (tb(\leg) - t) = tb(\leg) - 2t = \beta - t.$$ 
 
It remains to show that $r(\leg_1)=r(\leg)$.  Observe that topologically $\leg_1 = 1 \lambda + (tb(\leg) -t) \mu$, where $\lambda$ is a Legendrian divide and $\mu$ is the Legendrian boundary of a convex
meridonal disk $D$ for $N$;  $r(\lambda)$ and $r(\mu)$ will determine $r(\leg_1)$.  First observe that  $\lambda$ is isotopic to $\leg$, and thus $r(\lambda) = r(\leg)$.  
Next observe that $\mu$ is an unknot.  {Since $\partial N$ has two dividing curves, we see that $D$ will have a single dividing curve; it follows that 
 $tb(\mu) = -1$,}  and thus $r(\mu) = 0$.  Then 
 arguing as in Section~4.2 of
\cite{EtnyreHonda01b}, we see that $r(\leg_1)=1 r(\lambda) - (tb(\leg) -t) (r(\mu))= r(\lambda) = r(\leg)$, as claimed.
\end{proof}

\begin{remark} \label{t-twist-tb}  From the proofs of Lemmas~\ref{2maxtb} and~\ref{nondestab} below,  we see that the front projection of $T^{t}(n\leg)$ can be obtained as follows.  Start with the $n$-copy of $\leg$.  Then
replace a trivial $n$-stranded tangle with $t$ copies of the twist tangle as shown on the right side of Figure~\ref{fig-nt}. 
For a general $\leg$ with $tb(\leg) = \beta$, the $n$-copy of $\leg$ will
be the slope $\beta$ cable of $\leg$, and  $T^{t}(n\leg)$ will be the slope $\beta-t$ cable of $\leg$. 
\begin{figure}[ht]
\small
\begin{overpic} {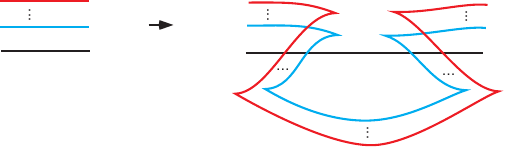}
\end{overpic}
        \caption{The Legendrian twist operation for $n$ strands.}                      
        \label{fig-nt}
\end{figure}
  \end{remark}

In fact, all Legendrian $(2, -2q)$-torus links with max $tb$   can be obtained  as the $2$-copy of 
an unknot with $tb = -q$ or as  Legendrian twists of the 2-copy of a Legendrian unknot with Thurston-Bennequin 
invariant in $\{ -q+1, \dots, -1\}$.

\begin{lemma}\label{2maxtb}
For the oriented topological $(2, -2q)$-torus link, there are precisely $q(q+1)/2$
unordered Legendrian realizations with maximal Thurston-Bennequin invariant. For every Legendrian
unknot $\leg_1$ with Thurston-Bennequin invariant in $\{ -1, \dots, -q\}$, there will be a unique
Legendrian $(2, -2q)$-torus link with maximal \tb invariant whose other component is a Legendrian unknot $\leg_2$ with $tb(\leg_2) =
-2q - tb(\leg_1)$ and $r(\leg_2)=r(\leg_1)$. 
\end{lemma}

\begin{proof}  
 { We have already seen that the $2$-copy of a Legendrian unknot with $tb = -q$ and the $t$-twisted $2$-copy of a Legendrian unknot with $tb = -q + t$, $t \geq 1$, are
max-$tb$ representatives of $(2,-2q)$-torus links; in these representatives, all components have equal rotation numbers. }

To prove the uniqueness statement, suppose $L^\prime = \leg'_1 \cup \leg_2^\prime$ is any max-$tb$ Legendrian representative of the $(2, -2q)$-torus link.  Since we know that 
$tb(\leg'_1) + tb(\leg_2') = -2q$, we know that $tb(\leg'_1) = -q + t$, for some $t \geq 0$.  Thus there is a unique $L=\leg_1 \cup \leg_2$, a $2$-copy or a twisted $2$-copy of a Legendrian unknot with max $tb$ as constructed above,
such that $tb(\leg'_1) = tb(\leg_1)$ and $r(\leg'_1) = r(\leg)$.  Since the unknot is a Legendrian simple knot, we can Legendrian isotop $\leg'_1$ to $\leg_1$, and hence we can assume that $L'=\leg_1 \cup \leg_2'$.
It remains to show that $L$ and $L^\prime$ are Legendrian
isotopic. We give separate uniqueness proofs in the cases of
$t =0$ and $t > 0$. 

When $t > 0$, as in Definition~\ref{t-twist}, we let $N_1$ be a standard neighborhood of
$\leg_1$: $\partial N_1$ is convex with two dividing curves of slope ${-q+t}$ and ruling
curves of slope $-q$. Let $T^\prime$ be a torus, parallel to $\partial N_1$, on which $\leg_2'$
sits. Since $tw(\leg_2', T')=- t <0$, we can make $T^\prime$ convex. Since we can assume
$T^\prime$ lies in the complement of $N_1$, we know the dividing slope of
$T^\prime$ is greater than or equal to ${-q+t}$. If the dividing slope $s$  is greater than ${-q+t}$, then we can argue\footnote{To see that $\#(\leg_2^\prime \cap \Gamma_{T^\prime}) > 2t$, 
choose an integer $k$ such that
$0 \geq -k \geq s\geq -k-1 \geq -q+t$.  
If $\gamma_k$, $\gamma_{k+1}$ are
simple closed curves in $T^2$ representing the homology class $(1,-k)$ and $(1, -(k+1))$
respectively, then 
$\leg_2^\prime \cdot \gamma_k=-k+q >  t$ and $\leg_2^\prime \cdot \gamma_{k+1} = -(k+1) + q \geq t$ 
with equality only when $-k-1 =-q+t$. Let $\gamma$ be a simple
closed curve on $T^2$ representing the homology class given by the slope $s$. 
We may find  non-negative integers $a,b$ such that
$\gamma=a\gamma_k+b\gamma_{k+1}$.  Notice that if $a = 0$, then $b = 1$ and $k+1 < q-t$ since, by assumption,
$s > {-q+t}$.  One may now easily see that  $\leg_2^\prime \cdot \gamma> t$, proving
$\#(\leg_2^\prime \cap \Gamma_{T^\prime}) > 2t$.},
$\#(\leg_2^\prime \cap \Gamma_{T^\prime}) > 2t$, and thus
$tw(\leg_2^\prime,T^\prime) < -t$, a contradiction. 
Thus the dividing slope of $T^\prime$ is ${-q + t}$, and  
$\#(\leg_2^\prime \cap \Gamma_{T^\prime}) = 2 \ell t$ where $2\ell$ is the number of dividing curves.
Then since $tb(\leg_2^\prime) = -q - \ell t$ must equal $-q - t$, we know that 
 $T^\prime$ has exactly two dividing curves of slope ${-q+t}$.
Moreover $\leg_2^\prime$ minimally intersects the dividing set and hence can be made one of the ruling
curves on $T^\prime$. Now $T^\prime$ and $\partial N_1$ cobound a $T^2\times [0,1]$. Since the
contact structure on here is minimally twisting and the dividing slope on $T^2\times\{0\}$ and
$T^2\times\{1\}$ are the same there is a product structure on $T^2\times[0,1]$ such that the contact
structure is $[0,1]$-invariant.
Thus there is a contact isotopy that takes $T^\prime$ to $\partial N_1$, and so we can
assume that $\leg_2^\prime$ is a ruling curve on $\partial N_1$. Since any two ruling curves on
$\partial N_1$ are Legendrian isotopic, we see that $L$ is Legendrian isotopic to
$L'$.
 
A similar proof will work to show uniqueness when $tb(\leg_1) = -q$ except now $\leg_2'$ is a Legendrian
divide on $T'$.  If there are only two Legendrian divides on $T'$ then we can proceed as above:
there will be a $[0,1]$-invariant neighborhood $T^2 \times [0,1]$ between $T^\prime$ and $\partial N_1$, and
thus we can assume $\leg_2^\prime$ is also a Legendrian divide of $\partial N_1$.  
Then as follows from Lemma~\ref{prelagmodel}, there exists a
pre-Lagrangian torus containing all the Legendrian divides, and so in particular $\leg_2$, $\leg_2^\prime$, are among its leaves. 
If there are more than two Legendrian divides on $T'$, then observe that we can find another standard neighborhood $N_1'$ of $\leg_1$
such that $T'$ is inside $N_1'$. Thus working in a standard model of a Legendrian knot it is easy to
see that we can reduce the number of dividing curves on $T'$ without moving $\leg_2'$. This completes
the uniqueness statements for the unordered Legendrian $(2,-2q)$-torus links with maximal
Thurston-Bennequin invariant. \end{proof}

Now that we understand the Legendrian $(n, -nq)$-torus links with max $tb$, it
is natural to ask if all Legendrian $(n, -nq)$ will destabilize to one with max $tb$.  
We will see that this is true if $n =2$, however,  this is not true if $n \geq 3$; see Figure~\ref{fig-non_destab}.
{
\begin{lemma}\label{non-destab-twists} For $n \geq 3$ and $t \geq 1$, the $t$-twisted $n$-copies of a Legendrian unknot with $tb = -q + t$ are
Legendrian representatives of the $(n,-nq)$-torus link that 
do not have maximal \tb invariant yet do not destabilize.
\end{lemma}
}

\begin{proof}   Let $L$ be a $t$-twisted $n$-copy of a Legendrian unknot with $tb = -q + t$, where $t \geq 1$.  Then, since $n \geq 3$,
$tb(\Lambda_1) + \dots + tb(\Lambda_n) = (-q+t) + (n-1)(-q -t) < -nq$, showing the non-maximality of $tb$. If $L$ had a destabilization, then using the component
that can be destabilized together with another appropriately chosen component of $L$,  we could
construct a $(2,-2q)$-torus link with $tb(\Lambda_1) + tb(\Lambda_2) > -2q$, a contradiction to Proposition~\ref{maxtb}.
\end{proof}

\begin{figure}[ht]
\small
\begin{overpic} 
{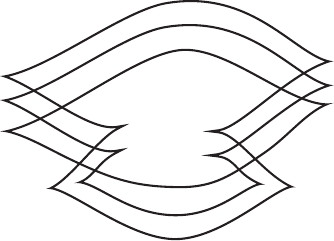}
\end{overpic}
        \caption{A Legendrian $(3,-6)$-torus link with non-maximal \tb invariant that cannot be
        destabilized.}                      
        \label{fig-non_destab}
\end{figure}

Our next lemma basically says that  as long as the \tb invariants of a Legendrian $(n, -nq)$-torus link do not 
agree with those of a  $t$-twisted $n$-copy of a Legendrian unknot with $tb = -q + t$, then we can destabilize.  

\begin{lemma}\label{destab} 
For $n \geq 2$,
 consider a Legendrian $(n, -nq)$-torus link $\leg$ with components
$$\leg_1, \leg_2, \dots, \leg_n$$
satisfying $tb(\leg_1) \geq tb(\leg_2) \geq \dots \geq tb(\leg_n)$.
\begin{enumerate}
\item If $tb(\leg_1) \leq -q$ and $tb(\leg_1) + \dots + tb(\leg_n) < -qn$, {then $\leg$  has a destabilization.} 
\item If $tb(\leg_1) = -q + t$, for  $1 \leq t \leq q-1$, and $tb(\leg_2)$, \dots, $tb(\leg_n)$ do
not all equal $-q - t$, then $\leg$ has a destabilization.  
\end{enumerate}
\end{lemma}

\begin{proof} 
We will proceed by 
considering the cases where $tb(\leg_1) = -q$, $tb(\leg_1) < -q$, and $tb(\leg_1) >-q$.
Below, $tw(\leg_i)$ will always refer to $tw(\leg_i,T)$, where $T$ is a standardly embedded Heegaard 
torus for $S^3$ on which $\leg_i$ sits as a $(1,-q)$-curve. 

In the case where $tb(\leg_1) = -q$, we have that $tw(\leg_1) = 0$ and $tw(\leg_n) =-k < 0$. Thus the link
$(\leg_1, \dots, \leg_n)$ lies as curves of slope $-q$ on a standardly embedded torus $T$ that can be
perturbed to be convex. Since $tw(\leg_1)=0$, we have $\#(\leg_1 \cap \Gamma_T) = 0$, and thus $\leg_1$ is
parallel to and disjoint from the dividing set $\Gamma_T$. The curve $\leg_n$ also has slope $-q$ and so could be topologically isotoped to be disjoint from $\Gamma_T$. However $\# (\leg_n \cap
\Gamma_T) = 2k > 0$, so there must be Whitney disks that cancels a pair of extraneous intersection points.
An innermost Whitney disk can be used to construct a bypass for $\leg_n$ on $T$ that is disjoint from
$\leg_1, \dots, \leg_{n-1}$.  Thus we may destabilize $\leg_n$.

In the case where $tb(\leg_1) < -q$, $tw(\leg_1), \dots, tw(\leg_n)$ are all negative so we can assume all the components
$\leg_1, \leg_2, \dots, \leg_n$ lie on a convex torus $T$ as slope $-q$ curves. If the dividing slope of $T$ is $-q$ then we can argue
as above to destabilize $\leg_n$ (and $\leg_1, \dots, \leg_{n-1}$), thus we assume that the dividing slope is $s\not= -q$. 
If $\leg_i$ does not intersect $\Gamma_T$ minimally, then there will be Whitney disks as
above and we can find a bypass to destabilize $\leg_i$.  So we can assume that $\leg_1, \dots, \leg_n$
intersect $\Gamma_T$ minimally, and hence we can isotop $T$, relative to $\leg_1\cup \dots \cup \leg_n$, so that
$\leg_1, \dots, \leg_n$  are ruling curves. The torus $T$ splits $S^3$ into two solid tori, one of which has
convex tori parallel to $T$ with any dividing slope in $(-\infty, s]$ and the other containing convex tori
parallel to $T$ with any dividing slope in $[s, 0)$. Thus we can find a standard convex torus $T'$
disjoint from $T$ with dividing slope $-q$. We can now take an annulus $A$ whose interior is
disjoint from $T\cup T'$ and has one boundary component a dividing curve on $T'$ and the other boundary
component being $\leg_n$. This annulus can be made convex. The dividing curves on $A$ will be disjoint from
$A\cap T'$ but non-trivially intersect $\leg_n$. Thus, by the Imbalance Principle, \cite[Proposition 3.11]{EtnyreHonda01b},  there will be boundary parallel dividing curves on $A$ that
can be used to construct bypasses for $\leg_n$ on $A$. Hence we can destabilize $\leg_n$ (without moving
$\leg_1\cup \dots \cup \leg_{n-1}$).

Lastly consider the case where $tb(\leg_1) > -q$; say $tb(\leg_1) = -q + t$, $1 \leq t \leq q-1$. So, in particular,
$tw(\leg_1) > 0$.  Recall that
by Lemma \ref{tb-observation}, we then know that $tw(\leg_i) < 0$, for $i \geq 2$.
Now $\leg_1$ does
not lie on a convex torus as a $-q$ curve. Consider a standard neighborhood of $\leg_1$. This
will be a solid torus $N_1$ with boundary a convex torus with two dividing curves of slope ${-q + t}$; by Giroux's Flexibility Theorem we can assume that the ruling curves are of slope
$-q$. Since $N_1$ can be made arbitrarily small, we can assume  the curves
$\leg_2, \dots, \leg_{n}$ lie on a convex torus $T$ in the complement of $N_1$, and thus the
slope of $\Gamma_T$ is greater than or equal to $-q+t$. As in the previous paragraph, we can
assume that $\leg_i$, $i \geq 2$, and $\Gamma_T$ intersect minimally (or we would already have a destabilization of
$\leg_i$), thus we can assume that $\leg_2, \dots, \leg_n$ are ruling curves on $T$. If the dividing slope is
${-q+t}$ then $\# (\leg_n\cap \Gamma_T)=2lt$, where $2l$ is the number of dividing curves on
$T$. Thus $tb(\leg_n)=-q - lt$. Since we know $tb(\leg_n)<-q-t$ we must have $l>1$. If we take an annulus
$A$ from $\partial N_1$ to $T$ with boundary on the ruling curves (and not $\leg_2\cup \dots \cup \leg_n$), then the
Imbalance Principle says there is a bypass for $T$ on $A$. Attaching this bypass to $T$ will 
reduce the number of dividing curves of $T$ so that $\leg_n$ (also $\leg_2, \dots \leg_{n-1}$) no longer intersects the dividing set minimally and we can
hence find a destabilization of $\leg_n$. We are left to consider the case when the dividing slope $s$ of
$T$ satisfies $s > {-q+t}$.  
An argument as in the proof of Lemma \ref{2maxtb} shows that $\# (\leg_n \cap \Gamma_T)> 2t$.
Given this we can
take an annulus $A$ between $T$ and $\partial N_1$ as above, except this time $A$ will have one
boundary component on $\leg_n$. The Imbalance Principle once again gives a bypass for $\leg_n$ on $A$, and
hence $\leg_n$ destabilizes. 
\end{proof}

\begin{lemma} \label{nondestab}
For $n \geq 3$, there exist Legendrian $(n, -nq)$-torus links with non-maximal \tb invariant that do not destabilize to one with maximal \tb
invariant. All such links have precisely one component $\leg_1$ with \tb invariant greater than $-q$ and
are $T^t(nU_{-q+t}^r)$ where $nU_{-q+t}^r$ is the Legendrian $n$-copy of a Legendrian unknot with
\tb invariant $-q+t$ and rotation number $r$.
\end{lemma}

\begin{proof}  Let $\leg$ be a Legendrian $(n, -nq)$-torus link with components 
$tb(\leg_1) \geq \dots \geq tb(\leg_n)$.  If $\leg$ does not have maximal $tb$ and it does not
destabilize to one with maximal $tb$, then by Lemma~\ref{destab}, we know $n \geq 3$,
$tb(\leg_1) =-q+t, 0<t<q$, and $tb(\leg_i)=-q-t$, for $i=2,\ldots, n$.

We can put $\leg'=\leg_2\cup \cdots \cup \leg_n$ on a torus $T$ that bounds a solid torus $N$ containing
$\leg_1$. Since the twisting of the $\leg_i$ with respect to $T$ is less than 0, we can make $T$ convex.
Arguing as in the proof of Lemma~\ref{2maxtb}, we can assume that the slope of the dividing curves on
$T$ is ${-q+t}$ and that there are just two dividing curves. Thus $\leg'$ sits as ruling curves
on the boundary of a standard neighborhood of $\leg_1$, and the isotopy class of $\leg$ is determined by
that of $\leg_1$.
\end{proof}

We now understand the set of links to which all Legendrian $(n, -nq)$-torus link destabilize.
As a last step, we need to understand how these non-destabilizable Legendrian $(n, -nq)$-torus
links are related under stabilization.  Recall we denote the $n$-copy of the unknot with $tb=-q$ and 
rotation number $r$ by $nU_{-q}^r$.
When listing the non-destabilizable Legendrian $(n, -nq)$-torus links we write $L_{-q}^r$
 for $nU_{-q}^r$. The other non-destabilizable links are Legendrian 
twists of $n$-copies of unknots with $tb > -q$: 
$T^t(nU_{-q+t}^r)$, where  $0< t < q$.  For shorter notation, we denote this
by $L_{-q+t}^r$. So the non-destabilizable Legendrian $(n, -nq)$-torus links are:
\[
\begin{aligned}
 &L_{-1}^{0} \\
 L_{-2}^{-1}, &\qquad L_{-2}^{1} \\
 &\hspace{.1in} \vdots \\
L_{-q}^{-q+1}, L_{-q}^{-q+3}, & \dots, L_{-q}^{q-3}, L_{-q}^{q-1} 
\end{aligned}
\]
We always label a component
with largest \tb invariant as $\leg_1$.
We then arbitrarily label the other components $\leg_2, \ldots, \leg_n$. A $\pm$-stabilization on the
$i^\text{th}$ component of a link $L$ is denoted by $S_{\pm,i}(L)$, and simultaneously stabilizing all components is denoted by $S_{\pm, all}(L)$.

\begin{lemma}\label{integer-simple}
Consider the non-destabilizable realizations of the $(n,-nq)$-torus link. We have the following
relations.
\[
 S_{+,all} (L_{-q}^j)=S_{-,all}(L_{-q}^{j+2}).
\]
If $-q< k\leq -1$, then
\[
S_{\pm,1}(L_{k}^j)= S_{\mp,2} \circ \cdots \circ S_{\mp,n}(L_{k-1}^{j\pm 1}).
\]
Thus as soon as the invariants of the components of two non-destabilizable Legendrian
$(n,-nq)$-torus links become the same under stabilization the links will become isotopic.
\end{lemma}

\begin{proof} The proof that $S_{+, all} (L_{-q}^j)=S_{-, all}(L_{-q}^{j+2})$ parallels the proof of Lemma~\ref{negstabtoiso}: the strategy is to 
show that both these links can be isotoped so
that they lie as ruling curves on a standard neighborhood of a Legendrian unknot with
$tb = -(q +1)$ and $r = j+1$.    We will  first argue that 
$S_{+, all}(L_{-q}^\ell)$   
 sits as ruling curves on the boundary of a standard neighborhood of a Legendrian unknot with $tb = -(q+1)$ and $r= \ell + 1$ as follows.  The link $L_{-q}^\ell$ lies as a subset of Legendrian divides of slope $-q$ on a convex torus $T$.
Consider the Heegard splitting of $S^3$ with respect to $T$: $S^3 = V_0 \cup_T V_1$.  
Inside $V_0$ there is a solid torus $V_{q+1}^+$ with two dividing curves of slope
$-(q+1)$ which is a standard neighborhood of a Legendrian unknot with $tb = -(q+1)$
and rotation number $\ell + 1$.  Let $T_{q+1}^+$ be the boundary of this solid torus.
By Giroux's Flexibility Theorem, we can assume the ruling slope of $T_{q+1}^+$ is $-q$.
Consider $n$ disjoing convex annuli $A_i$ of slope $-q$ between $T$ and
$T_{q+1}^+$ each having one boundary on a component of $L_{-q}^\ell$, which recall is a Legendrian divide of $T$, and the other
on a ruling curve  of $T_{q+1}^+$.  Let $K_i \subset L_{-q}^\ell \subset T$ and 
$K_i^+ \subset T_{q+1}^+$ denote the boundary components of $A_i$.
The dividing curves of $A_i$ do not intersect
$\partial A_i \cap T = K_i$ but will intersect $\partial A_i \cap T_{q+1}^+ = K_i^+$ twice.  
So on each $A_i$, there is one boundary parallel dividing curve separating a disk that must
be positive since $r(K_i^+) - r(K_i) = 1$.   Unfortunately we cannot get a bypass from this disk as we cannot Legendrian realize a bypass on $A_i$ since the dividing set is connected. However, we can isotop the $A_i$ so that its boundary component in $T$ intersects the dividing curve twice. Now the $A_i$ will have either two dividing curves running from one boundary component to the other, or two boundary parallel dividing curves, one on each boundary. In the former case we can foliate $A_i$ by ``ruling curves" and use those to isotop $K_i^+$ to a curve on $T$ and on $T$ we can use a Whitney disk for $K_i^+$ and the dividing set to destabilize $K_i^+$. In the latter case we can find a bypass on $A_i$ to destabilize $K_i^+$. In particular we get 
a link $\widetilde L$ that lies between $T_{q+1}$ and $T$.
Notice that each component of $\widetilde L$ has $tb = -q$ and $ r = \ell$.  
Put $\widetilde L$ on a convex torus $\widetilde T$ parallel to
but disjoint from $T$.  We know the dividing set of $\widetilde T$ will have slope
$-q$.  
Thus $\widetilde L$ is a subset of the Legendrian divides of $\widetilde T$ and thus as argued in the proof of Lemma~\ref{maxntb} we see that $\widetilde L$ can be taken to be leaves in the characteristic foliation of a pre-Lagrangian torus. Hence, by Lemma~\ref{ncopyalt}, we see $\widetilde L$ is
the $n$-copy of a $tb=-q$ unknot. Moreover, since the rotation numbers of the components of $\widetilde L$ 
agree with those of $L$ we know that $\widetilde L$ and $L$ are both $n$-copies of the same Legendrian unknot
and hence are isotopic.  
A similar argument shows
$S_{-, all} (L_{-q}^{\ell})$ sits as ruling curves on the
boundary of a standard neighborhood of a Legendrian unknot with $tb = -(q+1)$ and $r = \ell -1$.
Here instead of $V_{q+1}^+$, we consider $V_{q+1}^-$ which is a standard neighborhood
of a Legendrian unknot with $tb = -(q+1)$ and $r = \ell -1$.  The corresponding annuli will
now have one boundary parallel arc separating off a negative disk.  Combining these, we
see that both $S_{+, all} (L_{-q}^j)$ and $S_{-,all}(L_{-q}^{j+2})$ sit as ruling curves on the boundary of standard neighborhoods
of an unknot with $tb = -(q+1)$ and $r = j - 1$.  Since all such neighborhoods are
isotopic, we can assume  $S_{+, all} (L_{-q}^j)$ and $S_{-,all}(L_{-q}^{j+2})$ sit as ruling curves on the same torus, and thus they must be isotopic.

To see that if $-q< k\leq -1$, then
$
S_{\pm,1}(L_{k}^j)= S_{\mp,2} \circ \cdots \circ S_{\mp,n}(L_{k-1}^{j\pm 1})
$
notice that the proof of Lemma~\ref{destab} shows that if $\leg_1$ in $L_k^j$ is stabilized, then
one may destabilize the components with the most negative Thurston-Bennequin invariant
which, in this case, will be $\leg_2, \dots, \leg_n$. 
\end{proof}

 %%%%%%%%%%%%%%%%%%%%%%%%%%%%%
\section{Ordered Classification}\label{oclassification}
%%%%%%%%%%%%%%%%%%%%%%%%%%%%%
 
 Now that we have established the unordered classification of all Legendrian torus links, we move on to the ordered classification.  The positive
 torus links and stabilized negative torus links will have a great deal of flexibility in the permutations that are allowed, while the negative torus
links with maximum \tb invariant will have a great deal of rigidity. 
 
 The unordered classification of positive torus links is given in Theorem~\ref{thm:p-unorder-class}.  For positive torus
links, any $tb$ and $r$ invariant-preserving permutation of the components is possible:
 
\begin{theorem}  \label{oposlink} For $q \geq p \geq 1$ and $\gcd(p,q) = 1$,
consider an ordered, oriented Legendrian $(np,+nq)$-torus link $L=(\Lambda_1,\ldots, \Lambda_n)$.
Any permutation of the components of $L$ preserving
the  \tb and rotation number invariants can be achieved by a Legendrian isotopy. 
\end{theorem}

\begin{proof}
We prove the statement when $L$ is a  $(np,+nq)$-torus link with max $tb$; the general case follows from Lemma~\ref{posdestab}.
As in the proof of Lemma~\ref{uniquemaxpos}, we know $L$ sits on a convex torus $T$ as ruling curves of slope $q/p$. There is a
neighborhood $N=T^2\times [-1,1]$ of $T$ such that $T^2\times\{0\}=T$ and the contact structure is invariant in the $[-1,1]$
direction, \cite{EtnyreHonda01b, Honda00a}. So each $T^2\times\{pt\}$ is foliated by ruling curves of slope $q/p$. We can isotop each component of $L$
to a different torus in $N$, and  then further isotop the components on the different levels so that their order is permuted by
any preassigned permutation. Finally we isotop the permuted components back to $T$.
\end{proof}

\begin{remark}  
It is possible to do the permutations in the front diagram as was shown by the first author \cite{Dalton08}.  
\end{remark}

Next we move on to study the ordered classification of Legendrian negative torus links. 
 Theorem~\ref{uoclass}
gives the unordered classification of
Legendrian $(n, -nq)$-torus links with $q \geq 1$, while Theorem~\ref{oclass} gives the
unordered classification of  Legendrian $(np, -nq)$ 
torus links with $q > p \geq 2$.  We now consider the
ordered classification.  
 
We will see that there is rigidity to the allowable permutations among the set of components having the maximal of  $tb = -pq$; recall that 
all these components with $tb = -pq$ must have the same rotation number.  For the components with maximal $tb$, a pre-Lagrangian torus determines a cyclic ordering
of the components.
 
\begin{definition} \label{neg-pre-lag}  {Let $L = (\Lambda_1, \dots, \Lambda_n)$ be a Legendrian $(np, -nq)$-torus link such that $tb(\Lambda_{i}) = -pq$, for all $i$. Then $L$ is the union of leaves of a pre-Lagragian torus $T_0$, and $L$ can also be seen as the Legendrian divides of a convex torus $T$; if $p > 1$, $T, T_0$ are contained in a basic slice, and if $p = 1$, $T, T_0$ are  contained in a universally tight union of two basic slices (see Lemma~\ref{integral-union-basic}).  
Then $T_0$ (and equivalently, by Lemma~\ref{prelagmodel} and Remark~\ref{prelagmodel-extended}}, a complementary annulus for $T$, as defined in Definition~\ref{defn:complementary}) gives a \dfn{cyclic ordering} of the components of $L$.  We will always assume that the $\Lambda_i$ are numbered according to this ordering.  This ordering is well-defined up to cyclic permutation.
\end{definition}

\begin{lemma}\label{cyclic-perms} Let $L = (\Lambda_1, \dots, \Lambda_n)$ be a Legendrian $(np, -nq)$-torus link such that $tb(\Lambda_{i}) = -pq$, for all $i$. 
Then it is possible to do a cyclic permutation of the components of $L$ via a Legendrian isotopy.
\end{lemma}

\begin{proof} As described in Definition~\ref{neg-pre-lag}, $L$ lies among the leaves of a pre-Lagrangian torus $T_0$. One may cyclically permute the components of $L$ through the leaves of the pre-Lagrangian $T_0$.
\end{proof}
 
As seen in Lemma~\ref{negmaxtb} and Proposition~\ref{maxntb},  when $p > 1$, these  max $tb$ representatives of the $(np, -nq)$-torus links are the $n$-copies of a 
Legendrian  $(p,-q)$-torus knot  $\leg$ with  max $tb$, and when $p = 1$,   these  max $tb$ representatives of the $(n, -nq)$-torus links are the $n$-copies  of a Legendrian unknot  $\leg$ with 
  $tb = -q$.  In the front projections, these $n$-copies can be seen as slight shifts of $\Lambda$ in the $z$-direction, and 
  the cyclic ordering corresponds to increasing $z$-coordinate, with the uppermost component circling back to become the lowest.

\begin{theorem}\label{negperms}
For $q \geq p \geq 1$ and $\gcd(p,q) = 1$, let $L = (\Lambda_1, \dots, \Lambda_n)$ be an oriented, ordered Legendrian $(np, -nq)$-torus link. Let $I_{1}$ be the subset of $\{1,\ldots, n\}$ 
containing the indices such that $tb(\Lambda_i)= -pq$, and let $I_2$ be its complement. 
Then there is a Legendrian isotopy from $(\Lambda_1, \dots, \Lambda_n)$ to 
$(\Lambda_{\sigma(1)}, \dots, \Lambda_{\sigma(n)})$ where $\sigma$ is an element of the symmetric group $S_n$ if and only if 
\begin{enumerate}
\item $\sigma$ preserves the partition $I_1\cup I_2$,
\item $\sigma$ restricted to $I_1$ is a cyclic permutation, and
\item for all $i\in I_2$, $\Lambda_i$ and $\Lambda_{\sigma(i)}$ have the same \tb  and rotation number invariants.
\end{enumerate}
\end{theorem}

\begin{figure}[ht]
\begin{overpic} 
{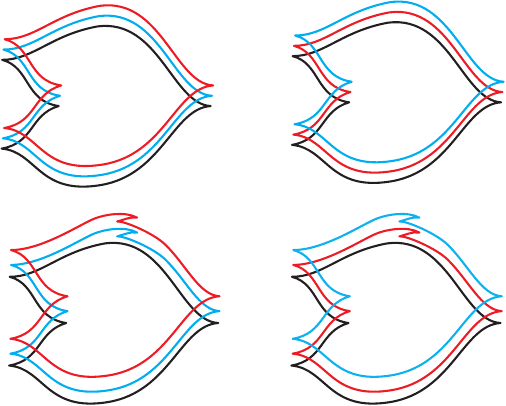}
\Large
\put(120,145){$\not=$}
\put(120,40){$=$}
\end{overpic}
        \caption{A noncylic permutation of the Legendrian $3(1,-2)$-torus link when all components have $tb = -pq$ is not possible, however once components are stabilized, non-cyclic permutations are possible.}                    
                \label{fig-$n$-copy_perm}
\end{figure}

Before proving Theorem~\ref{negperms}, we note that it is the last step needed to complete the 
proof of the ordered classification of torus links, Theorem~\ref{orderedclass}.

\begin{proof}[Proof of Theorem~\ref{orderedclass}]
The result immediately follows from the unordered classification,  
\ref{thm:p-unorder-class}, 
\ref{oclass}, and \ref{uoclass}, and Theorem~\ref{oposlink} and~\ref{negperms}.
\end{proof}

We begin the proof of Theorem~\ref{negperms} by observing that permutations mentioned are possible.

\begin{lemma} \label{stab-perm} 
With the notation from Theorem~\ref{negperms}, the permutations of the components of $L$ listed there can be achieved through a Legendrian isotopy. 
\end{lemma}

\begin{proof}
 Let $I_{1}$ be the subset of $\{1,\ldots, n\}$ 
containing the indices such that $tb(\Lambda_i)= -pq$, $I_2$ the subset with $tb(\Lambda_i)=-pq-1$ that have been positively stabilized once, and $I_3$ the subset with $tb(\Lambda_i)=-pq-1$ that have been negatively stabilized once.  We will show the result holds in the case where $I_1\cup I_2\cup I_3= \{1,\ldots, n\}$;  all the other cases will be stabilizations of this case, and hence the result will also follow in these other cases. 

We begin by assuming that $-q/p$ is not an integer, i.e., $p \neq 1$. Suppose that $-q/p$ arises in the Farey graph from the sum of $s$ and $s'$; observe that 
$s, s' \in \mathbb Q \cap [-\infty, -1]$.
Since there is also an edge in the graph between $s$ and $s'$, there will be a basic slice $B$ with dividing slopes $s$ and $s'$ that is 
embedded in $S^3$ as a neighborhood of a Heegaard torus. We take the ruling slopes on $\partial B$ to be $-q/p$. Inside $B$ is a pre-Lagrangian 
torus $T$ whose characteristic foliation has slope $-q/p$. Let $\Lambda_1$ be a leaf of the foliation of $T$. We claim that the ruling curves on one boundary component of $B$ are positive stabilizations of $
\Lambda_1$, and the ruling curves on the other boundary component are negative stabilizations of $\Lambda_1$. To see this, consider a convex annulus having one boundary component on $\Lambda_{1}$ and the other on a ruling curve of $\partial B$.  Since there is an edge in the Farey graph between $-q/p$ and $s$ (and between $-q/p$ and $s'$), we know that the ruling curves intersect each dividing curve on $\partial B$ exactly once.  Thus the \tb invariant of the ruling curve is one less that that of $\Lambda_1$. Moreover, an annulus between a ruling curve and $\Lambda_1$ has a singe boundary parallel dividing curve, thus the rotation number of the ruling curve differs by one from $\Lambda_1$ (the difference in rotation numbers can be computed in terms of the positive and negative regions of the convex annulus between them, see the proof of \cite[Section~3.2]{EtnyreHonda01b}). 
From the classification of Legendrian torus knots (or Lemma~\ref{negtorusstab}) we see that the ruling curves are stabilizations of $\Lambda_1$. To see that the curves on different boundary 
components of $B$ are different stabilizations, we note that if $A$ is an annulus in $B$ connecting the ruling curves then the relative Euler class of $B$ evaluated on $A$ is $\pm 2$; (this follows 
from the formula for the relative Euler class in \cite{Honda00a}, but to see it easily one may change coordinates so that $s=-\infty, -q/p=-1,$ and $s'=0$. Then the relative Euler class is Poincare dual to the $
\pm [1,1]$ curve on $T^2$ and hence evaluates on the $[1,-1]$ annulus as $\pm 2$). The difference in the rotation numbers of the boundary of this annulus agrees with this relative Euler class, 
\cite[Section~3.2]{EtnyreHonda01b}, and so we see that the rotation numbers of the ruling curves on the different boundary components of $B$ differ by $2$,  and thus they are different stabilizations. 
Now $L$ can be realized with the components indexed by $I_1$ as leaves in the foliation of $T$ and the components indexed by $I_2$ and $I_3$ as ruling curves on the boundary components of $B$. The components on $T$ can be cyclically permuted and, arguing as in the proof of Lemma~\ref{oposlink}, we can  arbitrarily permute the components on a common boundary component of $B$. 

If $p = 1$, which implies $-q/p$  is an integer, then for the $(n,-nq)$-torus links with non-maximal \tb invariant that do not destabilize, the result is clear from the model constructed in the proof of Lemma~\ref{destab}. For the $(n,-nq)$-torus links with max $tb$, the proof follows the ideas in the $p \neq 1$ case except we cannot take $s, s'$ as above since one of $s$ or $s'$ is $-\infty$, and so 
the desired basic slice does not exist in $S^3$. However, by Lemma~\ref{integral-union-basic}, we know that our link lies on a torus in the universally tight union of two basic slices. 
 Hence the pre-Lagrangian torus $T$ with foliation of slope $-q$ still exists \cite{Honda00a}, and the argument now proceeds exactly as above.  
 \end{proof}

We must now see that only cyclic permutations will produce
equivalent Legendrian links negative torus links  with maximal \tb invariant.  The first work along these lines was done for $n$-copies of any Legendrian unknot
by Mishachev using the Chekanov-Eliashberg DGA.

\begin{theorem}[Mishachev 2002, \cite{Mishachev03}]\label{Mishachev}
For any $q \geq 1$, let $(\Lambda_1, \dots, \Lambda_n)$ be the ordered max $tb$ Legendrian representative of the $(n, -nq)$-torus link obtained as
the $n$-copy of a Legendrian unknot with $tb = -q$. 
Then there is a Legendrian isotopy from $(\Lambda_1, \dots, \Lambda_n)$ to $(\Lambda_{\sigma(1)}, \dots, \Lambda_{\sigma(n)})$ where $\sigma \in S_n$ only if 
$\sigma$ is a cyclic permutation.  
 \end{theorem}

\begin{remark}
In Mishachev's DGA approach, he worked with the unique augmentation $\epsilon$ of $\Lambda$, where $\Lambda$ 
is either the Legendrian unknot with $tb = -1$ or the $2$-copy of a Legendrian unknot with $tb = -q$, for $q \geq 2$.
To different orderings, one can define ``characteristic algebras'' $CH_{123}(\epsilon)$ and $CH_{132}(\epsilon)$
 associated with non-cyclic permutations of $3$-copies of $\Lambda$.  In fact,  
  one of these characteristic algebras has zero divisors while the other does not, which shows that
 non-cyclic permutations are never possible.  Extending Mishachev's DGA approach to Legendrian $(np, -nq)$-torus links for $p \geq 2$ has difficulties.  
 There is again a unique augmentation of the $(2, -q)$-torus
 knot, for all $q > 2$, but  the characteristic algebras of interest both have $0$-divisors. L. Ng  
 observed that we can handle this $(2, -q)$ case by analyzing product structures.  When $p \geq 3$, one needs to double the knot to get an augmentation,
 and one gets {\it many} augmentations.   For example, when trying to apply this approach to study non-cyclic permutations of $(3, -4)$, one needs to find augmentations of
 $\leg = 2(3, -4) = (6, -8)$.  There are now $3$ augmentations of $\Lambda$, which translates into $27$ ``minimal" augmentations of the $3$-copy of 
 $\Lambda$ arising from lifts of augmentations of $\Lambda$.  
The authors made a number of attempts to extend the DGA approach, including looking at matrix-valued augmentations, but with little success. 
We give a convex surface argument that works in all cases. 
\end{remark}

   \begin{theorem} \label{non-cyclic-perm-gen} For $q \geq p \geq 1$ and $\gcd(p,q) = 1$,  
let $L$ be a Legendrian $(np, -nq)$-torus link such that all the components have $tb=-pq$. Then  no non-cyclic permutation can be achieved via a Legendrian isotopy.
\end{theorem}

With this theorem we can now complete the proof of Theorem~\ref{negperms}.
\begin{proof}[Proof of Theorem~\ref{negperms}]
The existence of the claimed permutations of the Legendrian link follows from Lemmas~\ref{cyclic-perms},  \ref{stab-perm}, and the restrictions on the permutations follows from Theorem~\ref{non-cyclic-perm-gen}.
\end{proof}
 
 We begin the proof of Theorem~\ref{non-cyclic-perm-gen} by observing limitations on Legendrian isotopies in a basic slice, and then reduce the above theorem to this observation.

 \subsection{Forbidding Non-cyclic Permutations in Basic Slices} \label{ssec:non-cyclic-slice}
 
\begin{theorem}\label{thickenedtori}
Let $(T^2\times[0,1], \xi)$ be a basic slice, and suppose
$T$ is a boundary-parallel, pre-Lagrangian torus in this basic slice.
Let $L=(\Lambda_1,\ldots, \Lambda_n)$ be a Legendrian link consisting of leaves in the characteristic foliation of $T$ ordered as they appear on $T$. 
Then a non-cyclic permutation of the components of $L$ cannot be achieved by a Legendrian isotopy.  
\end{theorem}

\begin{proof}
Let $(T^2\times[0,1], \xi)$ be a  basic slice with dividing curves on $T^2\times \{i\}$ having slope $s_i$, for $i = 0, 1$,
and $T \subset T^2 \times [0,1]$ be a boundary-parallel  pre-Lagrangian torus.  Since we are in a basic slice,  we know the characteristic foliation of $T$ will have slope $s\in (s_0,s_1)$.  
 
For the reader's convenience, we will first outline the steps of the proof.

\begin{step}\label{step 1} If there is a contact isotopy $\phi_t$ of $L$ in $T^2 \times [0,1]$ such that $\phi_1(L) = L$  and $\phi_1$ realizes a non-cyclic permutation of its components, then there
is a 3-component link $(\Lambda_1, \Lambda_2, \Lambda_3)$ consisting of leaves of the characteristic foliation of $T$, a neighborhood $N$ of $\Lambda_1$,  and a contact isotopy $\psi_t$ of $(T^2\times [0,1], \xi)$ such that   
 $\psi_{t} = id$ on $N$ and in a neighborhood of the boundary, $\psi_1(\Lambda_2) = \Lambda_3$, and $\psi_1(\Lambda_3) = \Lambda_2$.  
\end{step}
 
\begin{step}   \label{step 2}
There is a contact embedding of $(T^2\times[0,1], \xi)$, which is a basic slice with boundary slopes $s_0$ and $s_1$ that contains the pre-Lagrangian torus $T$ of slope $s$, into $(T^2\times [-1,2], \xi')$, where $\xi'$ is tight and each component of 
the boundary of $T^2\times [-1,2]$  is convex with two dividing curves of slope $s$. 
The contact structure $\xi'$ can be chosen such that the $S^1$-action given by rotating in the $s$-direction preserves $\xi'$. We extend the contact isotopy $\psi_t$ 
defined on $(T^2\times[0,1], \xi)$ in Step 1  by the identity to obtain an isotopy $\psi_{t}$ defined on $(T^2\times [-1,2], \xi')$.  We can write $(T^2\times [-1,2], \xi')= (\Sigma \times S^{1}, \xi')$, where $\Sigma = S^1 \times [-1, 2]$ and  $\Sigma \times \{ \theta \}$ is a convex annulus.
\end{step}

The $S^{1}$-invariance allows us to represent important objects in our $3$-manifold on the annular surface $\Sigma$.

\begin{step}  \label{step 3}
There are tubular neighborhoods $N_i$ of $\Lambda_i$, with $N_1 \subset N$, that are invariant under the $S^1$-action  from Step 2.
Letting $X_1=\overline{(T^2\times[-1,2])-N_1}$, $X_1$ is diffeomorphic to $\Sigma_{1} \times S^1$, where $\Sigma_{1}$ is an annulus with a disk removed, and
the contact isotopy $\psi_t$ from Step 2 restricts to define a contact isotopy $\psi_{t}^{1}$ of $ X_1$.
Letting $X_{123} =\overline{(T^2\times[-1,2])-(N_1 \cup N_{2} \cup N_{3})}$, $X_{123}$ is diffeomorphic to $\Sigma_{123} \times S^1$, where $\Sigma_{123}$ is $\Sigma_{1}$ with two disks removed,
and 
the contact diffeomorphism $\psi_{t}|_{t = 1}$ from Step 2 gives rise  to a contact diffeomorphism of $\psi^{123}$ of $ X_{123}$.
The neighborhoods $N_{i}$ can be chosen such that $X_{123}$ has  boundary consisting of convex tori with dividing curves of slope $\infty$, meaning parallel to the $S^{1}$-action,
and horizontal Legendrian rulings, meaning that  the boundary of $\Sigma_{123} \times \{\theta\}$ is a union of Legendrian ruling curves.
\end{step}

\begin{step}  \label{step 4}
For all $\theta$, $\Sigma_{123} \times \{ \theta \}$ is convex with Legendrian boundary.
 From  the convex surface $\psi^{123}(\Sigma_{123} \times \{ \theta \})$, we can construct a convex surface $\Sigma'$ whose boundary agrees with 
 the boundary of $\Sigma_{123} \times \{ \theta \}$ yet  the dividing curves of $ \Sigma_{123} \times \{ \theta \}$ and the 
  the dividing curves of $\Sigma'$  connect the 3 boundary components  coming from the neighborhoods of the $\Lambda_i$ in topologically different ways.
 Since  
 there is an isotopy rel boundary taking $\Sigma'$  to $\Sigma_{123} \times \{ \theta \}$,
    we get a contradiction to the following ``minimality'' proposition due to Honda.
  
  \begin{proposition} \cite[Proposition~4.4]{Honda00b} \label{KoS1invariant} Suppose $\Sigma$ is compact with boundary.
  Consider an $S^1$-invariant tight contact structure on 
  $\Sigma \times S^{1}$ where the boundary of $\Sigma \times S^{1}$
  consists of convex tori with dividing curves of slope $\infty$ (meaning parallel to the $S^{1}$-action) and horizontal Legendrian rulings (meaning  for all $\theta$,
  $\partial \Sigma \times \{ \theta \}$ is a Legendrian ruling curve).  Then $\Sigma \times \{\theta\}$ minimizes dividing curves in the following sense:
  if $\Sigma' \subset \Sigma \times S^{1}$ is a convex surface  that is isotopic rel boundary to $\Sigma \times  \{ \theta \}$, 
   then there exists an isotopy  $\phi_{t}$ of 
  $\Sigma'$  with $\phi_{1}(\Sigma') = \Sigma \times \{ \theta \}$ and $\phi_{t}|_{\partial \Sigma'} = id$  such that $\Gamma_{\Sigma \times \{ \theta\}} \subset \phi_{1}(\Gamma_{\Sigma'}) $.  
   {In particular,  if $\partial_1$ and $\partial_2$ are two components of $\partial(\Sigma \times \{\theta\})$ and there  there is a dividing curve  on $\Sigma \times \{ \theta \}$ that connects
   $p_1 \in \partial_1$ and $p_2 \in \partial_2$, then there will be a dividing curve on $\Sigma'$ that connects $p_1$ and $p_2$.} 
   \end{proposition}
\end{step}

Now we begin to verify the claims in Step 1.
 Suppose there is a contact isotopy $\phi_t$ of $L = (\Lambda_1, \dots, \Lambda_n)$ such that $\phi_1(L) = L$ and $\phi_1$ realizes a non-cyclic permutation of its components.
 We can do further cyclic isotopies by sliding $L$ along the leaves of the pre-Lagrangian torus $T$. By such an isotopy we can assume that $\phi_1(\Lambda_1) = \Lambda_1$. If we now label two components of $L$ that have their order interchanged (but the components are not necessarily sent to each other) by $\Lambda_2$ and $\Lambda_3$.  By a further isotopy  along $T$ we can assume that $\Lambda_2$ and $\Lambda_3$ are interchanged.  
Disregarding the other components we see that if there is an isotopy of $L$ realizing a non-cyclic permutation of its components, then there is a Legendrian link, $(\Lambda_1, \Lambda_2, \Lambda_3)$ on $T$, 
 and an isotopy $\phi_t$ such that $\phi_1(\Lambda_1) = \Lambda_1$, $\phi_1(\Lambda_2) = \Lambda_3$, and $\phi_1(\Lambda_3) = \Lambda_2$.    
{In our labeling, we are assuming that with respect to the  pre-Lagrangian $T$, one encounters $\Lambda_{1} = \phi_{1}(\Lambda_{1})$, then $\Lambda_{2} = \phi_{1}(\Lambda_{3})$, and then $\Lambda_{3} = \phi_{1}(\Lambda_{2})$.}

Consider the pre-Lagrangian torus $T' =\phi_1(T)$. Notice that $T \cap T'$ contains $L$.
The following lemma about pre-Lagrangian annuli will allow us to modify $\phi_{t}$ to a contact isotopy  that is the identity on a neighborhood of $\Lambda_{1}$.

\begin{lemma}\label{close}
Let $A$ and $A'$ be two pre-Lagrangian annuli whose characteristic foliations consist of circles. 
Assume that $\Lambda \subset A \cap A'$ is a Legendrian circle on both $A$ and $A'$. Then $A'$ may be isotoped through pre-Lagrangian annuli so that a neighborhood $U$ of $\Lambda$ in $A'$ is a subset of $A$. 
Moreover, $\Lambda$ and $\partial A'$ can be fixed throughout the isotopy.    
\end{lemma}

\begin{proof}
We construct a model for a neighborhood of $A.$ Let $A$ be the $xz$-plane in $\R^3/\sim$, where  $(x,y,z)\sim (x+1, y, z)$,
with contact structure  $\xi=\ker (dz-y\, dx)$; {$\leg$ is modeled by  the $x$-axis $S^1 \times \{0\}$. } We know that the annulus $A'$ is transverse to the $y$-direction along the $x$-axis, since a tangency would give a singularity in the characteristic foliation of $A'$.  Thus we can write $A'$ near $S^1\times\{0\}$ as a graph over $A$:  there
is some function $f(x,z)$ so that a neighborhood $U'$ of $S^1\times\{0\}$ in $A'$ is $\{(x,f(x,z),z)\}.$ The front projection of the foliation of $U'$ onto $A$
will give a foliation of a neighborhood of the $x$-axis by curves. As will be explained in the next paragraph, these curves can be straightened out near the $x$-axis, and this
straightening gives the isotopy from $U'$ to a subset of $A$.

To make this precise, notice that the projection of the foliation of $U'$ to the $xz$-plane can be parameterized by $F:S^1\times [-\epsilon, \epsilon]\to S^1\times \R$, where $F(\theta, t)= (\theta, t+f_t(\theta))$ 
 and the 
constant $t$ curves are the front projections of leaves of the characteristic foliation of $A'$. Notice that $f_0(\theta) =0$, and   near $t=0$ we have $\frac{\partial f_t(\theta)}{\partial t}$ is small since here the front projection of the corresponding leaves are almost horizontal: we can assume there exists  $\varepsilon$ and $B$ such that
$$\left| \frac{\partial f_t(\theta)}{\partial t} \right| \leq B < 1, \quad \forall t \in [-\epsilon, \epsilon], \quad \forall \theta \in S^1.$$

To construct an isotopy of pre-Lagrangian annuli, consider an increasing function $g: [-\epsilon,\epsilon] \to [-\epsilon, \epsilon]$ such that $g = 0$ on a neighborhood of $0$
and $g(t) = t$ near $\pm \epsilon$.   Then set $F_g(\theta, t)= (\theta, t+f_{g(t)}(\theta))$.  If $F_g$ is an embedding, then by considering each constant $t$ curve as the front projection of a Legendrian knot,
 we can lift the image of $F_g$ to a pre-Lagrangian annulus $U'$ that will coincide with $A$ along the center circle of this annulus and with $A'$ near its boundary.  Towards seeing that $F_g$
 is an embedding, we first claim that it is possible to choose $\epsilon$ and $g$ so that $F_g$ is an immersion.
Observe that $\det DF_{g} = (1 + \frac{\partial f_{g(t)}(\theta)}{\partial t} )$, and thus we want to guarantee that  with an appropriate choice of $g$, $\frac{\partial f_{g(t)} (\theta)}{\partial t} > -1$.
  Choose $C > 1$ such that $|BC| < 1$; it is possible to choose $g$ such that  $|g'(t)| < C$, for all $t \in [-\epsilon, \epsilon]$.
   By the chain rule, we then see that  $|\frac{\partial f_{g(t)} (\theta)} {\partial t} | < 1$,
 which will guarantee that $F_g$ is an immersion.  
 Moreover, we see that $F_g$ is an embedding by arguing $F_{g}$ is injective as follows. 
 First observe that if 
 $F_g(\theta_1,t_1)=F_g(\theta_2,t_2)$, then $\theta_1=\theta_2$. 
 Now by our choice of $g$, for a fixed $\theta$, the map $t\mapsto t + f_{g(t)}(\theta)$ has positive derivative  and thus is injective. 
  In addition, there is a family of such $g$ starting with the identity map and ending with a pre-chosen $g$, thus we get an isotopy through pre-Lagrangian annuli.
 \end{proof}

We can use Lemma~\ref{close} to complete Step 1. 

\begin{corollary} \label{fix1} If there is a contact isotopy $\phi_t$ of $L$ in the basic slice $(T^2 \times [0,1], \xi)$ such that $\phi_1(L) = L$ and $\phi_1$ realizes a non-cyclic permutation of its components, then there
is a 3-component link $(\Lambda_1, \Lambda_2, \Lambda_3)$, a neighborhood $N$ of $\Lambda_1$,  and a contact isotopy $\psi_t$ of $(T^2\times [0,1], \xi)$ such that   
 $\psi_{t} = id$ on $N$ and in a neighborhood of the boundary, $\psi_1(\Lambda_2) = \Lambda_3$, and $\psi_1(\Lambda_3) = \Lambda_2$.  
\end{corollary}

\begin{proof}
Let $A$ be an annular neighborhood of $\Lambda_1$ on the pre-Lagrangian torus $T$ and let $A'$  be an annular neighborhood of $\Lambda_1$ on $T' = \psi_1(T)$.
By Lemma~\ref{close}, we can isotop $T'$ to a pre-Lagrangian torus $T''$ that agrees with $T$ in a neighborhood $U$ of $\Lambda_1$. We can then use the leaves of the characteristic foliation on $T$ to isotop $\Lambda_2 \cup \Lambda_3$ into the neighborhood $U$ (without moving a neighborhood  of $\Lambda_1$), and then using $T''$ we can further isotop them back to $\Lambda_3 \cup \Lambda_2$ (interchanging order). We can now extend this Legendrian isotopy to a global contact isotopy $\psi_t$  as desired.  \end{proof}

Now we move onto Step 2, the first part of which is accomplished through the following lemma.

\begin{lemma} \label{basic-slice-extension} There exists an embedding of our basic slice $(T^2\times[0,1], \xi)$, which has boundary slopes $s_0$ and $s_1$ and contains the pre-Lagrangian torus $T$ of slope $s$, 
into $(T^2\times [-1,2], \xi')$ where $\xi'$ is tight and each of the  boundary components of $T^2\times [-1,2]$  is convex with two dividing curves of slope $s$.  
The contact structure $\xi'$ can be chosen such that the $S^1$-action given by rotating in the slope $s$-direction preserves $\xi'$. 
\end{lemma}

\begin{proof} Our strategy will be to build a model for $(T^2\times [-1,2], \xi')$ that contains a basic slice; {our result will then follow from the fact that basic slices are unique up to orientation.}  

Begin by writing $T^2=S^1\times S^1$ so that $\{p\}\times S^1$ is the curve of slope $s$. That is, we change coordinates on $T^2$ so that  the pre-Lagrangian is foliated by lines of slope $s=\infty$; let
$v_{i}$ denote the boundary slopes $s_{i}$ of the basic slice in these new coordinates.
We denote the angular coordinates on the $S^1$ factors by $\theta_1$ and $\theta_2$.   
In these coordinates, consider $T^2\times\R$ with the contact structure $\ker(-\cos t \, d\theta_1+ \sin t\, d\theta_2)$.  The torus $T^{2}\times \{0\}$ is pre-Lagrangian with slope $\infty$, as are the tori $T^2\times \{\pm \pi\}$. This contact structure is well-known to be tight\footnote{For example, this is the contact structure on a $\Z$-fold cover of the standard tight contact structure on $T^3$, which is universally tight.}, and the $S^1$-action given by rotating in the $\theta_2$-direction preserves the contact structure.   
As argued in the proof of Lemma~\ref{prelagmodel},
one may perturb the tori $T^2\times \{\pm \pi\}$ such that  each are convex, have two dividing curves of slope $\infty$, and have horizontal ruling curves. 
 The region $R$ between these convex tori is invariant under the $S^1$-action, and  
  $R$ is further divided into two regions $R_-$ and $R_+$ by $T^{2} \times\{0\}$, where $R_-$ contains $t$ coordinates with negative values. 
 From our local model, we see that by choosing a sufficiently small perturbation of the $T^2\times \{\pm \pi\}$,  we can  be guaranteed to
find a pre-Lagrangian torus of slope $v_0$ in $R_-$ and a pre-Lagrangian
 of slope $v_1$ in $R_+$.  Perturb each of these tori to be convex with two dividing curves, and let $B$ be the region between them. Then $T^{2}\times\{0\}$ is contained in $B$, and $B$ is a basic slice. 
 So we have built a model for our original basic slice $T^2\times [0,1]$.  Moreover the complementary regions $R\setminus B$ are the claimed thickened tori that can be added to $T^2\times [0,1]$ to create $(T^2\times [-1,2],\xi')$, as desired. 
\end{proof}

\begin{remark}
For future arguments, it will be convenient to think of $(T^{2} \times [-1, 2], \xi')$ as $(\Sigma \times S^{1}, \xi')$, where $\Sigma = S^1 \times [-1, 2]$ is an annulus, and $\xi'$ is
invariant in the $S^1$-direction.
\end{remark}

\begin{remark}\label{rem:Sigma-dividing}
 {The surfaces $\Sigma = \{ \theta_{2} = c \} \times [-1, 2]$ will be denoted by $\Sigma\times \{c\}$.  For later purposes,
  it will be important to notice that the dividing curves of $\Sigma \times \{c\}$ 
 will contain $(\Sigma \times \{c\}) \cap T$, where $T$ is the boundary-parallel, pre-Lagrangian torus containing our link $L = (\Lambda_{1}, \Lambda_{2}, \Lambda_{3})$: the pre-Lagrangian $T$
 is foliated by lines that are tangent to the direction of the $S^{1}$-action. }
\end{remark}

As the contact isotopy $\psi_t$ 
defined on $(T^2\times[0,1], \xi)$ in Corollary~\ref{fix1}  is the identity near $\partial (T^2 \times [0,1])$, $\psi_t$ extends by the identity to an isotopy $\psi_{t}$ of $(T^2\times [-1,2], \xi')$.  This completes Step 2.

\medskip

Turning to Step 3, we will first find special neighborhoods $N_i$ of $\Lambda_i$.

\begin{lemma}\label{invtnbhd} 
There are  $S^1$-invariant neighborhoods $N_i$ of $\Lambda_i$ such that each $N_i$ has a convex torus boundary with dividing slope parallel to the $S^1$-action
and ruling slope $0$.
Furthermore, we can assume that 
  $N_1 \subset N$,  and $\psi_{t}|_{N_{1}} = id$.   
\end{lemma}

\begin{proof} 
By changing coordinates the model for the $S^1$-invariant contact structure on a neighborhood of  our pre-Lagrangian $T$ is $S^1\times S^1\times (-1,1)$ with contact form $-d\theta_1 + t\, d\theta_2$, the $S^1$-action being rotation in the $\theta_2$-direction and $T=S^1\times S^1\times \{0\}$. 
We can moreover assume that
$\Lambda_1=\{ \theta_{1} = 0 \} \times S^1 \times \{0\}$, and note that a neighborhood of this curve can by given by $(-2\epsilon,2\epsilon)\times S^1\times (-1,1)$ {with contact from $-dx +t\, d\phi$}. 

Now consider the embedding 
 \[
 \begin{aligned}
 \nu: D^2\times S^1&\to (-2\epsilon,2\epsilon)\times S^1\times (-1,1) \\
 (r, \theta, \phi) &\mapsto ( \epsilon r\cos \theta, \quad \phi, \quad \epsilon r\sin \theta).
 \end{aligned}
 \]  
 By construction, the image of this map is  an $S^1$-invariant neighborhood $N_1$ of $\Lambda_1$, and in these coordinates the characteristic foliation on the boundary is given by 
 $\nu^{*}(-dx - t d\phi)|_{r = 1} = \epsilon\sin\theta\, (d\theta +  d\phi)$.  
 So $\partial N_1$ has two circles worth of singularities (Legendrian divides) at $\theta=0,\pi$,
 and is non-singular elsewhere. Thus     $\partial N_1$ is a standard convex torus, and in $(\theta,\phi)$ coordinates the dividing slope is $\infty$ and the ruling slope is $-1$.

 By construction we have vertical dividing curves, but the ruling curves  have slope $-1$ rather than the desired slope of $0$.  
To fix this, we claim that we can add an $S^1$-invariant collar neighborhood $\partial {N_1}$ so that the new boundary has ruling slope 0 and 
 dividing curves of slope $\infty$. 
That such an invariant neighborhood exists is implicit in \cite{Honda00a}, but as a proof does not seem to exist in the literature, we will prove it in the following lemma.  In fact, the following lemma shows that we could make the slope
of the ruling curves be any finite number while keeping the infinite slope of the dividing curves. To set up convenient coordinates,
observe that there is a diffeomorphism of $\partial N_1$ that takes the ruling curves of slope $-1$ and the dividing curves of slope $\infty$ to a torus $T_0$ with ruling curves of  slope $0$ and dividing curves of slope $\infty$.
As both $\partial N_1$ and $T_0$ will have $[-1,1]$-invariant neighborhoods that are related by a contact diffeomorphism, the desired collar of $\partial N_1$ will follow from the following statement.

\begin{lemma}
If $\xi$ is a $[-1,1]$-invariant contact structure on $T^2\times [-1,1]$  with dividing curves of slope $\infty$ and rulings of slope $0$ that is invariant under rotations in the $\infty$ direction, then in any neighborhood of $T^2\times \{0\}$ there is a torus $T'$ isotopic to $T^2\times \{0\}$ that is also invariant under rotation in the $\infty$ direction and has dividing curves of slope $\infty$ and rulings of any slope other than $\infty$. 
\end{lemma}

\begin{proof} 
We will use coordinates $(\theta,\phi, t)$ on $T^2\times [-1,1]$, and take our contact structure to be given by $\cos\theta \, dt + \sin \theta \, d\phi$. This contact structure is invariant in the $\phi$ and $t$ directions. Observe that any torus given by $t = \text{constant}$
has two Legendrian divides at $\theta = 0, \pi$ and has ruling curves of slope $0$, meaning in the $\theta$-direction in the $\theta\phi$-plane.
We will first build a piecewise smooth torus with the desired properties, and then show that we can find a smooth torus with the desired properties. 

Suppose we want to realize a torus $T_s$ with ruling curves of slope $s<0$; we later explain how the argument needs to be modified to handle $s>0$. 
Given any neighborhood of $T^2\times \{0\}$, there is some $\epsilon>0$ such that $T^2\times [-2\epsilon,2\epsilon]$ is contained in the neighborhood. For any fixed $\theta_0\in (0,\pi/2)$ consider the torus $T_{\theta_0}$ obtained from 
$T^2\times\{0\}$ by removing $A_{[\theta_{0}, \pi/2]}^{0} = \{(\theta, \phi,0): \theta_0 \leq \theta \leq \pi/2\}$ and replacing it with the union of three annuli 
$A_{\theta_{0}}^{[0,\epsilon]}=\{(\theta,\phi, t): \theta=\theta_0, 0\leq t\leq \epsilon\}$,
$A^{\epsilon}_{[\theta_{0}, \pi/2]} =\{(\theta, \phi,\epsilon): \theta_0< \theta < \pi/2\}$, 
 and 
$A_{\pi/2}^{[0,\epsilon]}=\{(\theta,\phi, t): \theta=\pi/2, 0\leq t\leq \epsilon\}$. Denote by $A_0$ the complement of $A_{[\theta_{0}, \pi/2]}^{0}$ in $T^{2} \times \{0\}$.   Then
$$T_{\theta_0} = A_{0} \cup A_{\theta_{0}}^{[0,\epsilon]} \cup A^{\epsilon}_{[\theta_{0}, \pi/2]} \cup A_{\pi/2}^{[0,\epsilon]}$$
is a piecewise smooth torus.  
The characteristic foliation on $A_0$ has ruling curves of slope $0$ (parallel to the $\theta$ direction in the $\theta\phi$-plane) and two dividing curves of slope $\infty$. In addition, foliations on the annuli $A_{\theta_{0}}^{[0,\epsilon]}$, $A^{\epsilon}_{[\theta_{0}, \pi/2]}$, and $A_{\pi/2}^{[0,\epsilon]}$ are non-singular and of slope $-\cot \theta_0$ (in the $t\phi$-plane), 
$0$ (parallel to the $\theta$-direction in the $\theta\phi$-plane), and $0$ (parallel to the $t$-direction in the $t\phi$-plane), respectively.  By choosing $\theta_0$ appropriately the slope on $A_{\theta_{0}}^{[0,\epsilon]}$ can be any negative number. 
In particular, if we choose $\theta_0$ so that the slope on $A_{\theta_{0}}^{[0,\epsilon]}$ is $s/\epsilon$ then the slope of the ruling curves on the piecewise smooth $T_{\theta_0}$ is $s$.

 Now to get a smooth torus, first notice that $T_{\theta_0} = \Gamma \times S^{1}_{\phi}$, where $\Gamma$ is a closed curve with four corners in the $\theta t$-plane.  By replacing each corner of $\Gamma$ with curves approximating quarter circles of radius $\delta$ where $0< \delta \ll \epsilon$, we can obtain a smooth curve $\gamma_{\delta}$ in the $\theta t$-plane and a smooth 
 torus $T_{\theta_0,\delta} = \gamma_{\delta} \times S^{1}$.  As $\delta$ goes to $0$, the ruling curves on $T_{\theta_0,\delta}$ approach those on $T_{\theta_0}$.  Choose $0 < \theta_{0}^{-} < \theta_{0} < \theta_{0}^{+} < \pi/2$.  
The construction of the previous paragraph shows that on the piecewise smooth tori $T_{\theta_0^{\pm}}$, the ruling curves will have slopes $s^{\pm}$ satisfying
 $s^{-} < s< s^{+}$.  By choosing $\delta$ sufficiently small, the 
smooth tori $T_{\theta_0^{\pm},\delta}$ will have ruling curves of slopes $s^{\pm}_{\delta}$ satisfying
$s^{-}_{\delta}<s < s^{+}_{\delta}$.  By continuity, the Intermediate Value Theorem tells us that there exists a $\theta_{s}$, with $\theta_{0}^{-} < \theta_{s} < \theta_{0}^{+}$
such that the smooth torus $T_{\theta_s,\delta}$ has ruling curves of slope $s < 0$ and dividing curves of infinite slope, as desired.

To realize a torus with ruling curves of slope $s>0$, we start by choosing $\theta_0\in (\pi/2,\pi)$ and removing the annulus $A_{[\pi/2, \theta_{0}]}^{0} = \{(\theta, \phi,0): \pi/2  \leq \theta \leq \theta_0 \}$.  Following a  parallel procedure produces a smooth torus
with ruling curves of slope $s>0$ and dividing curves of infinite slope, as desired. 
 \end{proof}

To complete the proof of Lemma~\ref{invtnbhd}, first observe that the desired neighborhoods $N_2$ and $N_3$ can be constructed similarly. Lastly, we need to verify that we can assume that $\psi_{t} = id$ on $N_{1}$.
 Recall that above we showed that our isotopy $\psi_t$ can be assumed to be the identity on a neighborhood  $N$ of $\leg_1$. Since the neighborhood $N_1$ constructed above can be assumed to be arbitrarily small we can assume that  $N_1 \subset N$, and thus our isotopy $\psi_t$ is the identity on $N_1$. 
 \end{proof}

 Step 3 will be concluded once we prove the following corollary.
  
 \begin{corollary}\label{remove-2-disks} The manifold $X_1:=\overline{(T^2\times[-1,2])-N_1}$ is diffeomorphic to $\Sigma_{1} \times S^1$, where $\Sigma_{1}$ is an annulus with a disk removed, and
the contact isotopy $\psi_t$ from Step 2 restricts to define a contact isotopy  $\psi_{t}^{1}$ of $(X_1, \xi'|_{X_{1}})$.
The space $X_{123} :=\overline{(T^2\times[-1,2])-(N_1 \cup N_{2} \cup N_{3})}$ is diffeomorphic to $\Sigma_{123} \times S^1$, where $\Sigma_{123}$ is $\Sigma_{1}$ with two disks removed,
and the contact diffeomorphism $\psi_{1}$ from Step 2 {gives rise to} a contact diffeomorphism of $\psi^{123}$ of $(X_{123}, \xi'|_{X_{123}})$.  \end{corollary}
 
Although we know that $\psi_1$ interchanges $\leg_{2}$ and $\leg_{3}$, we will need to argue that $\psi_1$ induces a contactomorphism that interchanges  the neighborhoods $N_{2}$ and $N_{3}$. To prove this, we will employ the following lemma that is well known, but does not seem to be in the literature.

\begin{lemma}\label{isotoptori}
Suppose that $T_1$ and $T_2$ are {disjoint} convex tori in standard form with the same ruling and dividing slopes in a contact manifold $(M,\xi)$ such that the region $R$ between $T_1$ and $T_2$ is non-rotative. Then there is a contact isotopy of $(M,\xi)$ that is supported in a neighborhood of $R$ and takes $T_1$ to $T_2$. 
\end{lemma} 

\begin{proof} 
We begin by building a model for $R$. 
Given a surface $\Sigma$ with a singular foliation $\mathcal{F}$ that admits dividing curves, there is an $\R$-invariant contact structure $\xi'$ on $\Sigma\times \R$ that induces $\mathcal{F}$ on each $\Sigma\times \{t\}$, see \cite[Proposition~3.4]{Giroux91}. We will show that given any $a<b$ there is a contact isotopy of $\Sigma\times \R$ taking $\Sigma\times \{a\}$ to $\Sigma\times \{b\}$ that is supported in an arbitrarily small neighborhood of $\Sigma\times[a,b]$.  The lemma will then follow this since the hypotheses on $R$ imply
that $R$  has a neighborhood contactomorphic to a region $T^2\times[a,b]$ with such an $\R$-invariant contact structure. 

Let $\alpha$ be an $\R$-invariant contact form for the $\R$-invariant contact structure $\xi'$ on $\Sigma \times \R$. The function $H=\alpha(\partial_t)$ is the contact Hamiltonian generating the contact vector field $\partial_t$. Now let $g(t)$ be a function that is $1$ on some interval $[a,b]$ and $0$ outside a slightly larger interval. Then $gH$ generates a new contact vector field $v$ that agrees with $\partial_t$ on $[a,b]$ and is zero where $g$ is zero. The flow of $v$ gives a contact isotopy of $\Sigma\times\R$ that will take $\Sigma\times \{a\}$ to $\Sigma\times \{b\}$ and has support near $\Sigma\times (a,b)$. 
\end{proof}

Now we can complete the proof of Corollary~\ref{remove-2-disks}, and thus complete Step 3.

\begin{proof}[Proof of Corollary~\ref{remove-2-disks}]
Since $N_1$ is not moved by $\psi_t$, we have an induced isotopy $\psi^1_t$ on $X_1$. 
The only thing left to see is that we may assume that $\psi^1_1$ interchanges $N_2$ and $N_3$.  We will use an ``intermediate'' torus and Lemma~\ref{isotoptori} to extend $\psi^1_1$ by an isotopy to guarantee this happens.
Notice that $\Lambda_2$ is contained in $N_2\cap \psi^1_1(N_3)$, and thus there is an $S^{1}$-invariant neighborhood $N_2'$ of $\Lambda_2$ that is contained in this intersection such that the dividing and ruling slopes on $\partial N_2'$ agree with those on $\partial N_2$ and $\partial \psi^1_1(N_3)$. Tightness implies that the regions between $\partial N_{2}'$ and $ \partial (\psi_1^1(N_3))$ and between $ \partial N_{2}'$ and $\partial N_{2}$ are non-rotative.
Thus, from Lemma~\ref{isotoptori}, we can find a contact isotopy extending $\psi^1_t$ that takes 
$\partial (\psi_1^1(N_3))$ to $\partial N_2'$, and then another isotopy taking $\partial N_2'$ to  $\partial N_2$.
 Renaming the new isotopy $\psi^1_t$ again, we have $\psi_1^1(N_3)=N_2$. A similar argument arranges that  $\psi_1^1(N_2)=N_3$. 
\end{proof}

We are now ready for Step 4, where we derive our contradiction. 
The contradiction will arise by studying dividing curves on convex surfaces. 
In $(T^{2} \times [-1, 2], \xi') = (\Sigma \times S^{1}, \xi')$,  observe that for all $\theta$, $\Sigma \times \{ \theta \}$ is a convex annulus with dividing set containing 
$(\Sigma\times \{ \theta \}) \cap T$; see Remark~\ref{rem:Sigma-dividing}.
Now consider the two convex surfaces $\Sigma_{123} \times \{ \theta \}, \psi^{123}(\Sigma_{123} \times \{ \theta \}) \subset X_{123}$.
We see that the dividing curves of $\Sigma_{123} \times \{ \theta \}$ and $\psi^{123}(\Sigma_{123} \times \{ \theta \})$ connect
 the boundary components  corresponding to $\Lambda_{1}, \Lambda_{2},$ and $\Lambda_{3}$ in topologically different ways:

\begin{lemma}\label{combinatorics}  
 On
  $\Sigma_{123} \times \{ \theta \}$ there is a component $\gamma$ of the dividing set that connects $\partial N_1$ and $\partial N_2$; letting $p = \gamma \cap \partial N_1$, the dividing curve $\gamma'$ of $\psi^{123}(\Sigma_{123} \times \{ \theta \})$ containing $p$ connects $\partial N_1$ and $\partial N_3$.
\end{lemma} 

\begin{proof} As the dividing set of $\Sigma_{123} \times \{ \theta \}$ contains $(\Sigma_{123} \times \{ \theta \} ) \cap T$, it is clear that there is a connected component
of the dividing set, $\gamma$, connecting $\partial N_1$ and $\partial N_2$.  
The end of $\gamma$ intersecting $\partial N_1$ never 
moves throughout the isotopy $\psi^1_t$ since the isotopy is supported away from $N_1$.  Let $p = \gamma \cap \partial N_1$, and let $\gamma'$ 
be the connected component  of the dividing set of $\psi^{123}(\Sigma_{123} \times \{ \theta \})$ that contains $p$.  
{Since $\psi_t(\Sigma_{123} \times \{ \theta \})$ is convex with dividing curves given as the image under $\psi_t$ of the dividing curves of $\Sigma_{123} \times \{ \theta \}$}, we know that
for all $t$, there will be a component of the dividing curve  on $\psi^1_t(\Sigma_1 \times \{ \theta \})$ that  connects $p$ to $\psi_t^1(\partial N_2)$. Thus $\gamma'$ connects $\partial N_1$ and $\partial N_3=\psi^1_1(\partial N_2)$.
\end{proof}
 
If we knew there was a topological isotopy relative to the boundary
 of 
 $\psi^{123}(\Sigma_{123} \times \{ \theta \})$
  to 
  $(\Sigma_{123} \times \{ \theta \})$, we would immediately have a contradiction to Honda's result mentioned in Propositon~\ref{KoS1invariant}.  Although $\psi^{123}(\Sigma_{123} \times \{ \theta \})$ may not satisfy this isotopy condition,
  we can guarantee the existence of a surface $\Sigma'$ that does.

\begin{lemma}  There exists  a convex surface $\Sigma_\psi'$ so that 
$\partial \Sigma'=\partial (\Sigma_{123} \times \{\theta\} )$, $\Sigma'_\psi$ is isotopic to $\Sigma_{123} \times \{\theta\}$, and the  dividing curves on $\Sigma'_\psi$  and  $\psi^{123}(\Sigma_{123} \times \{ \theta \})$  topologically connect the boundary components in the same way.
\end{lemma}

\begin{proof}   
To abbreviate notation, let  $\Sigma_{123}$ denote $\Sigma_{123} \times \{\theta\}$.
From $\psi^{123}(\Sigma_{123})$, we can construct a new surface $\Sigma'_\psi$  with the desired properties
by an isotopy we call ``sliding the boundary along the boundary''.    
To define this sliding, let  $A$ be an annulus in a (torus) boundary component $B$ of $X_{123}$ with one boundary component of $A$ on 
$\partial ( \psi^{123}(\Sigma_{123} )) \cap B$, the other boundary component a ruling curve on $B$, and the rest of $A$ disjoint from 
$\psi^{123}(\Sigma_{123})$.  Then we can glue $A$ and $\psi^{123}(\Sigma_{123})$ together, 
round the corner between the two pieces, and push the interior of the new surface slightly into the interior of $X_{123}$. 
We can think of this new surface as obtained by isotoping one of the boundary components of $\psi^{123}(\Sigma_{123})$ along a boundary component of $X_{123}$ guided by $A$; we say this is the result of {\it sliding the boundary}. 
This isotopy can also be done where $A$ is actually an entire (torus) boundary component of $X_{123}$ rather than an annulus. 
If one always wants to work with annuli, then this would be a two step process: write a boundary of component $B$ as the union of two annuli $A_1$ and $A_2$
 and perform two slides using $A_1$ then $A_2$. We call a slide of ${\psi^{123}}(\Sigma_{123})$ along an entire boundary component of $X_{123}$ a {\it complete boundary slide}.
The dividing curves on the annulus $A$ added to $\psi^{123}(\Sigma_{123} )$ during a slide just run from one boundary component of $A$ to the other, so when we slide the surface the combinatorics of how the dividing curves intersect the boundary   is unchanged. 

We can clearly slide the boundary components of $\psi^{123}(\Sigma_{123})$ along the boundary of $X_{123}$ to get a surface $\Sigma_\psi$ such that $\Sigma_\psi$ and $\Sigma_{123}$ 
have the same boundary, and  the dividing curves on $\Sigma_\psi$ and $\psi^{123}(\Sigma_{123})$ connect the boundary components the same way.  Also notice that, since $\psi^{123}$ is the identity near $\partial_1$, we do not need to slide $\psi^{123}(\Sigma_{123})$ along $\partial N_1$.

We now complete our argument by showing that 
we may isotop $\Sigma_\psi$   
via complete boundary slides to get a  convex surface $\Sigma'_\psi$  
that is isotopic to 
$\Sigma_{123}$ relative to the boundary. 
Let $a_1, a_2, a_{3}$, and $a_4$ be arcs   on $\Sigma_{123}$ that  cut $\Sigma_{123}$ into a disk. If we label the boundary components of $\Sigma_{123}$ by $C_{0}, C_1, C_2, C_3,$ and $C_4$, where $C_0\cup C_4$ is the  boundary of the original annulus then we can choose the $a_i$ such that $a_i$ connects $C_{i-1}$ to $C_{i}$. 
Let $A_i= a_i\times S^1$. We can isotop the $A_i$ so that they are transverse to $\Sigma_{\psi}$. The intersection of $A_i$ with $\Sigma_{\psi}$ will consist of a single arc going from one boundary component of $A_i$ to the other and possibly some simple closed curves. 
All such simple closed curves must  bound disks on $A_i$, since the arc prevents them from being essential, and each must bound a disk on $\Sigma_{\psi}$, 
since $\Sigma_{\psi}$ is incompressible. One may use a standard innermost disk argument to isotop $A_i$ to remove the circles of intersection. 
Thus each $A_i$ intersects $\Sigma_{\psi}$ in exactly one arc $\eta_i$. 
The arcs $a_i$ and $\eta_i$ have the same boundary, but $\eta_i$ might not be isotopic to $a_i$ on $A_i$; after complete boundary slides of $\Sigma_{\psi}$ along boundary 
components of $X_{123}$, we can 
assume $\eta_i$ is isotopic to $a_i$ on $A_i$. 
Since some of the $a_i$ have end points on the same boundary component, doing slides to ``fix" one of the $\eta_i$ can may mess up another. But we can fix all the $\eta_i$ if we do so in the correct order. For example, we can first fix $\eta_1$ and $\eta_2$ by boundary slides along $C_0$ and $C_2$, respectively. Then fix $\eta_3$ by boundary slides along $C_3$, and fix $\eta_4$ by slides along $C_4$. Observe that we never needed to move $\Sigma_{\psi}$ along the boundary component on $\partial N_1$. 
Thus we may assume, after compete boundary slides, that all the $\eta_i$ are isotopic to $a_i$ on $A_i$.   
We can extend the isotopies of the $\eta_i$ to an ambient isotopy that will take $\Sigma_{\psi}$ to a surface $\Sigma'_{\psi}$
whose boundary agrees with the boundary of $\Sigma_{\psi}$ (which agrees with the boundary of $\Sigma_{123}$), and $\Sigma'_{\psi}$ also agrees with $\Sigma_{123}$ along the curves $a_i$.
By cutting $X_{123}$ open along the $A_i$, $X_{123}$ is cut open to $D^2\times S^1$, and $\Sigma_{123}$ and $\Sigma'_{\psi}$ are cut open into meridional disks that have the same boundary; two such disks are isotopic relative to their boundary. This isotopy can be done back in $X_{123}$ so we see that $\Sigma'_{\psi}$ is isotopic relative to the boundary to $\Sigma_{123}$. Moreover, since $\Sigma'_{\psi}$ is obtained by boundary sliding $\psi^{123}(\Sigma_{123})$, but never moving the boundary component on $\partial N_1$,
 the combinatorics of the dividing curves on $\Sigma'_{\psi}$ agree with those of $\Sigma_{\psi}$, which agreed with those of $\psi^{123}(\Sigma_{123})$.  
  \end{proof}
This completes Step~4 and hence the proof of Theorem~\ref{thickenedtori}.
\end{proof}

 As a corollary to Theorem~\ref{thickenedtori}, we can say something about the intersections of pre-Lagrangian tori in a basic slice.

 \begin{corollary}\label{pre-lag-order}  
Let $(T^2\times[0,1], \xi)$ be a basic slice containing two boundary-parallel pre-Lagrangian tori $T, T'$ of the same slope.  Suppose a link $L$ can be realized as a collection of leaves in the foliations of both $T$ and $T'$.  Then with respect to both tori, $L$ has the same cyclic ordering.
 \end{corollary}

 \begin{proof}  For a contradiction, suppose there exists such a link $L$ whose components are ordered differently by the pre-Lagrangian tori $T$ and $T'$.
 By considering a sublink and using arguments as in Step~\ref{step 1} of the proof of Theorem~\ref{thickenedtori}, we can assume $L$ has  $3$ components that are ordered differently with respect to $T$ and $T'$.  Suppose on $T$, these components are cyclically ordered as $(\Lambda_1, \Lambda_2, \Lambda_3)$,
 and on $T'$ they are cyclically ordered as $(\Lambda_1, \Lambda_3, \Lambda_2)$.  As in Step~\ref{step 1} of the proof of Theorem~\ref{thickenedtori}, we can isotop $T'$ so that it agrees with $T$ near $\Lambda_1$ and then isotop $\Lambda_2$ and $\Lambda_3$ along $T$ into the region where $T$ and $T'$ agree near $\Lambda_1$ and then out along $T'$ to exchange them, which is a contradiction to Theorem~\ref{thickenedtori}.  
  \end{proof}
 
 \begin{remark} \label{pre-lag-questions} It would be interesting to know if in a basic slice, two pre-Lagrangian tori that contain common leaves in their intersection must in fact be isotopic. 
  We note that each leaf in the characteristic foliation of one of the tori must intersect the other torus. To see this, suppose we had two such tori $T$ and $T'$ in a basic slice $T^2\times [0,1]$ and a leaf $L$ of $T$ was disjoint from $T'$. Then let $T''$ be a convex torus disjoint from $T'$ that contains $L$. It will have to have dividing slope that agrees with the slope of the characteristic foliation on $T'$ (since the twisting of the contact structure along $L$  will be zero with respect to the torus framing and thus cannot intersect the dividing set). But since $T'$ is pre-Lagrangian we can use a local model for it to see that between $T'$ and $T''$ we have tori with dividing slopes different from that of $T''$. This implies that the region between $T'$ and $T''$ is not minimally twisting, contradicting the fact that this is all taking place in a basic slice. This observation together with the above corollary strongly indicate that two such pre-Lagrangian tori must indeed be isotopic. 
\end{remark}

 \subsection{Forbidding Non-cyclic Permutations in $S^3$} \label{ssec:non-cyclic}
Having establishing the fact that it is not possible to do a non-cyclic permutation of the leaves of a pre-Lagrangian in a basic slice, we now move to the setting of $S^3$.
The proof of Theorem~\ref{non-cyclic-perm-gen} will follow from the following proposition.

 \begin{proposition} \label{S3-basic-slice}  For $q \geq p \geq 1$, let $L$ be an ordered Legendrian $(np, -nq)$-torus link in $S^3$
 with each component having $tb = -pq$.  Let $T$ be a pre-Lagrangian torus containing $L$; assume the components of $L$ are given the cyclic ordering from $T$.
 If there exists a contact isotopy $\psi_t$ of $S^3$ such that $\psi_{1}(L)$ realizes a permutation of $L$, then there exists a basic slice containing  pre-Lagrangian tori $T$ and $T'$ such  that both $T$ and $T'$  contain $L$ as leaves of their foliations,
 $T$ induces the cyclic ordering given by $L$, and  $T'$ induces the cyclic ordering given by $\psi_{1}(L)$. 
 \end{proposition}
 
 Before proving Proposition~\ref{S3-basic-slice}, we observe that Propositon~\ref{S3-basic-slice} leads to a short proof of Theorem~\ref{non-cyclic-perm-gen}.
 
 \begin{proof}[Proof of Theorem~\ref{non-cyclic-perm-gen}]  By Proposition~\ref{S3-basic-slice},  a non-cyclic permutation of the leaves of 
 a Legendrian $(np, -nq)$-torus link $L$  in $S^{3}$ where all the components have $tb=-pq$ implies the existence of two pre-Lagrangian tori in a basic slice that induce different cyclic orderings on the leaves of $L$, a contradiction
 to Corollary~\ref{pre-lag-order}.  
 \end{proof}

To prove Proposition~\ref{S3-basic-slice}, we will first establish three lemmas about the image of a pre-Lagrangian torus containing $L$ under a contact isotopy $\psi_{t}$ such that   $\psi_{1}(L)$ is a
permutation of $L$.

First we develop a dimension reduction set up.
In the proof of Theorem~\ref{thickenedtori},  Step~\ref{step 2} allowed a dimension reduction so that we could represent important objects in our 3-dimensional basic slice in a 
2-dimensional annulus. Similarly, the following lemma will allow us to represent relevant tori, solid tori, and annuli in $S^2$ rather than $S^3$.
The annuli constructed in this lemma will later be used to find a basic slice containing two pre-Lagrangian tori containing $L$.

\begin{lemma}\label{setup} 
Suppose $L \subset S^{3}$ is a Legendrian $(p,-q)$-torus link that arises as the leaves of a 
 pre-Lagrangian torus $T_{0}$ of slope $-q/p$, and 
$\psi_{t}$ is a contact isotopy of $S^{3}$ such that $\psi_{1}(L)$ is a permutation of $L$; let $T_1= \psi_1(T_0)$.  
There is a Seifert fiber structure on $S^3$ with regular fibers being $(p,-q)$-torus knots. The base of the fibration is $S^2$; if $p\neq 1$, there are two singular fibers $K_0, K_1$, while if $p =1$  only $K_{1}$ is singular.  The fibers $K_{0}, K_{1}$, given as the
pre-image of the poles of $S^{2}$, form a Hopf link.  With respect to this Seifert fiber structure, we can arrange the following. 
\begin{enumerate}
\item We can assume our link $L$ is contained in the pre-Lagrangian torus $T_{0}$, which is the pre-image of the equator.  The torus $T_0$ separates $S^3$ into two solid tori $S_0$ and $S_1$.
\item The pre-Lagrangian torus $T_1=\psi_1(T_0)$ can be isotoped relative to $L$ to be a convex torus $T_{1}'$ such that $T_{1}'$ is the pre-image of a curve $c \subset S^{2}$ that separates  the poles; {the torus $T_{1}'$ separates $S^3$ into two solid tori $S_0'$ and $S_1'$}.
 Thus the pre-image of the poles of $S^{2}$ give $K_{i}$, which are core-curves of both the solid tori $S_{i}$ and $S_{i}'$, $i = 0,1$.
\item The pre-image of curves in $S^{2}$ joining the poles to points on the equator in the complement of the curve $c$ define annuli $A_{i}$ 
with embedded interiors such that $\partial A_i$ is the union of a Legendrian $(p, -q)$-curve on $T_0$ and,  when $i = 0$, a curve that wraps $p$ times around $K_{0}$  and, 
when $i = 1$, a curve that wraps $-q$ times around  $K_1$.  Moreover, $A_i$ is disjoint from $T_1'$ and intersects $T_0$ only along its boundary, which is  a leaf of the characteristic foliation of $T_{0}$.
\end{enumerate}
\end{lemma}

\begin{proof}[Proof of Lemma~\ref{setup}] {We will view $S^3$ as the unit sphere in $\C^2$; the standard contact structure on $S^3$ is then given by   the kernel of $\alpha|_{TS^{3}} = (r_1 d\theta_1 + r_2 d\theta_2)|_{TS^{3}}$, where $z_j = r_j e^{i\theta_j}$.}
The Seifert fiber structure is well known. We first view $S^3$ as the join of two circles: consider the map
\[
\Psi: S^1\times S^1\times [0,1]\to S^3: (\theta_1,\theta_2, t)\mapsto \left( \cos\left(\frac\pi 2 t\right)e^{i\theta_1}, \sin\left(\frac\pi 2 t\right)e^{i\theta_2} \right).
\]
Then $\Psi$ restricted to $S^1\times S^1\times(0,1)$ parametrizes  $S^3-H$, where $H = (S^1 \times \{0\}) \cup (\{0\} \times S^1)$ is the Hopf link, and $S^1\times S^1\times\{0,1\}$ is ``collapsed" onto $H$.   In these coordinates the vector field $p \partial_{\theta_1} - q \partial_{\theta_2}$ generates the orbits of the Seifert fiber structure on $S^3$. All the claimed properties about the Seifert fiber structure of $S^3$ are seen in this model. 
 
To verify (1), {first observe that in the coordinates given by $\Psi$ the contact planes are always tangent to the $[0,1]$-factor, have slope $-\infty$ as $t$ limits to $0$ and then rotation through negative numbers and limit to  slope $0$ as $t$ approaches $1$. In our identification of the  orbit quotient of $S^3$ to $S^2$, we can assume that the preimage of the equator is the pre-Lagrangian $T_0$ with characteristic foliation having slope $-q/p$.}
From our {unordered} classification  of negative torus links, we can assume that $L$ is a subset of the leaves of the foliation on $T_0$. We can now define $S_0$ and $S_1$ as the pre-image of the upper and lower hemispheres of $S^2$. The $K_i$, defined as the preimages of the poles, are  the cores of $S_i$. This completes Item~(1).

For Item~(2), we first isotop $T_1= \psi_1(T_0)$ relative to $L$ to a convex torus $T_1'$ such that $T_0$ and  $T_1'$ intersect transversally, and $T_0 \cap T_1'$  consists of $L$ and potentially other curves on both tori. 

\noindent
{\bf Claim:} We can choose the isotopy so  that $T_1'$ is
the pre-image of a curve $c \subset S^2$, where the curve $c$ separates  
 the poles of $S^2$. 

\begin{proof}[Proof of Claim]
The intersection $T_0\cap T_1'$ consists of simple closed curves containing $L$: $2n$ of them are $(p, -q)$ curves on both tori, and $m$ of them are null-homotopic curves on both tori. Using a standard innermost disk argument,  $T_1'$ can be isotoped, relative to $L$, to remove the null-homotopic curves. 

Now each component of $T_1' \setminus (T_0\cap T_1')$ is an annulus in $S_0$ or $S_1$ with boundary consisting of $(p, -q)$ curves in $T_{0}$. Thus the annuli are incompressible in $S_{0}$ or in $S_{1}$, since inclusion induces an injection on their fundamental groups. {We will have our desired existence of the curve $c$  if we can show that we can isotop $T_1'$ relative to $T_0\cap T_1'$ so that
each of these annuli  become a union of  fibers.} 
 An {incompressible} annulus in a Seifert fiber space is either boundary parallel or boundary incompressible, see \cite[Section~7]{JacoShalen79} and \cite[Section~5]{Johannson79}. Moreover a boundary incompressible surface is isotopic to either a vertical (union of regular fibers) or horizontal (transverse to each fiber) surface. Since the {boundary condition on the annuli imply that the} annuli cannot be horizontal, the annuli must be either boundary parallel annuli or vertical surfaces, which in this case are also boundary parallel. Notice that a boundary parallel annulus with boundary a union of fibers can be isotoped to be a vertical surface, since there is a vertical surface with the same boundary, and any two annuli with the same boundary will be isotopic.  Thus after isotopy relative to $T_0\cap T_1'$, we can assume $T_1'$ is also a vertical surface. In particular, there is a simple closed curve $c$ in $S^2$ such that $T_1'$ is the pre-image of $c$ under the fibration.  
 
 Furthermore, we can argue that $c$  separates the poles of $S^{2}$ as follows.  If not and $p \neq 1$
 then there is a disk $D$ bounded by $c$ that contains both singular points; the pre-image of $D$ is not a solid torus (in fact it is the complement of a $(p,-q)$-torus knot), which contradicts $T_1'$ being isotopic to $T_0$.  When $p =1$, if $c$ does not separate the poles, then since $K_{0}$ is a regular fiber, 
 and we can further isotop $T_1'$ past  $K_0$.
\end{proof}

\smallskip
Now that we have established Item~(2), we move on to the last item.   If one takes arcs $a_0$ from the north pole to the equator and $a_1$ from the south pole to the equator, both avoiding $c$, then their pre-images will be the annuli $A_0$ and $A_1$ satisfying the properties claimed in Item~(3).  
  \end{proof}

Moving on, it will be helpful to keep in mind that the convex $T_{1}'$, which has the nice description as the pre-image of the curve $c \subset S^{2}$, was obtained from a potentially large isotopy of the pre-Lagrangian $T_{1} = \psi_{1}(T_{0})$.  We now show that the cyclic ordering of the components of $L$ on $T_{1}'$ agrees with the cyclic ordering on  $T_{1}$.
To do this, we will find a non-rotative neighborhood for the convex torus $T'_1$;
recall that in Definition~\ref{defn:complementary},
 we defined a complementary annulus for a non-rotative
contact structure, and in Lemma~\ref{prelagmodel} we discussed how such a complementary annulus can be used to define a cyclic ordering of the Legendrian divides of a convex torus.

\begin{lemma}\label{claim2}
Suppose $L$, $T_{0}$, $T_{1}$, and $T_{1}'$ are as in the statement of Lemma~\ref{setup}.
The convex torus $T_1'$ has a neighborhood $R_{1}=T_1'\times[-\epsilon,\epsilon]$ such that $\xi$ restricted to $R_{1}$ is $[-\epsilon,\epsilon]$-invariant,
 each boundary component of $R_{1}$ has two dividing curves, and the cyclic ordering of $L$ induced by a complementary annulus in $R_{1}$ agrees with the one induced by the pre-Lagrangian torus $T_1$.  
Moreover, the annuli $A_i$ constructed in Lemma~\ref{setup} can be constructed to be disjoint from $R_{1}$.
\end{lemma}

\begin{proof}  
We will start by constructing a  convex torus $T_1''$ that has a neighborhood with the desired properties, and then use a Discretization of Isotopy technique to show that $T_1''$ can be moved to our $T_1'$ in such a way that the desired neighborhood properties persist.

By Lemma~\ref{local-model}, we know that from the pre-Lagrangian $T_1$, we can find a  convex torus $T_1''$ that has  a neighborhood $R''$ with the desired properties. 
Since both $T_1'$ and $T_1''$ are isotopic to $T_1$ fixing $L$, there is a smooth isotopy of $T_1''$ to $T_1'$ fixing $L$ . By the Discretization of Isotopy technique, \cite[Section~2.2.3]{Honda02}, we can find a sequence of convex surfaces $F_0$,\ldots, $F_l$, such that $F_0=T_1''$, $F_l=T_1'$, and each pair $F_i\cup F_{i+1}$ cobound a thickened torus.  
 As the $F_{i}$ are disjoint, they cannot all contain $L$; however each convex surface $F_{i}$ will contain an ordered link among its Legendrian divides that can be canonically identified with $L$ since
we can assume that all the $F_i$ intersect $T_1$ near $L$ in leaves isotopic to $L$. 
We will abuse terminology and say $L$ is on all the $F_i$ using this identification.

We will  inductively prove that each torus $F_i$ contains  $L$ as a link  with the same ordering as that given by $T_{1}$.  Namely we will argue that 
 each $F_i$  
contains $L$ as 
a subset of its Legendrian divides and is contained in an $I$-invariant thickened torus $R_i$ whose boundary components have two dividing curves each of slope $-q/p$, and the cyclic order of $L$
on $F_i$ induced from a complementary annulus in $R_i$ agrees with the order coming from $T_1$.

The base case is immediate from construction with $R_0=R''$. 
Now we inductively assume the result is true for $F_k$. 
Notice that $F_k$ splits $R_k$ into two pieces $R_k^\pm$.  
Now $F_{k+1}$ is either on the positive or negative side of $F_k$. We assume the positive side; the argument for the negative side is analogous.  {We will find a thickened torus neighborhood $R_{k+1}$ of $F_{k+1}$ where $R_{k+1}^{-}$ contains $R_{k}^{-}$.}
Consider $S^3\setminus F_{k+1}$: there will be two copies of $F^+_{k+1}$ and $F_{k+1}^-$ of $F_{k+1}$ in the cut open manifold, and $F_{k+1}^\pm$ will bound a solid torus $S_{k+1}^\pm$, {with $S_{k+1}^{-}$ containing $R_{k}^{-}$.} 
As the slope of convex tori in $S_{k+1}^+$ parallel to the boundary is not fixed, we know that there is a thickened torus $R_{k+1}^+$ in $S_{k+1}^+$ with one boundary component $F_{k+1}^+$ and the other boundary component being a convex torus with two dividing curves of slope $-q/p$. Let $R_{k+1}^m$ be the  region between $F_k$ and $F_{k+1}$, and set $R_{k+1}^-=R_{k}^-\cup R_{k+1}^m$. Then $R_{k+1}=R_{k+1}^+\cup R_{k+1}^-$ is an $I$-invariant thickened torus containing $F_{k+1}$.
 We now need to check the statement about the ordering on $L$ induced by a complementary annulus in $R_{k+1}$. To this end, notice that $R_k^+$ and $R_{k+1}^m\cup R_{k+1}^+$ are both non-rotative outermost layers for $F_k^+$ in $S^3\setminus F_k$ as described in \cite{Honda01}. So if  $A^+_{k}$ is any complementary annulus in $R_k^+$, and $A_k'$ is a complementary annulus in $R_{k+1}^m\cup R_{k+1}^+$, then the dividing curves on these annuli are disk equivalent \cite[Theorem~1.3]{Honda01}. This means that if we add a disk to the outermost boundary of these annuli and extend the dividing set by an arc in the new disk then the resulting multi-curves in the disk are isotopic. This implies that the ordering on the components of $L$ induced by $A_k^-\cup A_k^+$ and $A_{k}^-\cup A_k'$ are the same. But we can write the annulus $A_k'$ as $A_{k+1}^+\cup A_{k+1}^m$ by splitting it along a ruling curve in $F_{k+1}$ and then $A_{k}^-\cup A_k'=A_k^-\cup A_{k+1}^m\cup A_{k+1}^+=A_{k+1}^-\cup A_{k+1}^+$. And thus the orders on $L$ induced by the complementary annuli in $R_k$ and $R_{k+1}$ are the same, and the induction argument is complete.  The desired neighborhood $R_{1}$ of $T_{1}'$ is the $R_{l}$ constructed in the induction argument.
 
We must now see that the $A_i$ can be chosen to be disjoint from $R_1$. We first construct a thickened torus $R_2$ that contains $T_1'$ is non-rotative with two dividing curves on each side and is disjoint from the $A_i$. If we achieve this then notice that $T_1'$ breaks $R_1$ into two pieces $R_1^1\cup R_1^+$ and breaks $R_2$ into two pieces $R_2^-\cup R_2^+$. Both $R_1^-$ and $R_2^-$ are non-rotative outermost layers for a component of $S^3\setminus T_1'$ and so are disk equivalent, and similarly for $R_1^+$ and $R_2^+$. Thus as discussed above, the ordering of the Legendrian divides on $T_1'$ coming from $R_1$ and $R_2$ is the same. Thus we may replace $R_1$ with $R_2$ to complete the proof. 

To construct $R_2$ we will use the notation of Lemma~\ref{setup}. Notice that we can draw two circles $\gamma_\pm$ on $S^2$ that enclose the curve $c$ from Lemma~\ref{setup} that defines $T_1'$. These curves can be chosen to each intersect the equator transversely in two points and separate the poles. Moreover, let $T_\pm$ be the pre-image tori of $\gamma_\pm$ under the projection $S^3\to S^2$ given in Lemma~\ref{setup}. Notice $T_\pm$ has two Legendrian $-q/p$ curves coming from its intersection with the equator. We claim that $T_\pm$ has just two dividing curves. If it had more than two dividing curves, then one region in the complement of the curves would have to be disjoint from $T_0$ (the pre-image of $c$). Thus we could Legendrian realize a $(p,-q)$-torus knot on it. This knot would have $tb=-pq$ and be disjoint from $T_0$. However, arguing as in Remark~\ref{pre-lag-questions} this would imply that the contact structure was overtwisted. Thus the region $R_2$ is simply the thickened torus bounded by $T_-$ and $T_+$.  
  \end{proof}

 Next we show that $T_{0}$ and the neighborhood $R_{1}$ of $T_{1}'$ are contained in a basic slice.
 
\begin{lemma}\label{claim3} 
Suppose $L$, $T_{0}$, $T_{1}$, and $T_{1}'$ are as in the statement of Lemma~\ref{setup}, and 
$R_{1}$ is the non-rotative neighborhood of $T_{1}'$ as in the statement of Lemma~\ref{claim2}. Then
there is a thickened torus $V$ in $S^3$ that contains $T_0$ and $R_{1}$ such that 
$\xi|_{V}$ is a basic slice.
\end{lemma}

\begin{proof} 
We first consider the case where $p \neq 1$, and thus $-q/p$ is not an integer. The strategy here will be to start with standard neighborhoods of Legendrian realizations of the Hopf Link $K_0\cup K_1$ that are disjoint from $T_0$ and $R_1$, and use the annuli $A_{i}$ from Lemma~\ref{setup} to show that these neighborhoods can be thickened so that their complement is a basic slice $V$ that contains  $T_{0}$ and $R_{1}$.

Let $U_i$ be a small neighborhood of $K_i$ that is disjoint from $R_{1}\cup T_0$. To describe slopes on the torus boundary neighborhoods of $K_i$,  
we will always  use longitude-meridian coordinates coming from $T_0$, whose orientation coincides with the orientation as the boundary of the solid torus $S_0$ with core $K_0$.  We can assume that $\partial U_i$ is convex and that the dividing slope for $\partial U_0$ is $-l$ for some large integer $l$, and the slope for $\partial U_1$ is $-1/k$ for some large integer $k$\footnote{The meridian and longitude of $U_1$ is reversed from those of $U_0$, which explains the fraction}. We can use the annuli $A_i$ to create embedded annuli $A_i'$ with boundary consisting of a ruling curve {of slope $-q/p$} on $\partial U_i$ and a leaf  $\gamma_i$ of the characteristic foliation on $T_0$. Moreover the interiors of the $A_i'$ are disjoint from $R_{1}\cup T_0$.
Then $A_i$ intersects the dividing curves on $\partial U_i$ exactly {$2| lp-q|$},
 respectively 
 {$2 |p -kq|$} times. Thus, since the dividing curves of $A_i'$ do not
 intersect $\gamma_i$,  the Imbalance Principle \cite{Honda00a} says there is a bypass for $\partial U_i$ along $A_i'$.  Assume $-m-1 < -q/p < -m$, for $m \in \mathbb Z$; we can find such bypasses to raise the slope of $U_0$ to $-m-1$ and $U_1$ to $-m$. Notice that these solid tori are still disjoint from $R_{1} \cup T_0$. 
Now let $V=S^3\setminus(U_0\cup U_1)$. By construction $V$ is a basic slice and contains $T_0$ and $R_{1}$. 

In the case where $p = 1$, arguing as in the case when $p \neq 1$,  
we can find a thickened torus $T^2\times [0,1]$ in $S^3$ that contains $T_0$ and  the non-rotative neighborhood $R_{1}$ of $T_{1}'$ and has dividing curves of slope $s_0=-q-1$ and $s_1=-q+1$ on $T^2\times \{0\}$ and $T^2\times \{1\}$, respectively, if $-q\not=-1$ and $s_0=-1/2$ and $s_1=-2$ if $-q=-1$. Notice that $\xi$ restricted to $T^2\times [0,1]$ is the union of two basic slices one with boundary slopes $s_0$ and $-q$ and the other with slopes $-q$ and $s_1$. Since there is a pre-Lagrangian torus of slope $-q$ in this contact manifold, Lemma~\ref{notut} says the signs of the basic slices must agree; say they are both positive. We can now glue the positive basic slice with boundary slopes $0$ and $s_0$ to the back of $T^2\times [0,1]$ and the positive basic slice with slopes $s_1$ and $-\infty$ to the front of  $T^2\times [0,1]$. The resulting thickened torus $T_2\times [-1,2]$ is a basic slice by \cite[Theorem~4.25]{Honda00a}, and it contains $T_0$ and the non-rotative neighborhood $R_{1}$ of $T_{1}'$.
\end{proof}

We now complete the proof of Proposition~\ref{S3-basic-slice}. 

\begin{proof}[Proof of Proposition~\ref{S3-basic-slice}]   
Suppose that $L$ is a Legendrian $(np,-nq)$-torus link where each component has $tb = -pq$.
 Assume there exists a Legendrian isotopy $\psi_t$ of $S^{3}$ such that $\psi_1(L)$ is a permutation of $L$.  Let $T_0, T_1= \psi_1(T_0)$ be the pre-Lagrangian tori containing $L$ and $\psi_1(L)$, let $T_{1}'$ be the convex torus that is isotopic  to $T_{1}$ relative to $L$ guaranteed by Lemma~\ref{setup}, let $R_{1}$ be the non-rotative 
neighborhood of $T_{1}'$ guaranteed by Lemma~\ref{claim2}, and let $V$ be the basic slice containing $T_{0}$ and $R_{1}$ guaranteed by Lemma~\ref{claim3}. If we knew that
the pre-Lagrangian $T_{1}$ was also in this basic slice $V$, we would be done by Corollary~\ref{pre-lag-order}. As we cannot guarantee this inclusion of $T_{1}$, we  
will show that there is a contactomorphism $\kappa$ from the basic slice $V$  
 to a basic slice $(T^2\times [0,1],\xi)$ such  that $\kappa(T_1')$ is a convex torus  
 whose Legendrian divides are given by $\kappa(T_{1}') \cap P_{1}$, where $P_{1}$ is a pre-Lagrangian torus, and the ordering of $\kappa(L)$ on $\kappa(T_1')$ coming from $P_{1}$ and from 
  $\kappa(R_{1})$ are the same.  
 Observe that $\kappa(T_0)$  will be a pre-Lagrangian torus containing $\kappa(L)$. By Corollary~\ref{pre-lag-order}, we know that the cyclic ordering of $\kappa(L)$ from 
 $\kappa(T_{0})$ and $P_{1}$ (and thus $T_{1}$) must agree.  This means that the cyclic orderings of $L$ through $T_{0}$ and $T_{1}$ agree, thus establishing Proposition~\ref{S3-basic-slice}.

To see that the claimed contactomorphism exists notice that we can start with our non-rotative $R_{1}$ and attach a thickened torus to the front and back faces of $R_{1}$ to obtain a basic slice $(T^2\times [0,1],\xi)$ and thus a contactomorphism $\kappa: V \to T^2\times [0,1]$.    
Now  Lemma~\ref{prelagmodel} says there is a pre-Lagrangian torus $P_{1}$ inside of $T^2\times [0,1]$ such that the Legendrian divides of $\kappa(T_{1}')$ 
are given by  $\kappa(T_{1}') \cap P_{1}$,  and there is a non-rotative thickened torus $R_{P_{1}}$ containing $\kappa(T_{1}')$ and $P_{1}$ that orders the divides of $\kappa(T_{1}')$ in the same way as $P_{1}$. We claim that $R_{P_{1}}$ and $\kappa(R_{1})$ order the divides of $\kappa(T_{1}')$ in the same way, which will complete the proof. To see this notice that $\kappa(T_{1}')$ splits $\kappa(R_{1})$ and $R_{P_{1}}$ into two pieces $\kappa(R_{1})^{\pm}$ and $R_{P_{1}}^{\pm}$, respectively. Now $\kappa(R_{1})^{+}$ and $R_{P_{1}}^{+}$ are non-rotative outer layers for $(T^2\times [0,1])\setminus \kappa(T_{1}')$ and thus are disk equivalent (as discussed in the proof of Lemma~\ref{claim2}), and thus their complementary annuli are disk equivalent. Similarly for $\kappa(R_{1})^{-}$ and $R_{P_{1}}^{-}$. Thus, as in the proof of Lemma~\ref{claim2}, we see that $R_{P_{1}}$ and $\kappa(R_{1})$ define the same cylic ordering of the divides of $\kappa(T_{1}')$, as desired.
\end{proof}

\section{Cable links}\label{sec:cable}
In this section, we will always assume that $K$ is an oriented smooth knot type.  As described in Section~\ref{cable-links}, for $n \geq 1$ and  $p, q \in \mathbb Z$ with $p \geq 1$ and $\gcd(p,q) = 1$,  $K_{(np, nq)}$ will denote the $n$-component, slope $q/p$-cable link for the knot type $K$. 

We begin the section by describing the non-destabilizable Legendrian representatives of $K_{(np,nq)}$: these will  be ``standard Legendrian cables'' and, for some integral slope values, ``twisted $n$-copies''.  The standard Legendrian cables
will always have max $tb$, and the twisted $n$-copies will have max $tb$ if and only if $n = 2$. This leads to the unordered classification of Legendrian cables in  Section~\ref{cables-unordered};
see Theorem~\ref{unorderedcable}.
We then move on to understand the ordered classification.   All components of the standard Legendrian cables are Legendrian isotopic, and in Section~\ref{standard-perms} we examine which permutations can be realized.   This leads to the proof of the ordered classification of Legendrian cables stated in Theorem~\ref{orderedcable}.

\subsection{Non-destabilizable Legendrian Cables} \label{cables-non-destab}

We begin by defining ``standard Legendrian cables,'' which are Legendrian representatives of the cable links $K_{(np, nq)}$; their construction
 will depend on the slope $q/p$ and will use a standard neighborhood  of a Legendrian representative $\Lambda$ 
of  $K$ with specified $tb$ value, as defined in Definition~\ref{std-nbhd}.   We will use the notation $\lceil  q/p \rceil$ to denote the least integer greater than or  equal to   $q/p$.

\begin{definition}\label{cable-defn} Given a knot type $K$ and $p, q \in \mathbb Z$ such that  $p \geq 1$ and $\gcd(p, q) = 1$, fix a Legendrian representative $\Lambda$ of $K$ such that 
$$tb(\Lambda) = \begin{cases}
 \overline{tb}(K), & q/p \geq  \overline{tb}({K}) \\
\lceil  q/p \rceil, & q/p <  \overline{tb}({K}).
 \end{cases}
 $$
Fix $n \geq 1$. Then from $\leg$, when $q/p \geq \overline{tb}(K)$ (or when $p = 1$ and $q/p < \overline{tb}(K)$), we define the \dfn{standard Legendrian $(np, nq)$-cable of $K$}, denoted $\Lambda_{(np, nq)}$ (respectively, $\Lambda_{(n, nq)}$), 
and when $q/p < \overline{tb}(K)$ {and $p > 1$}, we define two  \dfn{standard Legendrian $(np, nq)$-cables of $K$}, denoted $\Lambda^{\pm}_{(np,  nq)}$, 
as follows.

\begin{description}[align=left]
\item[greater-slope cables] Suppose $q/p>\overline{tb}(K)$.   Let $N$ be a standard neighborhood of $\Lambda$ with ruling curves of slope 
$q/p$. Then $\Lambda_{(np,nq)}$ is defined by taking $n$ ruling curves on $\partial N$.
\item[$\overline{tb}(K)$-slope cables] When $p = 1$ and  $q/p=\overline{tb}(K)$,  
$\Lambda_{(n,nq)} $ is defined to be the $n$-copy of $\Lambda$. {Recall this involves taking leaves of a pre-Lagrangian annular neighborhood inside a standard neighborhood of $\Lambda$;}
see
Definition~\ref{n-copy}. 
\item[integral and lesser-slope cables] { 
When $p = 1$ and  $q/p<\overline{tb}(K)$,  $\Lambda_{(n,nq)} $ is defined to be the $n$-copy of $\Lambda$. }
\item[nonintegral and lesser-slope cables] Suppose $p > 1$ and $q/p<\overline{tb}(K)$.  Let $N$ be a standard neighborhood of $\Lambda$; inside $N$ are 
   standard neighborhoods $N^{\pm}$ of $\Lambda^{\pm}$, the $\pm$-stabilizations of $\Lambda$.  There exists 
   a pre-Lagrangian torus $T^{\pm}$ of slope
$q/p$ in the basic slice  $N\setminus N^{\pm}$, see Remark \ref{stab-basic-slice}, and
 $\Lambda^{\pm}_{(np,nq)}$ is defined by taking $n$ leaves in the foliation of $T^{\pm}$.
\end{description}
\end{definition} 

\begin{remark} {These definitions of the standard $(np,nq)$-cables generalize the construction of our max-$tb$ Legendrian torus links.  For suppose that $K$ is the unknot and take, as usual, $ |q| \geq p \geq 1$ with $\gcd(p,q) = 1$.
   Then observe:
 \begin{itemize}
\item $q/p > \overline{tb}(K) = -1$ implies $q \in \mathbb Z^{+}$, and  the construction of the standard Legendrian greater-slope cables of the unknot agrees with the construction of the max-$tb$ positive torus links;
\item $q/p = \overline{tb}(K) = -1$ implies $q = -1$,   and the construction of the  standard Legendrian $\overline{tb}(K)$-slope cables of the unknot agrees  with the construction of the max-$tb$ representatives of the $(n,-n)$-torus link.
\item   $q/p < \overline{tb}(K) = -1$ and $p = 1$ implies $q \in \mathbb Z^{-}$, and the construction of the standard Legendrian lesser-slope cables of the unknot agrees  with the construction of the symmetric max-$tb$ representatives of the $(n,-nq)$-torus links, which  are obtained as $n$-copies.  
\item $q/p < \overline{tb}(K) = -1$ and $p > 1$ implies $q \in \mathbb Z^{-}$, and the construction of the standard Legendrian lesser-slope cables of the unknot agrees  with the construction of the max-$tb$ representatives of the negative torus links with non-trivial components.
\end{itemize}}

\end{remark}

\begin{remark} \label{std-cable-tb-calc} All components of any of the standard Legendrian cables are Legendrian isotopic.  Moreover, in Lemma~2.1 of \cite{EtnyreHonda05} it is shown how to compute the Thurston-Bennequin invariant  of cables, while in Lemmas~2.2 and~3.8 of \cite{EtnyreHonda05} it is shown how to compute the rotation number of cables (but note that the slope conventions in \cite{EtnyreHonda05} are reversed to the conventions in this paper).  This leads to:  
\begin{enumerate}
\item In the standard Legendrian greater-slope cables,   each component $\leg_{i}$ of $\leg_{(np,nq)}$  has 
$$tb(\leg_{i}) =  pq - |p \cdot  \overline{tb}(K) - q| = pq - q + p\cdot \maxtb(K), \quad \text{ and } \quad r(\leg_{i}) = p r(\leg).$$
\item In the standard Legendrian $\overline{tb}(K)$-slope cables, each component $\leg_{i}$ of $\leg_{(n,nq)}$ has 
$$tb(\leg_{i}) = \overline{tb}(K), \quad \text{ and } \quad r(\leg_{i}) = r(\leg). $$
\item {
In the standard Legendrian {{integral and }lesser-slope cables}, if ${q} =    \overline{tb}(K)  -  {s}$, for  $s> 0$, each component $\leg_{i}$ of $\leg_{(n,nq)}$has 
$$tb(\leg_{i})  = \overline{tb}(K)  -  {s} = q , \quad \text{ and } \quad r(\leg_{i}) = r(\leg).$$}
\item In the standard Legendrian {{nonintegral and }lesser-slope cables}, if $\frac{q}{p} =   \lceil \frac{q}{p} \rceil  -  \frac{s}{p}$, for  $0 < s < p$, each component $\leg_{i}^{\pm}$ of $\leg^{\pm}_{(np,nq)}$has 
$$tb(\leg_{i}^{\pm})  =  pq, \quad \text{ and } \quad r(\leg_{i}^{\pm}) = p\, r(\leg) \pm s = p\, r(\leg) \pm (p \lceil q/p \rceil - q).$$
\end{enumerate}
\end{remark}

When $n \geq 2$, the standard Legendrian cables of uniformly thick knot types   
will always have maximal \tb invariant.

\begin{lemma} \label{std-cables-max-tb} If $K$ is a uniformly thick knot type, then, for $n \geq 2$,
the standard Legendrian $(np,nq)$-cable of $K$ realizes 
$\overline{tb}(K_{(np, nq)})$, for all $p, q \in \mathbb Z$, $p \geq 1$, and $\gcd(p,q) = 1$.
\end{lemma}

\begin{proof}  As described in Section~\ref{tb-torus-links}, for a link $L = (\leg_1, \dots, \leg_n)$, 
$$tb(L) = tb(\Lambda_1) + tb(\Lambda_2) + \dots + tb(\Lambda_n) + 2 \sum_{i< j} lk(\Lambda_i, \Lambda_j).$$
As the linking number contribution is a topological invariant, we see that $tb(L)$   is maximized when
 $tb(\Lambda_1) + tb(\Lambda_2) + \dots + tb(\Lambda_n)$ is maximized. Since each component of a $K_{(np,nq)}$ cable link
 is a $K_{(p,q)}$ cable knot, it is important to understand the max $tb$ values that can be obtained for cable knots and how those values compare
 to the Remark~\ref{std-cable-tb-calc} calculations of the $tb$ values  in the components of our standard Legendrian cable.  

In fact, the maximum value of $tb$ is known for  cable knots of uniformly thick knot types,   \cite{EtnyreHonda05}.  
When $p = 1$,  the $q/p$ cable is topologically $K$. It follows that when $q/p\in \mathbb Z$ we have 
 $$\overline{tb}(K_{p,q}) = \overline{tb}(K).$$
The max $tb$ of a cable knot when $q/p \notin \mathbb Z$ and $K$ is uniformly thick  was established  in  \cite[Theorem~3.2 and~3.6]{EtnyreHonda05}:
\begin{enumerate}
\item If $q/ p  \notin \mathbb Z$ and $q/ p > \overline{tb}(K)$, $\overline{tb}(K_{p, q}) =  pq - |p \cdot  \overline{tb}(K) - q|;$
\item If $q/p \notin \mathbb Z$ and $q/ p < \overline{tb}(K)$, $\overline{tb}(K_{p, q}) = pq.$
\end{enumerate}
Observe that when $q/p > \overline{tb}(K_{(p, q)})$, the integral and nonintegral formulas for $\overline{tb}(K_{(p, q)})$ agree: when $p = 1$,
 $$pq - |p \cdot  \overline{tb}(K) - q| = q - |\overline{tb}(K) - q| = q - (q - \overline{tb}(K)) = \overline{tb}(K).$$
 Thus the formula given in (1) applies for all slopes $q/p > \overline{tb}(K)$.
 However, if $q/p < \overline{tb}(K)$, then  the formula  for $\overline{tb}(K_{p, q}) $ is more restrictive in the nonintegral case: when $p = 1$, 
 $$pq = q < \overline{tb}(K).$$
   
In the construction of the standard Legendrian cables, we see
 that if either  $q/p \geq \overline{tb}(K)$ or $q/p < \overline{tb}(K)$ and $q/p$ is not an integer, then
each component of the standard Legendrian $(np,nq)$-cable has $tb$ equal to the maximum possible $tb$ for a $(p,q)$-cable of $K$; see Remark~\ref{std-cable-tb-calc}.
 It remains to show that
when $q/p < \overline{tb}(K)$ and $q/p$ is an integer, then  $\sum tb(\leg_{i})$ is bounded above by $nq$ rather than by the larger quantity $n\, \overline{tb}(K)$.

{By Equation~(\ref{tb-tw}), the inequality}
\begin{equation} \label{lower-tb-sum} 
tb(\Lambda_1)+\ldots + tb(\Lambda_n)\leq nq.
\end{equation}
is the same as saying that the sum of the contact twisting along the $\Lambda_i$, relative to a convex torus $T$ they sit on as $(1,q)$-curves, is less than or equal to $0$. 
For a contradiction, suppose $L$ violates Inequality~\eqref{lower-tb-sum}, and so the sum of the twistings is positive.   
By the uniform thickness of $K$ we can assume that $L$ is contained in a solid torus $S$ that is a standard neighborhood of a Legendrian representative of $K$. So the dividing curves on $\partial S$ have slope $\overline{tb}(K)$ and $L$ will be a collection of $(1,q)$-curves inside of $S$. The torus is contactomorphic to a solid torus with dividing slope $-1$ (given by cutting along a meridional disk and re-glueing after $-\overline{tb}(K)-1$ full twists) and under this contactomorphism $L$ will be sent to a collection of Legendrian $(1,q-\overline{tb}(K)-1)$-curves. Now this solid torus can  be identified with a neighborhood of the maximal \tb invariant unknot $U$ in $S^3$. When this is done $L$ will be a $(1,k)$-torus link for some $k<-1$ with total twisting relative to the torus it sits on greater than $0$, a contradiction to Proposition~\ref{maxtb}. Inequality~(\ref{lower-tb-sum}) follows.
\end{proof}

In parallel to what was seen for negative torus links with unknotted components, for the  { integral} and lesser-slope cables, there will be additional non-standard Legendrian representatives of  $K_{(np,nq)}$ formed by twisted $n$-copies, as defined in  Definition~\ref{t-twist}. When $n = 2$, these 
twisted versions will have max $tb$; for larger $n$, these will not have max $tb$ yet will not destabilize to one with max $tb$.

\begin{lemma}\label{max-tb-integral}  Suppose $p = 1$ , $q \in \mathbb Z$, and $q/p \leq \overline{tb}(K)$. 
Let $k$ be the number of lattice points in the Legendrian mountain range of $K$ on or above the line $tb= q$. Then consider the following set of  Legendrian representatives of $K_{(n,nq)}$
consisting of  $n$-copies  and $t$-twisted $n$-copies of Legendrian representatives $\leg$ and $\leg_{t}$ of $K$:
$$\mathcal A = \{ n\leg : tb(\leg) = q \} \cup \left\{ T^{t}(n\leg_{t}) : tb(\leg_{t}) = q + t, t > 0 \right\}.$$
Then all $k$ elements of $\mathcal A$ are non-destabilizable, and the Legendrian twist versions will have max $tb$ if and only if $n = 2$.
\end{lemma}

\begin{proof}
The argument parallels the proofs of Lemma~\ref{nondestab} and~\ref{destab}.  Here we will use the uniform thickness property of $K$: the standard neighborhood with two dividing curves of slope $\overline{tb}(K)$ replaces the role played by the Heegaard torus for the torus knots.  
\end{proof}

\subsection{Unordered Classification of Legendrian Cables} \label{cables-unordered}

Now that we understand all of the non-destabilizable Legendrian $(np,nq)$-cables, we can state the main unordered classification result for 
 cable links of knot types $K$ that are uniformly thick and Legendrian simple.  In this statement, the standard Legendrian cables of $K_{(np,nq)}$ are  defined in Definition~\ref{cable-defn},
 and the $t$-twisted $n$-copy of $\leg$ is defined in Definition~\ref{t-twist}.

\begin{theorem} \emph{(Unordered Cable Link Classification)}\label{unorderedcable}
Let $K$ be a uniformly thick and Legendrian simple knot type. 
Consider two oriented Legendrian links $L$ and $L'$ that are topologically
equivalent to $K_{(np, nq)}$,
where $n \geq 2$, $p \geq 1$, and $\gcd(p,q) = 1$. 
 If we can write $L = \amalg_{i=1}^{n} \Lambda_{i}$, $L' = \amalg_{i=1}^{n} \Lambda_{i}'$ such that $tb(\Lambda_i) = tb(\Lambda'_i)$ and 
$r(\Lambda_i) = r(\Lambda'_i)$, $i = 1, \dots, n$, then there exists a contact isotopy
taking $L$ to $L'$ (but not necessarily $\Lambda_i$ to $\Lambda'_i$). Moreover, the precise
range of the classical invariants is given as follows:
\begin{description}[align=left]
\item[greater-slope cables] Suppose ${q}/p>\overline{tb}(K)$.  For each max $tb$  Legendrian representative $\leg$ of $K$, 
there exists a unique max-$tb$  Legendrian representative of $K_{(np,nq)}$ given by $\leg_{(np,nq)}$, the standard Legendrian $(np,nq)$-cable of $K_{(np,nq)}$. Each component  of this representative associated to $\leg$ will  satisfy
   $$tb = pq- |p \cdot \maxtb(K) -  q     | = pq  - q +  p\cdot \maxtb(K), 
\quad \text{ and }\quad  r = p\, r(\leg).  $$ 
Any non-maximal \tb invariant representative of $K_{(np, \pm nq)}$ can be destabilized to one in this set of max $tb$ representatives, $\{ \leg_{(np,nq)} : tb(\leg) = \maxtb(K) \}$.

\item[nonintegral and lesser-slope cables]
Suppose ${q}/p<\overline{tb}(K)$ and ${q}/p \notin \mathbb Z$.   
For every Legendrian representative $\Lambda$ of $K$ with $tb(\Lambda) = \lceil {q}/p \rceil$, there exist two Legendrian representatives of  $K_{(np,nq)}$ with max $tb$ given by
$\leg^{\pm}_{(np,nq)}$, the standard Legendrian $(np,nq)$-cables of $K_{(np,nq)}$. 
Each component of
$\Lambda^{\pm}_{(np, \pm nq)}$  will have equal $tb$ and $r$ values given by 
$$tb = pq, \qquad r =   pr(\leg) \pm (p \lceil q/p \rceil - q).$$
 Any non-maximal $tb$  Legendrian representative of $K_{(np, \pm nq)}$ can be destabilized to one  in the set $\{ \leg^{\pm}_{(np,nq)} : tb(\leg) = \lceil q/p \rceil\}$.

 \item[$\maxtb(K)$- or integral and lesser-slope cables]
Suppose $p = 1$ and  $q/p  \leq \overline{tb}(K)$.  Let $k$ be the number of lattice points in the Legendrian mountain range of $K$ on or above above the line $tb= q$. 
Then there is a set of $ k$ non-destabilizable Legendrian realizations of $K_{(n,nq)}$
  consisting of:
  \begin{itemize}
  \item $n\leg$, the $n$-copy of a Legendrian representative $\leg$ of $K$ with $tb(\leg) = q$, and 
  \item for every Legendrian representative $\leg_{t}$ of $K$ with $tb(\leg_{t}) = q + t$ with $t >0$,   the  Legendrian $t$-twisted $n$-copy
  of $\leg_{t}$, $T^{t}(n\leg_{t})$.
  \end{itemize}
  The $n$-copy will have max $tb$ while
  the Legendrian twist version will have max $tb$ if and only if $n = 2$. Any other
  Legendrian representative of $K_{(np,nq)}$ will destabilize to one in this non-destabilizable set, $\{ n\leg : tb(\leg) = q \} \cup \{ T^{t}(n\leg_{t}) : tb(\leg_{t}) = q + t, t > 0 \}.$
\end{description}
\end{theorem}

The claimed max $tb$ representatives of $K_{(np,nq)}$ was established in Lemmas~\ref{std-cables-max-tb} and \ref{max-tb-integral}.  
So the unordered classification of Legendrian cables of a simple and uniformly thick knot type given by Theorem~\ref{unorderedcable} will follow easily from the next two propositions.

\begin{proposition}  \label{destab-non-integer} Suppose $K$ is a uniformly thick knot type, $q/p \notin \mathbb Z$, and $n \geq 2$. Then $K_{(np,nq)}$ is Legendrian simple and every Legendrian representative of $K_{(np,nq)}$  will destabilize to a standard Legendrian $(np,nq)$-cable of $K$.
\end{proposition}

\begin{proof} 
This proof closely follows the proof of the unordered classification of torus links and so we only sketch the proof. 

Let $K$ be a uniformly thick, Legendrian simple knot type. Assume that $q/p$ is not an integer. We begin by noticing that, according to \cite[Section~3]{EtnyreHonda05}, if $L$ is any Legendrian link in the link type of $K_{(np,nq)}$ then it can be put on a convex torus $T$ bounding a solid torus realizing the knot type $K$. This is because each of the components of $L$ has non-positive contact twisting with respect to $T$.  

If $q/p>\overline{tb}(K)$, then as in the proof of Lemma~\ref{posdestab} we can destabilize the components of $L$ until they become ruling curves on a convex torus isotopic to $T$ with two dividing curves of slope $\overline{tb}(K)$. The result in this case follows as in the proof of Lemma~\ref{uniquemaxpos} coupled with the argument in \cite[Section~3]{EtnyreHonda05} that maximal $tb$ invariant $(p,q)$-cables are distinguished by their rotation numbers and become Legendrian isotopic after stabilization as soon as their invariants become the same. 

If $q/p<\overline{tb}(K)$ is not an integer, then as in the proof of Lemma~\ref{negtorusstab} the components of $L$ can be destabilized to become Legendrian divides on a convex torus isotopic to $T$. Then as in the proof of Lemma~\ref{negmaxtb} such a destabilized $L$ has maximal \tb invariant and is an $n$-copy of a maximal \tb $(p,q)$-cable of $K$. 
Finally, the rest of the classification follows as in the proof of Lemma~\ref{negstabtoiso}.
\end{proof}

\begin{proposition}\label{destab-integer} \label{destab-integer} Suppose $K$ is a uniformly thick knot type, $p = 1$, and $q/p  \leq \overline{tb}(K)$. 
Let $\mathcal A$ be the set of  $k$ non-destabilizable
Legendrian representatives of $K_{(n,nq)}$ as described in Lemma~\ref{max-tb-integral}. 
  Any other
  Legendrian representative of $K_{(n,nq)}$ will destabilize to one in this non-destabilizable set; moreover, $K_{(n,nq)}$ is Legendrian simple.
\end{proposition}

\begin{proof} 
The proof is completely analogous to the proofs of Lemma~\ref{destab} and~\ref{integer-simple}. So we only sketch the ideas. 

First number the components of the link $L$ so that $tb(K_1) \geq tb(K_2) \geq \dots \geq tb(K_n)$. If $tb(K_1)=q$ then we can put the link on a convex torus $T$ as $n$ curves of slope $q$. Since the twisting of $K_1$ is $0$ we see that it will be parallel and disjoint from the dividing curves. We can destabilize the other $K_i$ until they are also dividing curves. Now using uniform thickness we can assume $T$ is contained in a standard neighborhood $N$ of a maximal Thurston-Bennequin representative of $K$. Lemma~\ref{prelagmodel2} says that there is a pre-Lagrangian torus in $N$ that contains the destabilized $L$ as a union of its leaves. 

Now if $tb(K_1)<q$ then we can again put $L$ on a convex torus $T$. If the slope of the dividing curves on $T$ is equal to $-q$ then we can proceed as above. If not, we can destabilize the components of $L$ until they become ruling curves on $T$. Now since $K$ is uniformly thick we can find a solid torus $S$ in the knot type $K$ that contains $T$ and has dividing slope $q$ and at least $n$ dividing curves. We can then use convex annuli with one boundary component on a component of $L$ and the other a dividing curves on $\partial S$ to destabilize $L$ further to be dividing curves on $\partial S$. So we are again finished as above. 

Finally if $tb(K_1)=q+t$ for some $0<t< \overline{tb}(K)+1-q$, then by the bound on the Thurston-Bennequin invariant in Lemma~\ref{std-cables-max-tb} we must have all the other components of $L$ have $tb\leq q-t$ (since taking any component together with $K_1$ we will have to have the sum of the \tb bounded above by $2q$). Thus we can put $K_2\cup \ldots \cup K_n$ on a convex torus around $K_1$. Now as in the proof of Lemma~\ref{destab} we can destabilize all the $K_i$ for $i>1$ until they are ruling curves on a standard neighborhood of $K_1$.  

The proof is complete by showing $K_{(n,nq)}$ is Legendrian simple. This follows exactly as in the proof of Lemma~\ref{integer-simple}.
\end{proof}

\begin{proof}[Proof of Theorem~\ref{unorderedcable}] Suppose $K$ is a uniformly thick and Legendrian simple knot type.  
Lemmas~\ref{std-cables-max-tb} and \ref{max-tb-integral} established the non-destabilizable  representatives of $K_{(np,nq)}$, and Propositions~\ref{destab-non-integer} and
\ref{destab-integer} show that every Legendrian representative of $K_{(np,nq)}$ will destabilize to one of these and are determined by their classical invariants.   The values of $tb$ and $r$ that can be obtained for those
in the non-destabilizable sets was established in Remark~\ref{std-cable-tb-calc}. 
\end{proof}

\subsection{Symmetries of Legendrian Cables} \label{standard-perms} 
We now move on to study the ordered classification of  Legendrian representatives of the cable link $K_{(np,nq)}$.  All the rigidity will appear in the max $tb$ representatives
formed as  standard Legendrian cables.  We now recall our main result from the Introduction.
\smallskip

\noindent
{\bf Theorem~\ref{cableperms}}.
{\em Let $K$ be a uniformly thick knot type. If $L=(\Lambda_1,\ldots, \Lambda_n)$ is a  {standard }Legendrian $(np,nq)$-cable of $K$, where the $\Lambda_i$ are ordered as they appear on the torus or annulus  
 used in the definition of the standard Legendrian cables, then  the following permutations of the components are possible via a Legendrian isotopy.
\begin{description}[align=left]
\item[greater-slope cables] If $q/p>\overline{tb}(K)$, then  any permutation of the $\Lambda_i$ is possible. 
\item[$\overline{tb}(K)$-slope cables]{ If  $q/p=\overline{tb}(K)$ and $K$ is not a cable knot or $K$ is an $(r,s)$-cable and $q/p\not=rs$, then no permutation of the $\Lambda_i$ can be realized by a Legendrian isotopy. }
\item[integral and nonintegral lesser-slope cables ] If $q/p<\overline{tb}(K)$ { and  $K$ is not a cable knot or $K$ is an $(r,s)$-cable and $q/p\not=rs$,
  then only cyclic permutations  of the $\Lambda_i$ can be realized.}
\end{description}
}

This theorem will be proven in Section~\ref{legcables}.  The proof that arbitrary permutations are possible for the greater-slope  standard Legendrian cables will 
parallel the proof that one can  
arbitrarily permute the components in a max-$tb$ Legendrian  positive torus link.    The  strategy to forbid arbitrary permutations of $\overline{tb}(K)$- and lesser-slope standard
Legendrian cables will
parallel the strategy used to forbid arbitrary permutations of the components in a max $tb$ Legendrian  negative torus link. This time, instead of initially restricting to a basic slice, we will
first work in a solid torus with convex boundary having two dividing curves and show that it is not possible to perform any permutation of  the leaves of a pre-Lagrangian annulus,  and it is only possible to do cyclic permutations of the leaves of a 
pre-Lagrangian torus.  We then show that the existence of permutations of the components of a  $\overline{tb}(K)$-slope (or lesser-slope) standard Legendrian cable implies the existence of permutations of components  
of the pre-Lagrangian annulus (or pre-Lagrangian torus) of the solid torus.  After establishing Theorem~\ref{cableperms}, we will easily be able to give the ordered classification of
Legendrian cables.

\subsubsection{Links in solid tori}\label{list}

In Theorem~\ref{thickenedtori}, we showed that it is not possible to do a non-cyclic permutation of the leaves of a boundary-parallel, pre-Lagrangian torus in a basic slice.  Now we will
show that when a solid torus has a    convex boundary with two longitudinal curves, it is not possible to do {\it any} permutations of the components of  a link $L$ that is formed from 
leaves in a pre-Lagrangian annular slice of the solid torus. This will later be important for studying forbidden permutations in the $\overline{tb}(K)$-slope cables.

\begin{theorem}\label{torus0perms}  
Let $(S^1\times D^2, \xi)$ be a solid torus with  convex boundary having two dividing curves of slope $0$, meaning they are parallel to   $S^1\times \{p\}$, and ruling curves of slope $\infty$.  
Inside the solid torus, there is a pre-Lagrangian annulus $A=S^1\times \gamma$ for some properly embedded $\gamma\subset D^2$; let $L=(\Lambda_1,\ldots, \Lambda_n)$ be the link consisting of distinct leaves of $A$. Then no non-trivial permutation of the components of $L$  can be realized by a Legendrian isotopy. \end{theorem}

The proof of Theorem~\ref{torus0perms} is very similar to the proof of Theorem~\ref{thickenedtori}.

\begin{proof}  By Kanda's classification result, see Theorem~\ref{thm:solid-tori}, 
it suffices to show the desired statement in a standard model: we will build  an $S^1$-invariant model  for $(D^2\times S^1, \xi)$   such that
 the meridional disks $\{\theta\}\times D^2$ are convex and have Legendrian boundary. 

To build the model, let $D^2$ be the unit disk in $\R^2$. We may find a map of $D^2$ into $(\R^3,\xi_{std})$ so that $\partial D$ is sent to a Legendrian unknot with Thurston-Bennequin invariant $-1$. 
This will induce a characteristic foliation on $D^2$  that can be divided by a single arc $\gamma$. Since the foliation on $D^2$ is divided by 
$\gamma$, there is an $\R$-invariant contact structure on $\R\times D^2$ such that each $\{p\}\times D^2$ has the given foliation. Let $S^1\times D^2$ be the $S^{1}$-invariant contact manifold
formed by the quotient of this manifold by the $\Z$-action 
generated by $(t,p)\mapsto (t+1,p)$.  Since the boundaries of the meridional disks are Legendrian and the contact planes are 
tangent to $\partial (S^1\times D^2)$ at exactly two points on each meridional disk, we see that there are two lines of singularities with slope $0$ in the characteristic foliation of $\partial (S^1\times D^2)$.  
Thus  $\partial (S^1\times D^2)$ has two Legendrian divides of slope $0$, and the rest of the foliation consists of non-singular leaves of slope $\infty$, as desired.

In this model,  $S^1 \times \gamma$ is a pre-Lagrangian annulus, $A$.  Choose $n$ distinct points on $\gamma$, and let 
$L=(\Lambda_1,\ldots, \Lambda_n)$ be the corresponding distinct leaves on $A$.  Suppose $\phi_t$ 
 is a Legendrian isotopy that realizes a   non-trivial permutation of the components of $L$. We can extend this to an ambient contact isotopy $\psi_t$ of $S^1 \times D^2$
that is the identity near the boundary.

Now as in the proof of Theorem~\ref{thickenedtori}, Step~\ref{step 3}, we can choose $S^1$-invariant neighborhoods $N_i$ of the Legendrian knots $\Lambda_i$ such that $\partial N_i $ is convex and the contact isotopy $\phi_t$ induces a contact diffeomorphism from $(S^1\times D^2)\setminus \cup N_i$. Moreover, $(S^1\times D^2)\setminus \cup N_i$ is diffeomorphic to $S^1\times \Sigma$ where $\Sigma$ is $D^2$ with $n$ sub-disks removed from the interior. The boundary of $S^1\times \Sigma$ is convex with dividing curves of slope 0 and ruling curves of slope $\infty$, and the contact structure on $S^1\times \Sigma$  is $S^1$-invariant. For a fixed $\theta$, $\{\theta\}\times \Sigma$ and $\psi_1(\{\theta\}\times \Sigma)$ are convex surfaces, and, after doing ``boundary slides'' as in the proof of Theorem~\ref{thickenedtori}, Step~\ref{step 4}, we can assume that their boundaries are the same and that these surfaces
are isotopic relative their boundaries. The dividing curves on $\{\theta\}\times \Sigma$ are simply the intersection of $\gamma$ with $\Sigma$ and the dividing curves on $\psi_1(\{\theta\}\times \Sigma)$ are the image of these curves under the map $\psi_1$. Thus any non-trivial permutation gives a contradiction to Proposition~\ref{KoS1invariant}, since the dividing curves connect the boundary components differently.
\end{proof}

Next we study allowed permutations of the leaves of a pre-Lagrangian torus inside a solid torus.  This will be used later to study permutations of the lesser-slope cables.

\begin{theorem}\label{torusnegperms} 
Let $(S^1\times D^2, \xi)$ be a solid torus with  a convex boundary having two dividing curves of slope 0,
meaning they are parallel to $S^1\times \{p\}$. For any slope $q/p <0$, there is a 
boundary-parallel, pre-Lagrangian torus $T$  with  characteristic foliation having slope $q/p$.   Let $L=(\Lambda_1,\ldots, \Lambda_n)$ be $n$ distinct leaves on $T$ labeled cyclically as they
appear  along $T$. Then via a Legendrian isotopy, cyclic permutations of the components of  $L$ are possible  but  non-cyclic permutations of the components of $L$ cannot be attained. 
\end{theorem}

\begin{proof} By moving along the leaves of $T$, it is clear that one can do cyclic permutations of $L$.

Now assume, for a contradiction, that it is possible to do a non-cyclic permutation of the leaves of $L$.  As argued in Corollary~\ref{fix1},  
 we can assume  there
is a $3$-component link $L=(\Lambda_1, \Lambda_2, \Lambda_3)$ consisting of leaves of the characteristic foliation of $T$ and a contact isotopy $\psi_t$ of $(S^1\times D^2, \xi)$ such that   
 $\psi_{t} = id$ in a neighborhood of the boundary, $\psi_1(\Lambda_1) = \Lambda_1$, $\psi(\Lambda_2) = \Lambda_3$, and $\psi_1(\Lambda_3) = \Lambda_2$.  
We will now argue that this non-cyclic isotopy implies the existence of a non-cyclic isotopy among the leaves of a pre-Lagrangian torus in a basic slice, thus contradicting Theorem~\ref{thickenedtori}. The argument is very similar to that in the proof of Proposition~\ref{S3-basic-slice}, so we only sketch it here. 

First assume that $q/p \notin \mathbb Z$.  
As in the proof of Lemma~\ref{setup}, we know $S^1\times D^2$ has the structure of a Seifert fiber space over a disk $\mathcal D \subset S^2$ with 
one singular fiber and the regular fibers being $q/p$ curves. 
 We can assume that the pre-Lagrangian torus $T_0=T$ containing $L$ is the pre-image of a curve $c_0 \subset \mathcal D$ that bounds a disk containing the singular point. Let $T_1=\psi_1(T_0)$; 
 we can isotop $T_1$ relative to $L$ to be a convex torus $T_1'$ that is the pre-image of a curve $c_1$ that bounds a disk containing the singular point and intersects $c_0$ transversely.  
 We can choose arcs $a_0$ and $a_1$ that are disjoint from $c_1$, the first of which connects the singular point to a point on $c_0$ and the second connecting a point on the  boundary of 
 $\mathcal D$ to a point on $c_0$, and are otherwise disjoint from $c_0$. The pre-image 
of these arcs are annuli $A_0$ and $A_1$.  Then, as argued in the proof of Lemma~\ref{claim2}, the torus $T_1'$ has a neighborhood $R_1=T_1'\times[-\epsilon,\epsilon]$, 
disjoint from the annuli $A_{0}, A_{1}$, such that $\xi$ restricted to $R_1$ is $[-\epsilon,\epsilon]$-invariant,
 each boundary component of $R_{1}$ has two dividing curves, and the cyclic ordering of $L$ induced by a complementary annulus in $R_{1}$ agrees with the one induced by the pre-Lagrangian torus $T_1$. 
 Lastly, as in the proof of Lemma~\ref{claim3}, we can further thicken $R_{1}$ to a basic slice $V$ that contains $T_0$ and $R_{1}$.  Thus, the existence of our non-cyclic permutation of our link $L$ implies
the existence of two pre-Lagrangian tori in a basic slice that induce different cyclic orderings on the leaves of $L$, a contradiction to Corollary~\ref{pre-lag-order}.

Now consider the case where $p = 1$ so $q/p = q \in \mathbb Z$:  now $S^1\times D^2$ will fiber over a disk $\mathcal D \subset S^{2}$ where all fibers are $q$-curves, and the center core curve of the solid torus corresponds to the center $\mathcal O \in \mathcal D$ is not a singular fiber. We can now again arrange $T_0$ to be the pre-image of a curve about $\mathcal O$. When trying to do the same for $T_1$, the initial curves whose pre-image is $T_1$ might not contain $\mathcal O$ (when $q/p\not\in \Z$ this is prevented for topological reasons, but since the center curve is a regular fiber now it is not), but we can further isotop $T_1$, relative to $L$, so that it does contain $\mathcal O$. Now we cannot argue as above to find a basic slice in the solid torus that contains $T_0\cup R_1$, but we can argue as in Lemma~\ref{claim3} to show that it is in the union of two basic slices that have the same sign that can be embedded  into a basic slice. 
\end{proof}

\subsubsection{Isotopies of the Standard Legendrian Cables}\label{legcables}
In this section, we will establish Theorem~\ref{cableperms}.  
We begin with the greater-slope cables of $K$,  that is  $q/p > \overline{tb}(K)$. 
\begin{lemma}\label{poscableperms}
Let  $K$ be a uniformly thick knot type.  Suppose $q/p>\overline{tb}(K)$ and $L=(\Lambda_1,\ldots, \Lambda_n)$ is a standard  Legendrian $(np,nq)$-cable of $K$.  Then any permutation of the $\Lambda_i$ can be achieved via a Legendrian isotopy. 
\end{lemma}

The proof of this lemma parallels the proof of Theorem~\ref{oposlink}, which states that components of a max-$tb$ Legendrian $(np, +nq)$-torus link can be arbitrarily permuted.
\begin{proof}
By definition, $L$ can be realized as the ruling curves on the boundary torus $T$ of a standard neighborhood of max-$tb$ Legendrian representative of $K$. Since $T$ is convex, there is a neighborhood $T\times I$ on which the contact structure is $I$-invariant. Thus, for all $a \in I$,   $T\times \{a\}$ is foliated by ruling curves of slope $q/p$. By a Legendrian isotopy, we may move   the components of $L$ to sit on different tori, then by another Legendrian isotopy we can independently move the components through ruling curves on these different tori, and lastly move the components back through a Legendrian
isotopy to the original torus.  In this way, we can achieve any permutation of the components of $L$ by a Legendrian isotopy. 
\end{proof}

To forbid non-cyclic permutations in the lesser-slope cables and any permutations in $\overline{tb}(K)$-slope cables of non-cable knots $K$, we will use the 
 following two lemmas that allow us to ``localize" isotopies into solid tori.
 
\begin{lemma}\label{negcableperms}
Let $K$ be a uniformly thick knot type, and suppose $q/p<\overline{tb}(K)$. We assume that $K$ is not a cable or if $K$ is a cable, $K = K'_{(r,s)}$, we additionally suppose that $q/p\not= rs$.   Let  $L=(\Lambda_1,\ldots, \Lambda_n)$  be a  standard  Legendrian $(np,nq)$-cable of $K$, with
the components ordered cyclically as they appear in the pre-Lagrangian torus used in the definition of the standard Legendrian cable. 
 If there is a non-cyclic permutation of the components of $L$, then there is also a non-cyclic permutation of the leaves of a boundary-parallel, pre-Lagrangian torus in a solid torus with  convex boundary having two Legendrian  divides. 
\end{lemma}

\begin{lemma}\label{0cableperms}
Let $K$ be a uniformly thick knot type, and suppose $q/p=\overline{tb}(K)$. We assume that $K$ is not a cable or if $K$ is a cable, $K = K'_{(r,s)}$, we additionally suppose that $q/p\not= rs$.  
 Let  $L=(\Lambda_1,\ldots, \Lambda_n)$ be a standard Legendrian$(np,nq)$-cable of $K$. 
If there is a non-trivial permutation of the components of $L$, then there is also a non-trivial permutation of the leaves of a properly embedded, pre-Lagrangian annulus in a solid torus with a tight contact
structure and a convex boundary having two Legendrian divides. 
\end{lemma}
Our main results about isotopies of the standard Legendrian cable links, Theorem~\ref{cableperms}, now easily follows.
\begin{proof}[Proof of Theorem~\ref{cableperms}]
The fact that all permutations of the components of the greater-slope standard Legendrian $(np,nq)$-cable link of $K$ are possible is the content of Lemma~\ref{poscableperms}. The statements about the
$\overline{tb}(K)$-slope and lesser-slope  standard Legendrian cables now follow from Lemma~\ref{negcableperms} coupled with Theorem~\ref{torusnegperms}, and Lemma~\ref{0cableperms} coupled with Theorem~\ref{torus0perms}. 
 \end{proof}
 
 Now we move on to proving our localization results for cables of uniformly thick, non-cable knot types.  We will first highlight a useful property we have for knots with essential annuli in their complements. This result seems to be well-known, see \cite[Lemma~15.26]{BurdeZieschang03}.

 \begin{theorem} \label{cable-prop} If $K \subset S^{3}$  and $A$ is an annulus in $X_{K}: = S^{3} \setminus N(K)$ with boundary on $\partial X_{K}$ that is not boundary-parallel, then either
 \begin{enumerate}
 \item $K = K_{1} \# K_{2}$ and $\partial A$  has slope $\infty$  on $\partial X_{K}$, or
 \item $K = K'_{(r,s)}$  and $\partial A$ has slope $rs$ on $\partial X_{K}$.
 \end{enumerate}
 \end{theorem}
 
{ In the following proof, we will apply Theorem~\ref{cable-prop} to deduce that an annulus is boundary parallel. }
 
\begin{proof}[Proof of Lemma~\ref{negcableperms}]
Assume, for a contradiction, that it is possible to do a non-cyclic permutation of the leaves of $L$.  As argued in Corollary~\ref{fix1},  
 we can assume  there
is a $3$-component link $L=(\Lambda_1, \Lambda_2, \Lambda_3)$ consisting of leaves of the characteristic foliation of $T$ and a contact isotopy $\psi_t$ of $S^{3}$ such that   
  $\psi_1(\Lambda_1) = \Lambda_1$, $\psi(\Lambda_2) = \Lambda_3$, and $\psi_1(\Lambda_3) = \Lambda_2$.

By construction, there is a pre-Lagrangian torus $T_0$ that contains $L$ as a subset of the leaves in its characteristic foliation. Let $T_1=\psi_1(T_0)$. Let $S_0$ be the solid torus with core in the knot type $K$ that $T_0$ bounds. 
\smallskip

\noindent
{\bf Claim:} We can perturb $T_1$, relative to $L$, to a convex torus $T_1'$ satisfying 
\begin{enumerate}
\item  $T_1'$ is transverse to $T_0$, 
\item the union $T_0 \cup T_1'$ is contained in a solid torus $S$ with core in the knot type $K$ and standard convex boundary with dividing slope $\overline{tb}(K)$, and
\item  $T_1'$ is contained in a non-rotative thickened torus $R_{1}$ with convex bound each having two dividing curves of slope $q/p$ and the ordering on the components of $L$ induced from a complementary annulus agrees with the ordering from $T_1$. 
\end{enumerate}
\begin{proof}[Proof of Claim]
There is a $C^\infty$-perturbation of  $T_{1}$ relative  to $L$ to a convex torus and a further $C^\infty$-perturbation, relative to $L$, that makes it transverse to $T_0$. Denote the resulting torus by $T_1'$.
Claim (1) has now been established.

To verify Claim (2), since $K$ is a uniformly thick knot type, it suffices to show that both $T_{0}$ and $T_{1}'$ are contained in a solid torus that has a core curve in the knot type of $K$; this is
where we will use the fact that $K$ is a non-cable knot type {or that $K$ is a cable knot type, $K = K'_{(r,s)}$, and $q/p \neq rs$.}
The intersection of $T_1'$ and $T_0$ consists of simple closed curves. The null-homologous curves can be eliminated by an isotopy of $T_1'$ using a standard innermost disk argument.
So we are left with $T_{1}' \cap T_{0}$ consisting of curves that are parallel to $L$, and thus have slope $q/p$.
The curves in $T_0\cap T_1'$ cut $T_1'$ into several annuli. The annuli that lie on the interior of $T_{0}$ do not obstruct the existence of the solid torus $S$ that engulfs $T_{0}$ and $T_{1}'$. 
If $A$ is an annulus  on the outside of the solid torus $S_0$,
 then $\partial A$ are not meridians (of slope $\infty$), and by hypothesis  $K$ is not a cable or  $K$ is a cable, $K = K'_{(r,s)}$, such that $q/p \neq rs$.
So by Theorem~\ref{cable-prop} we see that  
 the $A$ must be boundary-parallel.  In particular, each such exterior annulus $A$ together with an annulus on $T_{0}$ cobound a  solid torus. Thus the union of $S_0$ and the $A$ is contained in a solid torus whose core is in the knot type $K$.  By the uniform thickness of $K$ this torus can be further enlarged to a torus $S$ that is a standard neighborhood of a maximum \tb invariant representative of $K$, as stated in Claim (2).

For Claim (3), we can find $R_{1}$ using a state transition argument as we did in the proof of Lemma~\ref{claim2}.
\end{proof}
To complete the proof of the lemma we now notice that Lemma~\ref{prelagmodel2} says that inside $S$ there is a pre-Lagrangian torus $P_1$ that intersects $T_1'$ in its Legendrian divides (in particular the intersection contains $L$) and the ordering on $L$ coming from $R_{1}$ and $P_{1}$ is the same. Of course the ordering of the components of $L$ coming from $T_0$ is different by hypothesis. Thus as in Step~1 of the proof of Theorem~\ref{thickenedtori} we see that there is an isotopy of $L$ in $S$ that permutes the components non-cyclically, as desired.
\end{proof}
\begin{proof}[Proof of Lemma~\ref{0cableperms}]
By definition, the  $(np,nq)$-cable of $K$ is formed by taking a standard neighborhood of $N$ of a maximal \tb invariant representative of $K$ and then taking $n$ leaves of the foliation of the pre-Lagrangian annulus $A$ in $N$. If there were a non-trivial permutation of the components of the  $(np,nq)$-cable of $K$, then two of the components of the cable on $A$ interchange their order; by sliding through the leaves of $A$ we see that we have interchanged two leaves of characteristic foliation of $A$. So if we show that this cannot happen, then there are no permutations of the cable of $K$. 
So assume, for a contradiction, that there is a non-trivial permutation of the $(2p,2q)$-cable of $K$. More specifically, there is a Legendrian knot $\Lambda$ in the knot type $K$ such that $L$ is a union of two leaves $\leg_1$ and $\leg_2$ of the characteristic foliation on a pre-Lagrangian annulus $A$ in a standard neighborhood $N$ of $\Lambda$ and
there is a Legendrian isotopy exchanging $\leg_1$ and $\leg_2$.  This Legendrian isotopy 
 can be  extended it to an ambient contact isotopy and then further extended, using Lemma~\ref{close}, so that $A$ and the image of $A$ under the isotopy, denoted by $A'$,  agree in a neighborhood of $A_i$ of $\leg_i$. Because an orientation of the contact structure puts a co-orientation on the foliation of $A$ and $A'$, that must agree where the annuli agree, 
we see the intersection of $A$ and $A'$, near $L$, must be as shown in Figure~\ref{AandAprime}. 
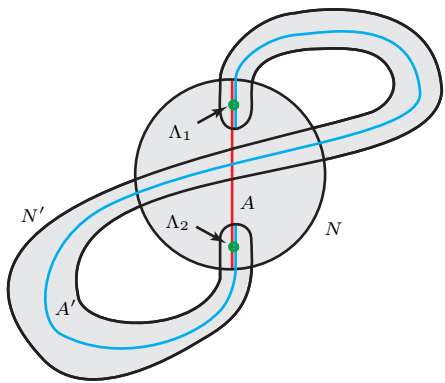
\begin{figure}[ht]
\begin{overpic} 
{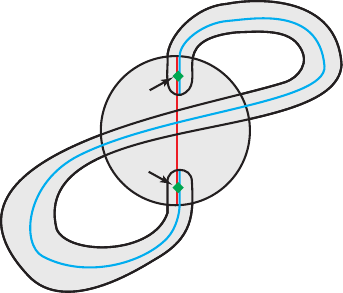}
\tiny
\put(88, 65){$A$}
\put(18,25){$A'$}
\put(61, 92){$\leg_1$}
\put(60, 58){$\leg_2$}
\put(120, 55){$N$}
\put(5, 60){$N'$}
\end{overpic}
        \caption{This is a cross section of $N\cup N'$. We see that annulus $A$ in red and the annulus $A'$ in blue. They agree near $\leg_1\cup \leg_2$; in the figure they have been slightly offset so that they can both be seen. 
        The grey regions are $N$ and $N'$. The sets $A'$ and $N'$ could be more complicated, but $A$, $N$ and $A\cap A'$ are as shown in the figure. }                    
                \label{AandAprime}
\end{figure}
Let $N'$ be the image of $N$ under this isotopy, and denote $\partial N$ and $\partial N'$ by $T$ and $T'$, respectively. By shrinking $N$ and $N'$ we may assume $A_i$ contains the dividing curves of $T$ and $T'$.
\smallskip

\noindent
{\bf Claim:} There is a contact isotopy, fixing the annuli $A_i$, that takes $A'$ to an annulus contained in a solid torus $B$ such that either  { $\partial B \subset N$} (but $B\not\subset N$), or $N\cup B \subset S$,
 where $S$ is a solid torus  in the knot type of $K$.

To finish the proof of Lemma~\ref{0cableperms}, suppose $\partial B \subset N$.  
Then $N\cup B=S^3$ (since $B$ is not contained in $N$)
so $T$ is a Heegaard torus for $S^3$, and thus $K$ is the unknot. Since we are assuming $K$ is uniformly thick and hence not an unknot, this cannot be the case. So we are left to deal with the case when $N\cup B$ is contained inside a solid torus $S$ in the knot type of $K$. By the uniform thickness of $K$, we can assume that $S$ is a standard neighborhood of a maximal Thurston-Bennequin invariant representative of $K$. But now $S$ contains $A\cup A'$ and we can use these to guide an isotopy exchanging $\leg_1$ and $\leg_2$ inside of $S$, giving us the desired conclusion of the lemma. 
\end{proof}

\begin{proof}[Proof of Claim]
We can $C^\infty$ perturb $T'$, fixing the $A_i$, so that $T'$ is transverse to $T$. 
We will first show that there is a topological isotopy of $T'$, fixing the annuli $A_i$, such that $T'$ is contained in $N$ or $T'$ intersects $T$ in a union of simple closed curves parallel to the dividing curves on $T$; we will then apply a Discretization of Isotopy argument {and, if needed, an ``engulfing'' argument to construct the desired $S$}.
 
Observe that
$T\cap T'$ consists of a disjoint union of simple closed curves where each closed curve is either null-homologous or parallel to the dividing curves on $T$. A standard innermost disk argument can be made to isotop $T'$ so as to remove null-homologous closed curves. In the case that are no additional curves of intersection, then we have shown that there is a topological isotopy of $N'$ such that $\partial N' \subset N$.  Otherwise, all the remaining circles of intersections between $T$ and $T'$ are homologically essential, and since they must be disjoint from $\partial A_i$ they must be parallel to $\partial A_i$. Since $A_i\cap T$ are the dividing curves of $T$, we see that the curves of intersection are parallel to the dividing curves of $T$, which have slope $q/p=q$.

Now using the Discretization of Isotopy technique, \cite[Section~2.2.3]{Honda02}, we can find a sequence of convex tori $T_0$,\ldots, $T_l$, such that $T_0$ is the original $T'$, $T_l$ is the result of $T'$ under the above isotopy, each $T_i$ intersects $A_i$ in a leaf of $A_i$, and each pair $T_i\cup T_{i+1}$ cobound a thickened torus that is the result of a bypass attachment. 
{We inductively claim that for each $i$ there exists a solid torus $B_{i}$ such that $\partial B_{i} = T_{i}$ and there exists a boundary parallel torus $T_{i}' \subset B_{i}$ that is contact isotopic to $T'$.
For the base case of $i=0$ , we can take $B_{0} = N'$.  
Now inductively assume this is true for $T_i$, and we will verify that it is true for $T_{i+1}$. 

{To see this if $T_{i+1}$ is 
not contained in  $B_i$, then let $B_{i+1}$ be the torus that $T_{i+1}$ bounds, and $T'_{i+1}=T'_{i}$ (clearly $T'_{i+1}=T'_i\subset B_i \subset B_{i+1}$) is the desired torus. If $T_{i+1}$ is contained in $B_i$ then let $B_{i+1}$ be the torus that it bounds. Since $B_{i+1}$ is a solid torus with longitudinal dividing curves on it, we know that it is contactomorphic to a standard model, and thus we know inside $B_{i+1}$ there is a torus $T_{i+1}'$ having two dividing curves of slope $q/p=q$. Now since $K$ is uniformly thick $B_i$ (and hence $B_{i+1}$) is contained in a standard neighborhood $S'$ of a maximal \tb invariant representative of $K$. We know that since $T'_i$ and $T'_{i+1}$ both have two dividing curves of slope the same as $\partial S'$ that $T_i$ and $T_{i+1}$ are both contact isotopic to $\partial S'$ and hence to each other. That is $T'_{i+1}$ is isotopic to $T'$ as claimed. }}

So after contact isotopy we can assume that $N'$ (and $A'$) is contained in a solid torus $B$ such that $\partial B$ is contained in $N$ or intersects $T$ transversely and in curves parallel to the dividing set of $T$ which have slope $q/p=\overline{tb}(K)$. In the former case we are done and in the latter we have that $\partial B$ is divided into annuli by $\partial B\cap T$. If one of these annuli is outside of $N$ then since $K$ is not cable, or if $K$ is a $(r,s)$-cable then $q/p\not = rs$, we know by Theorem~\ref{cable-prop} it must be parallel to $\partial N$. Thus there is a solid torus that contains $N$ and this annulus. Arguing similarly for all the annuli there is a solid torus $S$ that contains $N$ and $B$, and hence $N\cup N'$. 
\end{proof}
 
Now that the hard work has been done to prove Theorem~\ref{cableperms}, which determines the possible permutations  in the max $tb$ representatives, we can easily establish the ordered classification of cable links.

\begin{proof}[Proof of Theorem~\ref{orderedcable}]
With Theorem~\ref{cableperms} in hand, the proof is almost identical to that of the proof of the ordered classification of Legendrian torus knots, Theorem~\ref{orderedclass}. 
\end{proof}

\def\cprime{$'$}

%\bibliography{references}
%%\bibliographystyle{gtart}
%\bibliographystyle{plain}

\end{document}